\documentclass[12pt]{amsart}

\usepackage[hmargin=0.75in,height=9in]{geometry}
\usepackage{amssymb,amsthm, times}
\usepackage{delarray,verbatim}
\usepackage{ifpdf}
\ifpdf
\usepackage[pdftex]{graphicx}
 \else
\usepackage[dvips]{graphicx}
 \fi

\usepackage{bm}

\usepackage{ifpdf}
\usepackage{color}
\definecolor{webgreen}{rgb}{0,.5,0}
\definecolor{webbrown}{rgb}{.8,0,0}
\definecolor{emphcolor}{rgb}{0.95,0.95,0.95}

\usepackage{hyperref}
\hypersetup{%
          colorlinks=true,
          linkcolor=webbrown,
          filecolor=webbrown,
          citecolor=webgreen,
          breaklinks=true}
\ifpdf \hypersetup{pdftex,
            pdfstartview=FitH, %%Fit, FitB, FitH
            bookmarksopen=true,
            bookmarksnumbered=true
} \else \hypersetup{dvips} \fi

\newcommand {\B}{\mathcal{B}}

\numberwithin{equation}{section}

\newtheorem{theorem}{Theorem}[section]
\newtheorem{proposition}{Proposition}[section]

\newtheorem{remark}{Remark}[section]
\newtheorem{lemma}{Lemma}[section]
\newtheorem{example}{Example}[section]
\newtheorem{assump}{Assumption}[section]

\numberwithin{remark}{section} \numberwithin{proposition}{section}
\numberwithin{corollary}{section}
\newcommand {\R}{\mathbb{R}}

\newcommand {\p}{\mathbb{P}}

\newcommand {\E}{\mathbb{E}}

\newcommand{\diff}{{\rm d}}

\newcommand{\lev}{L\'{e}vy }

\newcommand{\EQ}{\rm{e}_q}

\title[Optimality of Refraction Strategies   for Spectrally Negative L\'{e}vy Processes]{Optimality of Refraction Strategies   for Spectrally Negative L\'{e}vy Processes}
\thanks{$*$\, Department of Probability and Statistics,
Centro de Investigaci\'on en Matem\'aticas, Apartado Postal 402, Guanajuato GTO 36000, Mexico.  Email: dher@cimat.mx, jluis.garmendia@cimat.mx}
\thanks{$\dag$\, Department of Mathematics,
Faculty of Engineering Science, Kansai University, 3-3-35 Yamate-cho, Suita-shi, Osaka 564-8680, Japan. Email: kyamazak@kansai-u.ac.jp.   }

\thanks{This version: \today.  J. L. P\'erez  is  supported  by  CONACYT,  project  no.\ 241195.
K. Yamazaki is in part supported by MEXT KAKENHI Grant Number  26800092, the Inamori foundation research grant, and the Kansai University Subsidy for Supporting young Scholars 2014.}

\author[D.\ Hern\'andez-Hern\'andez]{Daniel Hern\'andez-Hern\'andez$^*$}
\author[J. L. P\'erez]{Jos\'e-Luis P\'erez$^*$}
\author[K. Yamazaki]{Kazutoshi Yamazaki$^\dag$}
\date{}
\begin{document}

\begin{abstract}
We revisit a stochastic control problem of optimally modifying the underlying spectrally negative \lev process.  A strategy must be absolutely continuous with respect to the Lebesgue measure, and  the objective is to minimize the total costs of the running and controlling costs. Under the assumption that the running cost function is convex, we show the optimality of a refraction strategy.
We also obtain the convergence of the optimal refraction strategies and the value functions, as the control set is enlarged, to those in the relaxed case without the absolutely continuous assumption. Numerical results are further given to confirm these analytical results.

\noindent \small{\textbf{Key words:} stochastic control;
refracted \lev processes; scale functions }\\
\end{abstract}

\maketitle

\section{Introduction}
We consider a class of stochastic control problems under the restriction that strategies must be absolutely continuous with respect to the Lebesgue measure.  Given a spectrally negative \lev process, the controller aims to optimally decrease the process so as to
minimize the sum of the expected running and controlling costs.  The running cost is modeled as a convex function $h$ of the controlled process that is accumulated over time.  The controlling cost is proportional to the size of control.  Our objective is to show the optimality of the \emph{refraction strategy} so that the controlled process becomes a \emph{refracted \lev process} of  Kyprianou and Loeffen \cite{kyprianou2010refracted}, with a suitable choice of the refraction boundary level.

In the last decade, there have been great developments in the theory of \lev processes and its applications in stochastic control. A representative example is its contributions in de Finetti's  optimal dividend problem, where the classical compound Poisson model has been generalized to a spectrally negative \lev model.  The expected net present value (NPV) of total dividends under the \emph{reflection strategy}, so that the controlled process becomes a reflected \lev process, can be concisely written in terms of the scale function (see \cite{Avram_et_al_2007}).  While the optimality may fail depending on the choice of the \lev measure, Loeffen \cite{Loeffen_2008} obtains a sufficient condition for optimality and in particular shows that it holds when the \lev measure has a {\it completely monotone density}.  Other stochastic control problems that have been explicitly solved via the scale function include  the dual model of the optimal dividend problem \cite{Bayraktar_2012}, inventory control problems as in \cite{Baurdoux_Yamazaki_2015,Yamazaki_2013}, and games between two players \cite{Hernandez_Yamazaki_2013}.

While the reflection strategy is commonly shown to be optimal in these papers, implementing it in real life is often not feasible nor realistic.  Most likely in every scenario, the ability of controlling is limited and hence is more reasonable to restrict the amount of modification one can make to the underlying process.

In classical settings of the optimal dividend problem, reflection strategies that are shown to be optimal always lead to  ruin in finite time; these results have been in dispute and there has been a need for new models so as to avoid these undesirable conclusions.  In currency rate control (see, e.g., \cite{jeanblanc1993impulse} and \cite{mundaca1998optimal}), where a central bank controls the currency rate so as to prevent it from going too high or too low,  reflection strategies, that are shown to be optimal  under many models, are not implementable in real life. In 2011, Swiss central bank introduced the peg so as to stabilize their currency in response to its escalating value against other currencies: they set a ceiling on its rate against Euro, and literally implemented a reflection strategy. However, this only lasted until the beginning of 2015 when they had to admit the enormous cost of doing it and scraped the ceiling.

One way to model these limitations on the control set is to add an extra restriction on the set of admissible strategies to be absolutely continuous with respect to the Lebesgue measure.  Kyprianou et al.\ \cite{Kyprianou_Loeffen_Perez}, in particular, consider the optimal dividend problem under this restriction. They show under the {\it completely monotone assumption on the \lev measure} that a refraction strategy is optimal: Roughly speaking,  the resulting controlled process becomes a refracted \lev process that progresses like the original \lev process below a chosen threshold while it does like a drift-changed process above it.

In this paper, we revisit the stochastic control problem that minimizes the convex running costs.  The setting is similar to the existing literature on singular/impulse control problems as in \cite{Bensoussan_2009}, \cite{Bensoussan_2005} and \cite{Yamazaki_2013}.  The convex assumption is typically needed and is required
 in, e.g., \cite{baccarin2002optimal, constantinides1978existence, sulem1986solvable}.   Differently from these papers, we consider a version with an additional condition that the strategy is absolutely continuous.  More precisely, a strategy must be of the form $L_t = \int_0^t \ell_s \diff s$, $t \geq 0$, with $\ell$ restricted to take values in $[0,\delta]$ uniformly in time.
 As in \cite{Kyprianou_Loeffen_Perez}, our objective is to show the optimality of a refraction strategy.  While a reflected \lev process can be dealt relatively easily because it moves just like the original \lev process except at the times running maximum/minimum is updated, a refracted \lev process moves in two different ways depending on whether it is above or below the given threshold.  The computation thus becomes more intricate.
 
  In order to tackle our problem, we take the following steps.
 \begin{enumerate}
 \item By directly using the resolvent measure for the refracted \lev process as in \cite{kyprianou2010refracted}, we first write  the expected NPV of total costs using the scale function under each refraction strategy.
 \item The candidate threshold $b^* \in [-\infty, \infty]$ is then chosen by observing the identity \eqref{identity_v_u} below.  We show that such $b^*$ is given as a root of a monotone function $I$ defined in \eqref{def_I_new} below, or otherwise it is either $-\infty$ or $\infty$.
 \item We then verify the optimality of such strategy.  Toward this end, we derive a sufficient condition for optimality (verification lemma), and then show that it is equivalent to certain conditions on the smoothness and slope of the value function.   We observe that the level $b^*$ is such that the candidate value function becomes smooth and convex; using these facts,  we complete the proof of optimality.\end{enumerate}

 It is important to highlight the fact that this set of results are obtained without using directly the HJB equation associated with this control problem and, instead, it is used in an indirect way to derive sufficient conditions for optimality, based only on the derivative of the value function; see Lemmas \ref{verificationlemma} and \ref{tussenlemma}.

In addition to solving this problem, we analyze the behavior as the upper bound $\delta$, of the control set, becomes large. Being consistent with our intuition, we show that the optimal refraction strategy of this problem converges to the reflection strategy that has been shown to be optimal in the original problem without the absolutely continuous condition.  In particular, we show the convergence of the optimal threshold $b^*$ and value function to those obtained in  \cite{Yamazaki_2013}.
In order to discuss the practical side of implementing the obtained optimal strategy, we  give a series of computational experiments using the phase-type \lev process and quadratic cost function.  We confirm the optimality as well as the convergence as $\delta \uparrow \infty$ to the reflection strategy.

% Differently from reflection strategies, dealing with refraction strategy, we need to deal with two different processes.
%
%
% In addition, typical singular control models often lead to undesirable results.  For example, in de Finetti's problem,  optimal reflection strategy always leads to ruin in a finite time.  For these reasons, it is more reasonable to set a limit and focus on absolutely continuous strategies.
%
%Under the restriction that the strategy must be absolutely continuous strategy, it is naturally guessed under certain conditions that an optimal strategy becomes a
%
%
%In the above mentioned problems where the reflection strategy is shown to be optimal, implementing the strategy is in real life typically not feasible nor realistic.  In de Finetti's dividend problem, for example, it is reasonable to restrict the amount of dividend to be paid all time.
%A refracted \lev process, recently developed by \cite{kyprianou2010refracted}, is in word a process that is refracted at  a given threshold.  It progresses like the original \lev process below the threshold while it does like a drift-changed process above the threshold.  This process is important in applications.
%
%
%Recently, Kyprianou and Loeffen developed a version of the reflected \lev process..
%
%Sometimes, reflection strategy is not realistic.  For example, in de Finetti's problem, the ruin occurs always occur at a finite time.  In inventory (and currency control), it is not practically possible.

The rest of the paper is organized as follows. The problem  is defined in Section \ref{section_model}.  In Section \ref{section_refraction_strategy}, we first review the refracted spectrally negative \lev process and the scale function.  We then compute the NPV of total costs under a  refraction strategy and choose a candidate refraction boundary level $b^*$.  Section \ref{section_optimality} verifies the optimality of the selected refraction strategy.  Section \ref{section_convergence} studies the convergence as $\delta \uparrow \infty$ of the optimal refraction strategy to the optimal reflection strategy.  We conclude the paper with numerical results in Section \ref{section_numerics}.

\section{Mathematical Model} \label{section_model}
Let $(\Omega, \mathcal{F}, \p)$ be a probability space hosting a \emph{spectrally negative \lev process} $X = \left\{X_t; t \geq 0 \right\}$ whose \emph{Laplace exponent} is given by
\begin{align}
\psi(\theta) := \log\E [ \mathrm{e}^{\theta X_1}] =  \gamma \theta + \frac {\sigma^2} 2 \theta^2 + \int_{(-\infty, 0)} ( \mathrm{e}^{\theta z} - 1 -\theta z \mathbf{1}_{\{z > -1\})} \nu (\diff z), \quad \theta \geq 0, \label{laplace_spectrally_positive}
\end{align}
where $\nu$ is a \lev measure with  support in $(-\infty,0)$ and  satisfying the integrability condition $\int_{(-\infty,0)} (1 \wedge z^2) \nu(\diff z) < \infty$.  It has paths of bounded variation if and only if $\sigma = 0$ and $\int_{(-1,0)}|z|\, \nu(\diff z) < \infty$; in this case, we write \eqref{laplace_spectrally_positive} as
\begin{align*}
\psi(\theta)   =  \tilde{\gamma} \theta +\int_{(-\infty,0)} ( \mathrm{e}^{\theta z}-1 ) \nu (\diff z), \quad \theta \geq 0,
\end{align*}
with $\tilde{\gamma} := \gamma - \int_{(-1,0)}z\, \nu(\diff z)$.  We exclude the case in which $X$ is the negative of a subordinator (i.e., $X$ has monotonically decreasing paths a.s.). This assumption implies that $\tilde{\gamma} > 0$ when $X$ is of bounded variation.
Let $\mathbb{P}_x$ be the conditional probability under which $X_0 = x$ (also let $\mathbb{P} \equiv \mathbb{P}_0$), and let $\mathbb{F} := \left\{ \mathcal{F}_t; t \geq 0 \right\}$ be the filtration generated by $X$.

Fix $\beta \in \R$, $\delta > 0$ and a measurable function $h: \R \rightarrow \R$ that is specified in Assumption \ref{assump_h} below.  Define $\Pi_\delta$ as the set of \emph{absolutely continuous strategies} $\ell$ given by adapted processes $L_t = \int_0^t \ell_s \diff s$, $t \geq 0$, with $\ell$ restricted to take values in $[0,\delta]$ uniformly in time.  For $q > 0$ fixed, the objective is to consider the NPV of the expected total costs
\begin{align}
v_{\ell}(x) := \E_x \Big[\int_0^\infty  \mathrm{e}^{-qt} (h (U_t^{\ell}) + \beta \ell_t ) \diff t \Big], \label{v_pi_def}
\end{align}
where
\begin{align*}
U_t^{\ell} := X_t - L_t, \quad t \geq 0,
\end{align*}
and compute the (optimal) value function
\begin{align}
v(x) := \inf_{{\ell} \in \Pi_\delta} v_{\ell}(x), \quad x \in \R, \label{def_value_function}
\end{align} as well as the optimal strategy that attains it, if such a strategy exists.

We make the following  standing assumptions on the \lev process $X$ and the running cost function $h$.
\begin{assump} \label{assump_levy_measure}
\begin{enumerate}
\item For the case $X$ is of bounded variation, we assume that $\tilde{\gamma} - \delta > 0$.
\item We assume that there exists $\bar{\theta} > 0$ such that  $\int_{(-\infty, -1]} \exp (\bar{\theta} |z|) \nu(\diff z) < \infty$.
%and hence $\exp \psi(-\bar{\theta}) = \E [\exp (-\bar{\theta} X_1)] < \infty$.
\end{enumerate}
\end{assump}
%We further assume the following for the running cost function $h$.
\begin{assump} \label{assump_h}
We assume that $h$ is convex and has at most polynomial growth in the tail.  That is to say, there exist $k_1, k_2 > 0$ and $N \in \mathbb{N}$ such that $h(x) \leq k_1 + k_2 |x|^N$ for all $x \in \R$ such that $|x| > m$.
\end{assump}

Define the drift-changed \lev process
\begin{align}
Y_t := X_t - \delta t, \quad t \geq 0, \label{def_Y}
\end{align}
which is the resulting controlled process if $\ell$ is uniformly set to  be the maximal value $\delta$. Assumption \ref{assump_levy_measure} (1) guarantees that $Y$ is not the negative of a subordinator, and hence is again a spectrally negative \lev process.  This is commonly assumed (see, e.g., \cite{Kyprianou_Loeffen_Perez}) so that the refracted \lev process given below is well-defined.

By Assumption \ref{assump_levy_measure} (2), the domain of the Laplace exponent \eqref{laplace_spectrally_positive} can be extended to $(-\bar{\theta}, \infty)$ and, by its continuity, we can choose sufficiently small $\theta > 0$ such that $\psi(-\theta) < q$. With this choice of $\theta$,
\begin{align}
\E [ \mathrm{e}^{-\theta X_{\EQ} }] = q \int_0^\infty  \mathrm{e}^{-q t} \E [ \mathrm{e}^{-\theta X_t }] \diff t = q \int_0^\infty  \mathrm{e}^{-q t}  \mathrm{e}^{\psi(-\theta)t} \diff t < \infty, \label{laplase_x_eq_minus_theta}
\end{align}
where $\EQ$ is an independent exponential random variable with parameter $q$.

Assumptions \ref{assump_levy_measure} (2) and \ref{assump_h} guarantee that $v_{\ell}$ as in \eqref{v_pi_def} is well-defined and finite for all ${\ell} \in \Pi_\delta$. To see this, by the convexity of $h$ and because $Y_t \leq U_t^{\ell} \leq X_t$,
\begin{align}
h(U_t^{\ell}) \leq   h(X_t) \vee h(Y_t), \label{h_upper}
\end{align}
while, with $\underline{h} := \inf_{x \in \R} h(x) \in [-\infty, \infty)$, we have
\begin{align}
h(U_t^{\ell}) \geq  \left\{ \begin{array}{ll}  \underline{h} & \textrm{if } \underline{h} > -\infty, \\ h(X_t) \wedge h(Y_t) & \textrm{otherwise}.  \label{h_lower}
\end{array} \right.
\end{align}
Hence, for any random time $T$ (independent of $\EQ$),
\begin{align} \label{uniformly_integrability_A}
\begin{split}
\E_x \Big[\int_T^\infty  \mathrm{e}^{-qt} \sup_{{\ell} \in \Pi_\delta} |h (U_t^{\ell})| \diff t \Big] \leq \E_x \Big[\int_T^\infty  \mathrm{e}^{-qt} (|h (X_t)|+ |h (Y_t)|+|\underline{h}| \mathbf{1}_{\{\underline{h} > -\infty\}} ) \diff t \Big] 
\\
= q^{-1}\E_x \Big[\mathbf{1}_{\{ {\EQ} > T \}} (|h (X_{\EQ})| + |h (Y_{\EQ})| +  |\underline{h}| \mathbf{1}_{\{\underline{h} > -\infty\}}) \Big] 
< \infty, 
\end{split}
\end{align}
where the finiteness holds by \eqref{laplase_x_eq_minus_theta} and Assumption \ref{assump_h}.
The finiteness \eqref{uniformly_integrability_A} will be particularly important in the verification step (Lemma \ref{verificationlemma} below). As is commonly used in verification arguments, localization is first needed;  \eqref{uniformly_integrability_A} lets us interchange the limits over integrals.

%where
%\begin{align*}
%h_b(x) := h(x) + \delta 1_{\{x > b\}}, \quad x \in \R.
%\end{align*}

\begin{remark} \label{remark_dual}
We can also consider a version of this problem where a linear drift is added to the increments of $X$ (as opposed to being subtracted): one wants to minimize, for some $\tilde{\beta} \in \R$, the NPV
\begin{align*}
\tilde{v}_{\ell}(x) := \E_x \Big[\int_0^\infty  \mathrm{e}^{-qt} (h (X_t + L_t) + \tilde{\beta} \ell_t ) \diff t \Big].
\end{align*}
However, it is easy to verify that this problem is equivalent to the problem described above.

To see this, we use $Y$ as in \eqref{def_Y} and set $\tilde{L}_t := \delta t - L_t$ (which is admissible with $\tilde{L}_t = \int_0^t \tilde{\ell}_s \diff s$ where $\tilde{\ell}_s := \delta - \ell_s \in [0, \delta]$ a.s.). Then, we can write
\begin{align*}
\tilde{v}_{\ell}(x) = \E_x \Big[\int_0^\infty  \mathrm{e}^{-qt} (h (Y_t - \tilde{L}_t) - \tilde{\beta} \tilde{\ell}_t ) \diff t \Big] + \frac {\tilde{\beta} \delta} q.
\end{align*}
Hence it is equivalent to solving our problem for $\beta := -\tilde{\beta}$.
\end{remark}

\section{Refraction Strategies} \label{section_refraction_strategy}

One of the objectives of this paper is to show the optimality of the \emph{refraction strategies}, say $\ell^b \in \Pi_\delta$, under which the controlled process becomes the refracted \lev process $U^b = \{U^b_t; t \geq 0\}$ of  \cite{kyprianou2010refracted}, with a suitable choice of the refraction boundary $b \in \R$.  By \cite[Theorem 1 and Remark 3]{kyprianou2010refracted}, this is a strong Markov process given by the unique strong solution to the SDE
\begin{align*}
\diff U_t^b  = \diff X_t - \delta \mathbf{1}_{\{U_t^b > b\}} \diff t, \quad t \geq 0.
\end{align*}
Namely, $U^b$ progresses like $X$ below the boundary $b$ while it does like $Y$ above $b$.
Let us write the corresponding NPV of the total costs associated with $\ell^b$ by
\begin{align}
v_b(x) := \E_x \Big[\int_0^\infty  \mathrm{e}^{-qt} (h (U_t^b ) + \beta \delta \mathbf{1}_{\{U_t^b > b\}}) \diff t \Big], \quad x \in \R.  \label{v_b_def}
\end{align}

\subsection{Scale functions}
As it has been studied in \cite{kyprianou2010refracted}, the expected NPV \eqref{v_b_def}  can be expressed in terms of  the scale functions of the two spectrally negative \lev processes $X$ and $Y$.
Following the same notations as in \cite{kyprianou2010refracted}, we use $W^{(q)}$ and $\mathbb{W}^{(q)}$ for the scale functions of $X$ and $Y$, respectively.  Namely,  these are the mappings from $\R$ to $[0, \infty)$ that take value zero on the negative half-line, while on the positive half-line they are strictly increasing functions that are defined by their Laplace transforms:
\begin{align} \label{scale_function_laplace}
\begin{split}
\int_0^\infty  \mathrm{e}^{-\theta x} W^{(q)}(x) \diff x &= \frac 1 {\psi(\theta)-q}, \quad \theta > \Phi(q), \\
\int_0^\infty  \mathrm{e}^{-\theta x} \mathbb{W}^{(q)}(x) \diff x &= \frac 1 {\psi(\theta)-\delta \theta -q}, \quad \theta > \varphi(q),
\end{split}
\end{align}
where
\begin{align}
\begin{split}
\Phi(q) := \sup \{ \lambda \geq 0: \psi(\lambda) = q\} \quad \textrm{and} \quad \varphi(q) := \sup \{ \lambda \geq 0: \psi(\lambda) - \delta \lambda = q\}.
\end{split}
\label{def_varphi}
\end{align}
By the strict  convexity of $\psi$, we derive the strict inequality $\varphi(q) > \Phi(q) > 0$.
\par Scale functions appear in the vast majority of
	known identities concerning boundary crossing problems and related path decompositions.  In order to illustrate the importance of the scale functions we provide 
	the following result, the so-called two-sided exit problem (see for instance Theorem 8.1 in \cite{Kyprianou_2006}). Define
	\[
	\tau_a^+=\inf\{t>0:X_t>a\}\quad\text{and}\quad\tau_0^-=\inf\{t>0:X_t<0\}.
	\]
	Then for all $q\geq 0$ and $x<a$, $\E_x \big[e^{-q\tau_a^+}1_{\{\tau_a^+<\tau_0^-\}} \big]= W^{(q)}(x) / W^{(q)}(a)$.
For other applications of the scale function, see, among others,  \cite{Kuznetsov2013} and \cite{Kyprianou_2006}.

Fix $\lambda \geq 0$ and define $\psi_\lambda(\cdot)$ as the Laplace exponent of $X$ under $\p^\lambda$ with the change of measure
\begin{align*}
\left. \frac {\diff \p^\lambda} {\diff \p}\right|_{\mathcal{F}_t} = \exp(\lambda X_t - \psi(\lambda) t), \quad t \geq 0;
\end{align*}
see  \cite[Section 8]{Kyprianou_2006}.
% for all $s > -\lambda$,
%\begin{align*}
%\psi_\lambda(s) :=\Big(  \lambda \sigma^2  + c - \int_0^1 u (e^{-\lambda u}-1) \nu(\diff u)\Big) s
%+\frac{1}{2}\sigma^2 s^2 +\int_0^\infty (e^{- s u}-1 + s u 1_{\{ 0 < u < 1 \}}) e^{-\lambda u}\,\nu(\diff u).
%\end{align*}
 Suppose that $W_\lambda^{(q)}$ is the scale function associated with $X$ under $\p^\lambda$ (or equivalently with $\psi_\lambda(\cdot)$).
Then, by Lemma 8.4 of \cite{Kyprianou_2006}, $W_\lambda^{(q-\psi(\lambda))}(x) =  \mathrm{e}^{-\lambda x} W^{(q)}(x)$, $x \in \R$.  In particular, we define
\begin{align}
W_{\Phi(q)}(x) := W_{\Phi(q)}^{(0)}(x) =  \mathrm{e}^{-\Phi(q) x} W^{(q)}(x), \quad x \in \R, \label{scale_function_version}
\end{align}
which is known to be an increasing function and,  as in Lemma 3.3 of \cite{Kuznetsov2013},
\begin{align}
W_{\Phi(q)} (x) \nearrow \psi'(\Phi(q))^{-1}  < \infty \quad \textrm{as } x \rightarrow \infty. \label{W_q_limit}
\end{align}

Below, we summarize the properties of the scale function that will be necessary in deriving our results.
\begin{remark} \label{remark_smoothness_zero}
\begin{enumerate}
\item If $X$ is of unbounded variation or the \lev measure is atomless, it is known that $W^{(q)}$ is $C^1(\R \backslash \{0\})$; see, e.g.,\ \cite[Theorem 3]{Chan_2009}.
%Hence,
%\begin{enumerate}
%\item $Z^{(q)}$ is $C^1 (\R \backslash \{0\})$ and $C^0 (\R)$ for the bounded variation case, while it is $C^2(\R \backslash \{0\})$ and $C^1 (\R)$ for the unbounded variation case, and
%\item $\overline{Z}^{(q)}$ is $C^2(\R \backslash \{0\})$ and $C^1 (\R)$ for the bounded variation case, while it is $C^3(\R \backslash \{0\})$ and $C^2 (\R)$ for the unbounded variation case.
%\end{enumerate}
\item Regarding the asymptotic behavior near zero, as in Lemmas 3.1 and 3.2 of \cite{Kuznetsov2013},
\begin{align}\label{eq:Wqp0}
\begin{split}
W^{(q)} (0) &= \left\{ \begin{array}{ll} 0 & \textrm{if $X$ is of unbounded
variation,} \\ \frac 1 {\tilde{\gamma}} & \textrm{if $X$ is of bounded variation,}
\end{array} \right. \\
W^{(q)\prime} (0+) &:= \lim_{x \downarrow 0}W^{(q)\prime} (x) =
\left\{ \begin{array}{ll}  \frac 2 {\sigma^2} & \textrm{if }\sigma > 0, \\
\infty & \textrm{if }\sigma = 0 \; \textrm{and} \; \nu(-\infty,0) = \infty, \\
\frac {q + \nu(-\infty, 0)} {\tilde{\gamma}^2} &  \textrm{if }\sigma = 0 \; \textrm{and} \; \nu(-\infty, 0) < \infty.
\end{array} \right.
\end{split}
\end{align}
\item As in (8.22) and Lemma 8.2 of \cite{Kyprianou_2006}, $ {W^{(q)\prime}(y+)} / {W^{(q)}(y)} \leq  {W^{(q)\prime}(x+)} / {W^{(q)}(x)}$ for  $y > x > 0$.
In all cases, $W^{(q)\prime}(x-) \geq W^{(q)\prime}(x+)$ for all $x >0$.
\end{enumerate}
\end{remark}

We also define, for all $x > 0$,
\begin{align}
\Theta^{(q)} (x) &:= W^{(q)\prime}(x+) - \Phi(q) W^{(q)}(x) =  \mathrm{e}^{\Phi(q)x} W_{\Phi(q)}'(x+) > 0.\label{def_Theta}
\end{align}
Let $\underline{X}_t := \inf_{0 \leq t' \leq t} X_{ t'}$, $t \geq 0$, be the running infimum process.
By Corollary 2.2 of \cite{Kuznetsov2013}, for Borel subsets in the nonnegative half line,
\begin{align}
\p \{ -\underline{X}_{\EQ} \in \diff x \} = \frac q {\Phi(q)} W^{(q)} (\diff x) - q W^{(q)} (x) \diff x = \frac q {\Phi(q)}[\Theta^{(q)}(x) \diff x + W^{(q)}(0) \delta_0(\diff x)], \label{density_running_min} \end{align}
where $W^{(q)}(\diff x)$ is the measure such that $W^{(q)}(x) = \int_{[0,x]}W^{(q)}(\diff z)$  (see  \cite[(8.20)]{Kyprianou_2006}) and $\delta_0$ is the Dirac measure at zero.

\begin{remark} \label{remark_wiener_hopf}
\begin{enumerate}
\item  For the case of spectrally negative \lev process $X$,
\begin{align*}
\E [ \mathrm{e}^{\theta \underline{X}_{\EQ}}]= \frac q { \Phi(q)}\frac { \theta - \Phi(q)} {\psi(\theta) -q}, \quad \theta \geq 0,
\end{align*}
and $\overline{X}_{\EQ}$ is exponentially distributed with parameter $\Phi(q)$ where $\overline{X}_t := \sup_{0 \leq t' \leq t} X_{ t'}$, $t \geq 0$, is the running supremum process; see Section 8 of \cite{Kyprianou_2006}.
\item It is easy to see that \begin{align} \label{expressions_theta}
\begin{split}
\int_0^{\infty}\Theta^{(q)}(u)\diff u &= \frac {\Phi(q)} q \p \{ \underline{X}_{\EQ} \neq 0\} = \frac{\Phi(q)}{q} - W^{(q)}(0), \\
\int_0^{\infty}u\Theta^{(q)}(u)\diff u&=\frac{\Phi(q)}{q}\mathbb{E}[-\underline{X}_{\EQ}],\\
%\int_0^{\infty}u\Theta^{(q)\prime}(u)\diff u&=-\frac{\Phi(q)}{q}\notag\\
\int_0^{\infty} \mathrm{e}^{-\varphi(q)u}\Theta^{(q)}(u)\diff u &= \frac {\Phi(q)} q \E \Big[  \mathrm{e}^{\varphi(q) \underline{X}_{\EQ}}\mathbf{1}_{\{\underline{X}_{\EQ} \neq 0\}} \Big] = \frac {\varphi(q)-\Phi(q)} {\delta \varphi(q)} - W^{(q)}(0).
\end{split}
\end{align}
For any $-M < x$, by \eqref{def_Theta} and \eqref{density_running_min},
\begin{multline}
\int_{-\infty}^{-M} |h(y)|   \mathrm{e}^{-\Phi(q)y} W_{\Phi(q)}' (x-y) \diff y =   \mathrm{e}^{-\Phi(q)x} \int_{-\infty}^{-M} |h(y)|  \Theta^{(q)}(x-y) \diff y \\
= \frac {\Phi(q)} q  \mathrm{e}^{-\Phi(q)x}  \int_{-\infty}^{-M} |h(y)|  \p \{ \underline{X}_{\EQ} +x \in \diff  y\}  =  \frac {\Phi(q)} q  \mathrm{e}^{-\Phi(q)x} \E_x [|h(\underline{X}_{\EQ})| \mathbf{1}_{\{\underline{X}_{\EQ} \leq -M \}}]. \label{eq_integrability1}
\end{multline}
\end{enumerate}
\end{remark}

Now, recall Assumption \ref{assump_levy_measure} (2) and fix any $0 < \theta < \bar{\theta}$.
By the independence of $X_{\EQ} - \underline{X}_{\EQ}$ and $\underline{X}_{\EQ}$, $\E [ \mathrm{e}^{-\theta X_{\EQ}}] = \E [ \mathrm{e}^{-\theta (X_{\EQ} - \underline{X}_{\EQ})}] \E [ \mathrm{e}^{\theta |\underline{X}_{\EQ}| }]$, where in particular $X_{\EQ} - \underline{X}_{\EQ} \sim \overline{X}_{\EQ}$ (by duality) is exponentially distributed with parameter $\Phi(q)$ as in Remark \ref{remark_wiener_hopf} (1).   Hence, we see that
\begin{align}
\E [ \mathrm{e}^{\theta |\underline{X}_{\EQ}| }] < \infty. \label{finiteness_laplace_negagtive}
\end{align}
By \eqref{finiteness_laplace_negagtive}
 and Assumption  \ref{assump_h},
 \eqref{eq_integrability1} is  finite.

%\red{may be able to remove about $\Upsilon$}We also define, for all $x > 0$,
%\begin{align*}
%\overline{\Upsilon}^{(q)} (x) &:= \int_0^\infty e^{- \varphi(q) z} \Theta^{(q)} (z+x) \diff z, \\
%\Upsilon^{(q)} (x) &:= \overline{\Upsilon}^{(q)'} (x) =- \Theta^{(q)}(x) + \varphi(q) \int_0^\infty e^{-\varphi(q) z} \Theta^{(q)} (z+x) \diff z \\
%  &= - W^{(q)\prime}(x) + (\varphi(q) - \Phi(q)) \int_0^\infty e^{-\varphi(q) z} W^{(q)\prime}(z+x) \diff z.
%\end{align*}

\subsection{The expected NPV \eqref{v_b_def} in terms of scale functions}
Fix $b, x \in \R$. By Theorem 6 (iv) of \cite{kyprianou2010refracted}, the resolvent measure
\begin{align*}
R_b (x,B) :=q^{-1}\p_x \{ U^b_{\EQ} \in B \} =  \E_x \Big[\int_0^\infty  \mathrm{e}^{-qt} \mathbf{1}_{\{U_t^b \in B\}} \diff t \Big], \quad B \in \mathcal{B} (\R),
\end{align*}
admits a density
\begin{align} \label{resolvent_measure}
R_b (x, \diff y) = (r_b^{(1)}(x,y) + r_b^{(2)}(x,y) \mathbf{1}_{\{x > b\}})\diff y, \quad y \in \R,
\end{align}
given by
\begin{align} \label{r_expression}
\begin{split}
r_b^{(1)} (x,y) &:= \left\{ \begin{array}{ll}  \frac {\varphi(q)-\Phi(q)} {\delta \Phi(q)}   \mathrm{e}^{\Phi(q) (x-b) - \varphi(q) (y -b)}  & y \geq b, \\  \mathrm{e}^{\Phi(q)(x-b)}  \frac {\varphi(q)-\Phi(q)} {\Phi(q)}  \mathrm{e}^{\varphi(q) b} \int_b^\infty  \mathrm{e}^{-\varphi(q)z} W^{(q)\prime} (z-y) \diff z - W^{(q)}(x-y) & y < b, \end{array} \right. \\
r_b^{(2)} (x,y) &:=\left\{ \begin{array}{ll}    \mathrm{e}^{-\varphi(q)(y-b)}M(x; b) - \mathbb{W}^{(q)}(x-y) & y \geq b, \\
 \delta M(x; b)  \mathrm{e}^{\varphi(q) b} \int_{b}^\infty  \mathrm{e}^{-\varphi(q) z} W^{(q)\prime}(z-y) \diff z - \delta \int_{b}^x \mathbb{W}^{(q)}(x-z) W^{(q)\prime} (z-y) \diff z & y < b, \end{array} \right.
 \end{split}
\end{align}
where we define
\begin{align}
%B^{(q)}(s, b) &:= e^{\varphi(q) b} \int_b^\infty e^{-\varphi(q)z} W^{(q)\prime} (z+s) \diff z \\
%&= - W^{(q)}(b+s) + \varphi(q) e^{\varphi(q) b} \int_b^\infty e^{-\varphi(q) z} W^{(q)} (z+s) \diff z \\
%&=- W^{(q)}(b+s) + \varphi(q) \int_0^\infty e^{-\varphi(q) z} W^{(q)} (z+b+s) \diff z, \quad s > -b, \\
M(x;b) &:=  (\varphi(q)-\Phi(q))   \mathrm{e}^{-\Phi(q) b} \int_b^x  \mathrm{e}^{\Phi(q)z} \mathbb{W}^{(q)}(x-z) \diff z, \quad x > b. \label{def_M}
\end{align}

Hence, \eqref{v_b_def} can be written (see  Remark \ref{remark_finiteness_v_1_2} below) as
\begin{align}
v_b(x) = v_b^{(1)} (x) + v_b^{(2)} (x) \mathbf{1}_{\{x > b\}} \label{v_b_sum},
\end{align}
 where
\begin{align} \label{v_b_expression}
\begin{split}
v_b^{(1)}(x)
%&=\int_b^\infty (h(y)+ \beta \delta)  r_b(x, y) \diff y + \int_{-\infty}^b h(y) r_b(x,y) \diff y \\
&:=    \mathrm{e}^{\Phi(q) (x  - b)} \frac {\varphi(q)-\Phi(q)} {\delta \Phi(q)} \Big[  \int_0^\infty h(y+b)   \mathrm{e}^{- \varphi(q) y} \diff y +   \frac {\beta \delta}  {\varphi(q)} \Big]  \\ &+ \int_{-\infty}^0 h(y+b) \Big[  \mathrm{e}^{\Phi(q)(x-b)} \frac {\varphi(q)-\Phi(q)} {\Phi(q)} \int_0^\infty  \mathrm{e}^{-\varphi(q)z} W^{(q)\prime} (z-y) \diff z - W^{(q)}(x-b-y) \Big] \diff y,
\end{split} \\
\label{v_J_diff_b_star}
\begin{split}
v_b^{(2)}(x) &:=\int_{0}^\infty (h(y+b) + \beta \delta) \Big\{   \mathrm{e}^{-\varphi(q)y}M(x; b) - \mathbb{W}^{(q)}(x-b-y) \Big\} \diff y \\
&+ \delta \int_{-\infty}^{0} h(y+b) \Big\{ M(x; b) \int_{0}^\infty  \mathrm{e}^{-\varphi(q) z} W^{(q)\prime}(z-y) \diff z - \int_{b}^x \mathbb{W}^{(q)}(x-z) W^{(q)\prime} (z-b-y) \diff z \Big\} \diff y.
\end{split}
\end{align}

\begin{remark} \label{remark_finiteness_v_1_2}
In view of  \eqref{v_b_sum}, both $v_b^{(1)}$ and $v_b^{(2)}$ are finite and hence the decomposition is well-defined.  To see this notice that, by the estimation \eqref{uniformly_integrability_A}, $v_b^{(1)}(x)$  is finite for $x \leq b$. Moreover, in view of the form of $r_b^{(1)}$ in \eqref{r_expression}, $v_b^{(1)}(x)$ for $x > b$ is also finite. 
Indeed, for any $x > b$ and $y < b$,
\begin{align*}
r_b^{(1)} (x,y) = e^{\Phi(q) (x-b)}r_b^{(1)} (b,y) - W^{(q)} (x-y) + e^{\Phi(q)(x-b)} W^{(q)} (b-y).
\end{align*}
By the decomposition (95) of \cite{Kuznetsov2013}, $-W^{(q)} (x-y) + e^{\Phi(q)(x-b)} W^{(q)} (b-y)  =  \hat{u}^{(q)} (x-y) - e^{\Phi(q)(x-b)} \hat{u} (b-y)$ (where $\hat{u}^{(q)}$ is the $q$-resolvent density of $-X$ as in Theorem 2.4 (iv) of \cite{Kuznetsov2013}), with which $h$ is integrable by \eqref{uniformly_integrability_A}.
Finally, using  again  \eqref{uniformly_integrability_A}, we conclude that  $v_b^{(2)}(x)$ is finite as well, for $x > b$.
%\red{[So, I just realized that $r_b^{(1)}$ and $r_b^{(2)}$ may get negative; but I think it is ok from the above that the negative part of $r_b^{(1)}$ is integrable.]  
In addition,
\begin{align}
\int_{-\infty}^{\infty} |h(y)| |r_b^{(1)}(x, y)| \diff y < \infty. \label{h_r_abs_integrable1}
\end{align}
 Since $\int_{-\infty}^{\infty} |h(y)| (r_b^{(1)}+r^{(b)}_2)(x, y) \diff y  < \infty$ and  $r_b^{(1)} + r_b^{(2)} > 0$, we must also have 
\begin{align}
\int_{-\infty}^{\infty} |h(y)| |r_b^{(2)}(x, y)| \diff y<\infty. \label{h_r_abs_integrable2}
\end{align}
\end{remark}

%Because
%\begin{align*}
%e^{\varphi(q) b} \int_b^\infty e^{-\varphi(q) z} W^{(q)} (z-y) \diff z = \int_b^\infty e^{-\varphi(q) (z-b)} W^{(q)} (z-y) \diff z = \int_0^\infty e^{-\varphi(q) z} W^{(q)} (z+b-y) \diff z,
%\end{align*}
%we can write
%\begin{multline*}
%v_b(x)
%= \frac {\varphi(q)-\Phi(q)} {\delta \Phi(q)} \int_b^\infty (h(y)+ \delta)  e^{\Phi(q) x - \varphi(q) y - b(\Phi(q) - \varphi(q))} \diff y \\ + \int_{-\infty}^b h(y) \Big[ e^{\Phi(q)(x-b)} \frac {\varphi(q)-\Phi(q)} {\Phi(q)} \Big\{- W^{(q)}(b-y) + \varphi(q) \int_0^\infty e^{-\varphi(q) z} W^{(q)} (z+b-y) \diff z \Big\}- W^{(q)}(x-y) \Big] \diff y.
%\end{multline*}

\subsection{First order condition}  We shall first obtain our candidate refraction boundary $b^*$.  In view of \eqref{v_b_expression} and \eqref{v_J_diff_b_star},  (once it is confirmed that the derivatives $\partial / \partial b$ and $\partial / \partial x$  can be interchanged over the integrals) the next  identity holds:
\begin{align}
\frac \partial {\partial b} v_b (x) = u_b(x), \quad x \neq b. \label{identity_v_u}
\end{align}
where
\begin{align*}
u_b(x) := \E_x \Big[ \int_0^\infty  \mathrm{e}^{-qt} h'(U_t^b) \diff t \Big] - v_b'(x), \quad x \neq b.
\end{align*}
Since the first-order condition $\partial v_b (x) /{\partial b} |_{b=b^*} = 0$ is a necessary condition for the optimality of the refraction strategy $\ell^{b^*}$, we shall pursue $b^*$ such that $u_{b^*}(x)$ vanishes; consequently, (if such $b^*$ exists) $v_{b^*}'(x) =\E_x [ \int_0^\infty  \mathrm{e}^{-qt} h'(U_t^{b^*}) \diff t ]$. This identity will  be important  later in the verification of optimality.

Let us define, for  $k = 1,2$,
\begin{align}
u_b^{(k)}(x) &:=\int_{-\infty}^\infty h'(y) r_{b}^{(k)}(x,y) \diff y - v_b^{(k) \prime}(x), \quad x \neq b,\label{u_b_k_def}
\end{align}
so that
\begin{align} \label{u_b_sum}
u_b(x) = u_b^{(1)}(x) + u_b^{(2)}(x) \mathbf{1}_{\{x > b\}}, \quad x \neq b.
\end{align}
%Because these are simplified when $x < b$, we shall first choose $b$ such that $u_b^{(1)}(x)$ vanishes for $x < b$.

% by the first order condition.  In view of \eqref{v_b_sum}, the expression of $v_b$ is simplified when $x < b$.  Hence, we shall obtain the partial derivative $\partial v_b(x) / \partial b$ for $x < b$ and obtain $b^*$ such that it vanishes if such $b^*$ exists.

The next lemma shows that the right hand side of \eqref{u_b_k_def} can be written succinctly using the function:
\begin{align}  \label{def_I_new}
\begin{split}
I(b) &:=  \frac {\varphi(q)-\Phi(q)} {\varphi(q)} \int_0^\infty  h'(y+b)   \mathrm{e}^{- \varphi(q) y}\diff y   + \delta \Big[ \int_{-\infty}^0 h'(y+b) \int_0^\infty  \mathrm{e}^{- \varphi(q) z} \Theta^{(q)} (z-y) \diff z  \diff y -  \beta \frac {\Phi(q)} {\varphi(q)} \Big] \\
&= \frac {\varphi(q)-\Phi(q)} {\varphi(q)} \int_0^\infty  h'(y+b)   \mathrm{e}^{- \varphi(q) y}\diff y \\
&\quad + \frac \delta {\varphi(q)} \Big[ \int_{-\infty}^{0} h'(y+b) \Big\{ (\varphi(q)-\Phi(q))  \int_{0}^\infty  \mathrm{e}^{-\varphi(q) z} W^{(q)\prime}(z-y) \diff z - {\Phi(q)}   W^{(q)} (-y)   \Big\} \diff y
-\beta  \Phi(q)  \Big],
\end{split}
\end{align}
where the second equality holds because, by integration by parts,
\begin{align}
\int_0^\infty  \mathrm{e}^{-\varphi(q)z} W^{(q)\prime} (z-y) \diff z +W^{(q)}(-y) = \varphi(q) \int_0^\infty  \mathrm{e}^{-\varphi(q) z} W^{(q)} (z-y) \diff z, \quad y < 0. \label{some_result_int_by_parts}
\end{align}
 The proof of this result is technical and long (partly because we need to confirm that the derivative $\partial / \partial x$ can go into the integral over the unbounded set), and therefore we shall defer the proof to Appendix \ref{proof_lemma_u_b}.

%\begin{remark}
%Because
%\begin{align*}
%\int_0^\infty e^{-\varphi(q)z} W^{(q)\prime} (z-y) \diff z
%&=- W^{(q)}(-y) + \varphi(q) \int_0^\infty e^{-\varphi(q) z} W^{(q)} (z-y) \diff z,
%\end{align*}
%we can also write
%\begin{multline*}
%I(b) =  \frac {\varphi(q)-\Phi(q)} {\varphi(q)} \int_0^\infty  h'(y+b)  e^{- \varphi(q) y}\diff y \\
%+ \frac \delta {\varphi(q)} \Big[ \int_{-\infty}^{0} h'(y+b) \Big\{ (\varphi(q)-\Phi(q))  \int_{0}^\infty e^{-\varphi(q) z} W^{(q)\prime}(z-y) \diff z - {\Phi(q)}   W^{(q)} (-y)   \Big\} \diff y
%-\beta  \Phi(q)  \Big].
%\end{multline*}
%\end{remark}

\begin{lemma} \label{lemma_u_b}
We have
\begin{align}
u_b^{(1)}(x)&= \frac {\varphi(q) - \Phi(q)} {\delta \Phi(q)}   \mathrm{e}^{\Phi(q)(x-b)} I(b) - h(b)  W^{(q)}(x-b), \quad x \neq b,
 \label{v_b_x_derivative_b} \\
 u_b^{(2)}(x) &=  (M(x; b) - \mathbb{W}^{(q)} (x-b)) I(b) + {h(b)}  W^{(q)} (x-b), \quad x > b.  \label{v_b_x_derivative_b_2}
\end{align}
%\begin{align} \label{def_I}
%I(b)
%&:=  (\varphi(q)-\Phi(q)) \int_0^\infty  h(y+b)   \mathrm{e}^{- \varphi(q) y}  \diff y+  \delta \Big[ \int_{-\infty}^0 h(y+b) \Upsilon^{(q)}(-y) \diff y   -h(b)   W^{(q)}(0) - \frac {\Phi(q)} {\varphi(q)} \beta \Big].
%\end{align} \red{in fact this}
\end{lemma}
 In view of  \eqref{u_b_sum}, this lemma directly implies the following.
\begin{proposition} \label{prop_derivative_all}
For all $x, b \in \R$ such that $x \neq b$,
\begin{align}
u_b(x)&= \Big[ \frac {\varphi(q) - \Phi(q)} {\delta \Phi(q)}   \mathrm{e}^{\Phi(q)(x-b)} + \mathbf{1}_{\{ x > b \}}  (M(x; b) - \mathbb{W}^{(q)} (x-b)) \Big]  I(b). \label{u_b_decomposition}
\end{align}
\end{proposition}

\subsection{Candidate value function} \label{subsection_candidate}

In view of our discussion in the previous subsection and Proposition \ref{prop_derivative_all}, the choice of our candidate threshold level $b^*$ is clear. In \eqref{u_b_decomposition}, the bracket term on the right hand side is uniformly positive for all $x \in \R$. Indeed, for $x > b$, it equals the resolvent density  $r_b^{(1)}(x,b) + r_b^{(2)}(x,b)$ at $b$ in view of  \eqref{r_expression}.  Hence, we shall pursue $b^*$ such that $I(b^*)$ vanishes.  However, the existence and uniqueness of such $b^*$ is not guaranteed, and hence we shall define $b^*$ carefully here.

First, by \eqref{def_Theta}, the first equality of \eqref{def_I_new} and the convexity of $h$, the function $I$ is \emph{nondecreasing}. 
%\begin{align*}
%I'(b) &=  \frac {\varphi(q)-\Phi(q)} {\varphi(q)} \int_0^\infty  h''(y+b)  e^{- \varphi(q) y}\diff y   + \delta \Big[ \int_{-\infty}^0 h''(y+b) \int_0^\infty e^{- \varphi(q) z} \Theta^{(q)} (z-y) \diff z  \diff y \Big] \geq 0.
%\end{align*}
Hence we can define the limits $I(\infty) := \lim_{b \uparrow \infty} I(b) \in (-\infty, \infty]$ and $I(-\infty) := \lim_{b \downarrow -\infty} I(b) \in [-\infty, \infty)$.
We set our candidate optimal threshold level $b^*$ to be \emph{any root} of $I(b) = 0$ if $I(-\infty) < 0 < I(\infty)$.  If $I(\infty) \leq 0$, we let $b^*=\infty$; if $I(-\infty) \geq 0$, we let $b^* = -\infty$. 

%Notice that for any $b \in \R$, $I'(b) = 0$ if and only if $h'$ is a constant (or $h$ is affine); otherwise, $b^*$ (if $-\infty < b^* < \infty$) becomes the \emph{unique} root of $I(b)=0$; see Example \ref{example_linear} below.

In fact, if $h$ is affine, then $I$ is a constant.  Conversely, if $h$ is not affine, $I$ is strictly increasing everywhere. To see the latter, suppose $h'$ is not constant and fix any $b \in \R$. Then, for sufficiently small $-M$, we must have $\int_{-M}^\infty h''(y+b) \diff y + \sum_{-M < y < \infty} \Delta h'(y+b) e^{-\varphi(q) y} > 0$ with $\Delta h'(y) := h'(y+) - h'(y-)$.
% or there exist $y$ such that $\Delta h'(y) > 0$ on $(-M, \infty)$.  
 Then, because $h'$ is nondecreasing,
\begin{align*} 
I'(b) &\geq  \frac {\varphi(q)-\Phi(q)} {\varphi(q)} \Big( \int_0^\infty  h''(y+b)   \mathrm{e}^{- \varphi(q) y}\diff y  + \sum_{0 \leq y < \infty} \Delta h'(y+b) e^{-\varphi(q) y}\Big) \\ &+ \delta  \Big( \int_{-M}^0 h''(y+b) \int_0^\infty  \mathrm{e}^{- \varphi(q) z} \Theta^{(q)} (z-y) \diff z  \diff y + \sum_{-M < y < 0}  \Delta h'(y+b) \int_0^\infty  \mathrm{e}^{- \varphi(q) z} \Theta^{(q)} (z-y) \diff z  \Big)  \\ &> 0.
\end{align*}
Therefore,  for the affine case with $I \equiv 0$  (which holds if and only if $h'(b) = q \beta$ uniformly in $b \in \R$ by Lemma \ref{lemma_classification} below), we set $b^*$ to be any value on $\R$.
Otherwise, $b^*$ (if $-\infty < b^* < \infty$) becomes the \emph{unique} root of $I(b)=0$.

%\begin{remark} \label{remark_I_zero_case}
%For the case $I(-\infty) = I(\infty) = 0$ (which holds if and only if $h'(b) = q \beta$ uniformly in $b \in \R$ by Lemma \ref{lemma_classification} below), we shall let $b^* = \infty$.  However, in this special case with $I(-\infty) = I(\infty) = 0$, the arguments below hold for any choice of $b^* = [-\infty, \infty]$. \end{remark}

%In order to give a specific case where $b^*$ is finite and explicit, we present an example with the quadratic cost function.

\begin{example} \label{example_quadratic}For the case $h(y) := \alpha y^2$, $y \in \R$, for some $\alpha > 0$,
$b^*= \beta q /(2\alpha)+\mathbb{E}(-\underline{X}_{\EQ})-{\varphi(q)}^{-1}$. %\red{removed the proof}
\end{example}

For the cases $b^* = \infty$ and $b^*= -\infty$, we shall set our candidate value functions to be $v_{\infty}(x) := \limsup_{b \uparrow \infty} v_b(x)$ and $v_{-\infty} (x) := \liminf_{b \downarrow -\infty} v_b(x)$, respectively.
%; note that the limit $\lim_{b \rightarrow \infty} v_b(x)$ in fact exists by \eqref{v_b_x_derivative_b} and $\sup_{b \in \R}I(b) \leq 0$.
By \eqref{uniformly_integrability_A},  if we set $\tau_b^+ := \inf \{ t > 0: X_t > b\}$ and  $\kappa_b^- := \inf \{ t > 0: Y_t < b\}$, the estimation \eqref{uniformly_integrability_A} immediately gives
$\E_x [\int_{\tau^+_b}^\infty  \mathrm{e}^{-qt} |h (U_t^b)| \diff t ] \xrightarrow{b \uparrow \infty} 0$ and  $\E_x [\int_{\kappa^-_b}^\infty  \mathrm{e}^{-qt} |h (U_t^b)| \diff t ] \xrightarrow{b \downarrow -\infty} 0$.
Hence, we have probabilistic representations:\begin{align} \label{v_infty_by_expectation}
v_\infty(x)  = \E_x \Big[ \int_0^{\infty}  \mathrm{e}^{-qt} h(X_t) \diff t \Big] \quad \textrm{and} \quad v_{-\infty}(x)  = \E_x \Big[ \int_0^{\infty}  \mathrm{e}^{-qt} h(Y_t) \diff t \Big] + \frac {\beta \delta} q,
\end{align}
which are the NPVs corresponding to the strategies $\ell^\infty \equiv 0$ and $\ell^{-\infty} \equiv \delta$.
By the resolvent measures for $X$ and $Y$, we can write \eqref{v_infty_by_expectation} succinctly using the scale function: by Corollary 8.9 of \cite{Kyprianou_2006},
\begin{align}  \label{v_infty_analytical}
\begin{split}
v_\infty(x)  &= \frac {1} {\psi'(\Phi(q))}  \int_0^\infty h(x+z)  \mathrm{e}^{-\Phi(q)z} \diff z + \int_{-\infty}^0 h(x+z) \hat{u}^{(q)} (-z) \diff z, \\
v_{-\infty}(x)  &= \frac {1} {\psi'(\varphi(q))-\delta}  \int_0^\infty h(x+z)  \mathrm{e}^{-\varphi(q)z} \diff z + \int_{-\infty}^0 h(x+z) \tilde{u}^{(q)} (-z) \diff z + \frac {\beta \delta} q,
\end{split}
\end{align}
where $\hat{u}^{(q)} (w) := \frac {1} {\psi'(\Phi(q))}   \mathrm{e}^{\Phi(q)w} - W^{(q)}(w)$ and $\tilde{u}^{(q)} (w) := \frac {1} {\psi'(\varphi(q)) - \delta}   \mathrm{e}^{\varphi(q)w} - \mathbb{W}^{(q)}(w)$, for all $w \geq 0$.

The criteria for $b^* = \infty$ and $b^* = -\infty$ can be obtained concisely below. Let $h'(\infty) :=\lim_{y \rightarrow \infty} h'(y) \in (-\infty, \infty]$ and $h'(-\infty) :=\lim_{y \rightarrow -\infty} h'(y)  \in [-\infty, \infty)$ (which exist by the convexity of $h$).
\begin{lemma} \label{lemma_classification} We have
\begin{align*}
I(\infty)  = \frac {\delta \Phi(q)} {\varphi(q)} \Big(  \frac {h'(\infty)} q
     -  \beta   \Big) \quad \textrm{and} \quad I(-\infty)  = \frac {\delta \Phi(q)} {\varphi(q)} \Big(  \frac {h'(-\infty)} q
     -  \beta   \Big).
  \end{align*}
  Hence, $I(-\infty) = I(\infty) = 0$ if and only if $h'(b) = q\beta$ for all $b \in \R$
  % (by Remark \ref{remark_I_zero_case}, we set $b^*=\infty$). 
  Otherwise, $b^* = \infty$ if and only if $h'(\infty) \leq q \beta$ and $b^* = -\infty$ if and only if $h'(-\infty) \geq q \beta$.
\end{lemma}

\begin{proof}
By monotone convergence applied to the first equality of \eqref{def_I_new},
\begin{align*}
I(\infty)  &= h'(\infty)
\Big[ \frac {\varphi(q)-\Phi(q)} {\varphi(q)^2}  + \delta  \int_{-\infty}^0  \int_0^\infty  \mathrm{e}^{- \varphi(q) z} \Theta^{(q)} (z-y) \diff z  \diff y \Big] -  \beta  \delta \frac {\Phi(q)} {\varphi(q)}.
  \end{align*}
Hence, when $h'(\infty) = \infty$, $I(\infty) = \infty$.

Suppose $h'(\infty) < \infty$. By Fubini's theorem,  integration by parts and Remark  \ref{remark_wiener_hopf} (1), we have \begin{align*}
&\int_{-\infty}^0  \int_0^\infty  \mathrm{e}^{- \varphi(q) z} \Theta^{(q)} (z-y) \diff z  \diff y =   \int_0^\infty  \mathrm{e}^{- \varphi(q) z} \int_{-\infty}^0 \Theta^{(q)} (z-y)  \diff y  \diff z \\
%=   \frac {\Phi(q)} q \int_0^\infty  \mathrm{e}^{- \varphi(q) z} \int_{-\infty}^0 \p \{ - \underline{X}_{e_q} \in \diff (-y) + z \}  \diff y  \diff z \\
&=   \frac {\Phi(q)} q \int_0^\infty  \mathrm{e}^{- \varphi(q) z}\p \{ - \underline{X}_{\EQ} > z \}  \diff z
={  \frac {\Phi(q)} {q } \Big[ \int_{(0,\infty)} \p\{ - \underline{X}_{\EQ}\in \diff z\}\int_0^z e^{-\varphi(q)w} \diff w  \Big] }\\
&= {\frac {\Phi(q)} {q } \Big[ \int_{(0,\infty)} \p\{ - \underline{X}_{\EQ}\in \diff z\}\frac{1-e^{-\varphi(q)z}}{\varphi(q)} \Big]}
=   \frac {\Phi(q)} {q \varphi(q)} \Big[1 - \int_{[0, \infty)}  \mathrm{e}^{- \varphi(q) z}\p \{ -\underline{X}_{\EQ} \in \diff z \}  \Big] \\
&= \frac {\Phi(q)} {q \varphi(q)} (1- \E  \mathrm{e}^{\varphi(q) \underline{X}_{\EQ}} )
%&=  \frac {\Phi(q)} {q \varphi(q)} (1-\E  \mathrm{e}^{\varphi(q) X_{\EQ}} /\E  \mathrm{e}^{\varphi(q) (X_{\EQ} - \underline{X}_{\EQ})}) \\
= \frac {\Phi(q)} {q\varphi(q)} -  \frac {\varphi(q)-\Phi(q)} {\varphi(q)^2 \delta},
\end{align*}
where the last equality holds by  \eqref{expressions_theta}.
%where the last equality holds because
%\begin{align*}
% \E  \mathrm{e}^{\varphi(q) X_{\EQ} } &= \int_0^\infty q   \mathrm{e}^{t (\psi(\varphi(q))-q)} = \frac q {\psi(\varphi(q))-q} = \frac q{\varphi(q) \delta}, \\  
% \E  \mathrm{e}^{\varphi(q) (X_{\EQ} - \underline{X}_{\EQ})} &= \E  \mathrm{e}^{-\varphi(q) \overline{X}_{\EQ}} = \frac 1 {\varphi(q)-\Phi(q)}.
%\end{align*}
Substituting this, we have the expression for $I(\infty)$. The proof for $I(-\infty)$ is similar.
\end{proof}

The following example is a direct consequence of Lemma \ref{lemma_classification}.
\begin{example} \label{example_linear}For the case $h(y) := \alpha y+\eta$, $y \in \R$, for some $\alpha, \eta \in \R$, we have $b^*=-\infty$ when $\alpha/q > \beta$,  $b^* = \infty$ when $\alpha/q > \beta$, and any value on $[-\infty, \infty]$ when $\alpha/q = \beta$.
\end{example}

Using the obtained $b^* \in [-\infty, \infty]$, our candidate optimal strategy is $\ell^{b^*}$ with its expected NPV given by $v_{b^*}$. For the case $b^* \in (-\infty, \infty)$, by our choice that $u_{b^*}(x) = 0$, we have
\begin{align}
v_{b^*}'(x) = \E_x \Big[ \int_0^\infty  \mathrm{e}^{-qt} h'(U_t^{b^*}) \diff t \Big], \quad x \in \R, \label{v_prime_equals_expectation}
\end{align}
where the differentiability at $b^*$ holds because the right hand side is continuous at $x=b^*$. 

 In view of  \eqref{v_infty_analytical}, the relation \eqref{v_prime_equals_expectation} also holds for $b^* = -\infty$ and $b^* = \infty$. To see how the derivative can go into the integral in \eqref{v_infty_analytical}, by the convexity of $h$, we can choose a sufficiently small $-m$ such that $h(x+z)$ is of the same sign for all $z < -m-x$. Hence, the convexity of $h$ implies by monotone convergence that
\begin{align*}
 \int_{-\infty}^{-m}  \frac {|h(x+z) - h(x+\epsilon + z)|} \epsilon \hat{u}^{(q)} (-z) \diff z \xrightarrow{\epsilon \rightarrow 0} \int_{-\infty}^{-m}  |h'(x+z)| \hat{u}^{(q)} (-z) \diff z;
 \end{align*}
 the same can be said when $\hat{u}^{(q)}$ is replaced with $\tilde{u}^{(q)}$.

\section{Verification of optimality} \label{section_optimality}
We shall now show the optimality of the strategy $\ell^{b^*}$ and  the associated expected NPV of total costs $v_{b^*}$. Toward this end, we shall first obtain a sufficient condition of optimality (Lemma \ref{verificationlemma} below) and then show that  $v_{b^*}$ satisfies it.
\subsection{Verification lemma}
For the given spectrally negative L\'evy process $X$, we call a function $f$ \emph{sufficiently smooth} if $f$ is continuously differentiable (resp.\ twice continuously differentiable) on $\mathbb{R}$ when $X$ has paths of bounded (resp.\ unbounded) variation.
We let $\Gamma$ be the operator acting on sufficiently smooth functions $f$, defined by
\begin{equation*}
\begin{split}
\Gamma f(x) := \gamma f'(x)+\frac{\sigma^2}{2}f''(x) +\int_{(-\infty,0)}[f(x+z)-f(x)-f'(x)z\mathbf{1}_{\{-1<z <0\}}]\nu(\mathrm{d}z),
\end{split}
\end{equation*}
and let $(\Gamma-q) f(x) := \Gamma f(x) - q f(x)$.
If, for some $\ell \in \Pi_\delta$,  $v_{\ell}$ is sufficiently smooth, then $\Gamma v_{\ell}$  is well-defined.  In addition, it is finite by \eqref{uniformly_integrability_A} and Assumptions \ref{assump_levy_measure} and \ref{assump_h}.
%Let us now give a verification lemma:
\begin{lemma}[Verification lemma]
\label{verificationlemma}
Suppose a strategy $\hat{{\ell}} \in \Pi_\delta$ is such that $v_{\hat{{\ell}}}$ is sufficiently smooth on $\mathbb{R}$ and satisfies\begin{equation}
\label{equiv_inequality}
\begin{cases}
(\Gamma-q) v_{\hat{{\ell}}}(x)+h(x) \geq 0 & \text{if $ v_{\hat{{\ell}}}'(x)\leq \beta$}, \\
(\Gamma-q) v_{\hat{{\ell}}}(x) -\delta (v_{\hat{{\ell}}}'(x) - \beta) +h(x)\geq 0 & \text{if $v_{\hat{{\ell}}}'(x)>\beta$}.
\end{cases}
\end{equation}
Then $\hat{{\ell}}$ is an optimal strategy and $v(x) = v_{\hat{{\ell}}}(x)$ for all $x \in \R$.
\end{lemma}
\begin{proof}
We first note that \eqref{equiv_inequality} is equivalent to the condition
\begin{equation}
\label{HJB-inequality}
\inf_{0\leq r\leq\delta} \{ (\Gamma -q) v_{\hat{{\ell}}}(x) -r (v'_{\hat{{\ell}}}(x)-\beta ) \} +h(x)\geq 0.
\end{equation}
Hence, we assume \eqref{HJB-inequality}. For brevity, we write $w:=v_{\hat{{\ell}}}$ throughout the proof.

Fix any admissible strategy ${\ell} \in \Pi_\delta$, and let $(T_n)_{n\in\mathbb{N}}$ be a sequence of stopping times defined by $T_n:=\inf\{t>0:|{U}^{\ell}_t|>n\}$.
Since ${U}^{\ell} =X - L$ is a semi-martingale and $w$ is sufficiently smooth by assumption, we can use the change of variables/It\^o's formula (cf.\ \cite{MR2273672}, Theorems II.31 and II.32) to the stopped process $\{\mathrm{e}^{-q(t\wedge T_n)}w({U}^{\ell}_{t\wedge T_n}); t \geq 0 \}$ and deduce under $\mathbb{P}_x$ that
\begin{multline*}
%\label{impulse_verif_1}
\mathrm{e}^{-q(t\wedge T_n)}w({U}^{\ell}_{t\wedge T_n})-w(x)
=  -\int_{0}^{t\wedge T_n}\mathrm{e}^{-qs} q w({U}^{\ell}_{s-}) \mathrm{d}s
+\int_{0}^{t\wedge T_n}\mathrm{e}^{-qs}w'({U}^{\ell}_{s-}) \mathrm{d}  ( X_s - L_s  )  \\
  + \frac  {\sigma^2} 2 \int_0^{t \wedge T_n}  \mathrm{e}^{-qs} w''(U_{s-}^{\ell}) \diff s + \sum_{0<s\leq t\wedge T_n}\mathrm{e}^{-qs}[\Delta w({U}^{\ell}_{s-}+\Delta X_s)-w'({U}^{\ell}_{s-})  \Delta X_s  ],
\end{multline*}
where $\Delta \zeta_s:=\zeta_s-\zeta_{s-}$ and $\Delta w(\zeta_s):=w(\zeta_s)-w(\zeta_{s-})$ for any right continuous process $\zeta$.
Rewriting the above equation leads to
\begin{align*}
\mathrm{e}^{-q(t\wedge T_n)}w({U}^{\ell}_{t\wedge T_n})  -w(x) = \int_{0}^{t\wedge T_n}\mathrm{e}^{-qs}   (\Gamma-q)w({U}^{\ell}_{s-})   \mathrm{d}s
 -\int_{0}^{t\wedge T_n}\mathrm{e}^{-qs}w'({U}^{\ell}_{s-})\mathrm{d}{L}_s + M_{t \wedge T_n}
  \end{align*}
 with
 \begin{align*}
 M_u &:= \int_0^u \sigma  \mathrm{e}^{-qs} w'(U_{s-}^{{\ell}}) \diff B_s +\lim_{\varepsilon\downarrow 0}\int_{(0,u]} \int_{(-1,-\varepsilon)}  \mathrm{e}^{-qs}w'(U_{s-}^{{\ell}})y (N(\diff s\times \diff y)-\nu(\diff y) \diff s)\\
 &+\int_{(0,u]} \int_{(-\infty,0)} \mathrm{e}^{-qs}(w(U_{s-}^{\ell}+y)-w(U_{s-}^{\ell})-w'(U_{s-}^{\ell})y\mathbf{1}_{\{y\in(-1,0)\}})(N(\diff s\times \diff y)-\nu(\diff y) \diff s), \quad u \geq 0,
 \end{align*}
  where $\left\{B_s; s \geq 0 \right\}$ is a standard Brownian motion and $N$ is a Poisson random measure in   the measure space  $([0,\infty)\times(-\infty,0),\B [0,\infty)\times \B (-\infty,0), \diff s \times\nu( \diff x))$.
%\begin{equation*}
%\begin{split}
%\mathrm{e}^{-q(t\wedge T_n)}w({U}^\pi_{t\wedge T_n}) & -w(x)
%\\
%= &  \int_{0+}^{t\wedge T_n}\mathrm{e}^{-qs}   (\Gamma-q)w({U}^\pi_{s-})   \mathrm{d}s
% -\int_{0+}^{t\wedge T_n}\mathrm{e}^{-qs}w'({U}^\pi_{s-})\mathrm{d}{L}^\pi_s   \\
%\end{split}
%\end{equation*}
%\begin{equation*}
%\begin{split}
% & + \Bigg\{ \int_{0+}^{t\wedge T_n}\mathrm{e}^{-qs}w'({U}^\pi_{s-})\mathrm{d}  \Big[ X_s-\Big( \gamma-\int_0^1 z\nu(\mathrm{d}z) \red{infinite}\Big) s-\sum_{0<u\leq s}\Delta X_u\mathbf{1}_{\{ |\Delta X_u| >1\}} \Big]   \Bigg\} \\
% & + \Bigg\{ \sum_{0<s\leq t\wedge T_n}\mathrm{e}^{-qs}[\Delta w({U}^\pi_{s-}+\Delta X_s)-w'({U}^\pi_{s-})\Delta X_s\mathbf{1}_{\{ |\Delta X_s| \leq1\}} ]\\
% & -   \int_{0+}^{t\wedge T_n}\int_{0+}^{\infty}\mathrm{e}^{-qs}\left[w({U}^\pi_{s-}+y)-w({U}^\pi_{s-})-w'({U}^\pi_{s-})y\mathbf{1}_{\{-1<y<0\}}\right]\nu(\mathrm{d}y)\mathrm{d}s \Bigg\}.
%\end{split}
%\end{equation*}
By the compensation formula (cf.\ \cite{Kyprianou_2006}, Corollary 4.6), $\{M_{t \wedge T_n}; t \geq 0 \}$ is a zero-mean martingale.

Now we write
\begin{multline*}
%\label{first_inequality}
w(x) =
 -\int_{0}^{t\wedge T_n}\mathrm{e}^{-qs}  \left[ (\Gamma-q)w({U}^{\ell}_{s-})-\ell_s (w'({U}^{\ell}_{s-})- \beta)\right]  \mathrm{d}s  \\
  + \beta \int_{0}^{t\wedge T_n}\mathrm{e}^{-qs} \ell_s\mathrm{d}s + \mathrm{e}^{-q(t\wedge T_n)}w({U}^{\ell}_{t\wedge T_n}) - M_{t \wedge T_n}.
\end{multline*}
%where $\{M_t:t\geq0\}$ is a zero-mean $\mathbb{P}_x$-martingale.
By \eqref{HJB-inequality}, after taking expectation, $w(x) \leq
 \E_x [ \int_{0}^{t\wedge T_n}\mathrm{e}^{-qs}(h({U}^{\ell}_{s})+\beta \ell_s)\mathrm{d}s ] +   \E_x [\mathrm{e}^{-q(t\wedge T_n)}w({U}^{\ell}_{t\wedge T_n}) ]$.
 
By \eqref{uniformly_integrability_A},  $\lim_{t, n \uparrow \infty}\E_x [ \int_{0}^{t\wedge T_n}\mathrm{e}^{-qs}h({U}^{\ell}_{s})\mathrm{d}s] = \E_x [ \int_{0}^{\infty}\mathrm{e}^{-qs}h({U}^{\ell}_{s})\mathrm{d}s]$ by dominated convergence.  On the other hand, because $\ell$ is nonnegative, we also have $\lim_{t, n \uparrow \infty} \E_x [ \int_{0}^{t\wedge T_n}\mathrm{e}^{-qs}\beta \ell_s\mathrm{d}s ] = \E_x [ \int_{0}^{\infty}\mathrm{e}^{-qs}\beta \ell_s\mathrm{d}s ]$ via monotone convergence.
Finally, by $0\leq \ell_s\leq\delta$, \eqref{h_upper}, \eqref{h_lower}
and the strong Markov property of $X$ and $Y$, $\E [\mathrm{e}^{-q(t\wedge T_n)}w({U}^{\ell}_{t\wedge T_n}) ] \leq \E [ \int_{t \wedge T_n}^\infty  \mathrm{e}^{-qs} (|h (X_s)|+ |h (Y_s)|+|\underline{h}| \mathbf{1}_{\{\underline{h} > -\infty\}} + \beta \delta ) \diff s] $.
Therefore, Fatou's lemma and \eqref{uniformly_integrability_A} give $\limsup_{t, n \uparrow \infty} \E_x [\mathrm{e}^{-q(t\wedge T_n)}w({U}^{\ell}_{t\wedge T_n}) ] \leq 0$.

Hence, $w(x) \leq v_{\ell} (x)$.  Since ${\ell} \in \Pi_\delta$ was chosen  arbitrarily, we have that $w(x) \leq v(x)$.
On the other hand, we also have the reverse inequality,  $w(x) \geq v(x)$, because $w$ is attained by an admissible strategy $\hat{{\ell}}$.  This completes the proof.

%Now taking expectations and
%letting $t$ and $n$ go to infinity and using the monotone convergence theorem we get, noting that $T_n\nearrow\infty$ $\mathbb{P}_x$-a.s., \red{[$h$ may not be nonnegative]}
%\begin{equation*}
%w(x) \leq \mathbb{E}_x \left[\int_{0}^{\infty}\mathrm{e}^{-qs}h({U}^\pi_{s})\mathrm{d}s+ \red{\beta}\int_{0}^{\infty}\mathrm{e}^{-qs}\ell^\pi(s)\mathrm{d}s \right] =v_\pi(x).
%\end{equation*}
%Hence we proved $w(x)\leq v_*(x)$ for all $x\in\mathbb{R}$.
\end{proof}

\subsection{Optimality of $v_{b^*}$}

In view of Lemmas \ref{verificationlemma}, it suffices to show that the function $v_{b^*}$ is sufficiently smooth and satisfies \eqref{equiv_inequality}. The former can be verified by  using the form of $v_{b^*}'$ as in \eqref{v_prime_equals_expectation} and considering its derivative; we defer the proof to Appendix \ref{proof_lemma_sufficiently_smooth}.
\begin{lemma} \label{lemma_sufficiently_smooth}The function $v_{b^*}$ is sufficiently smooth.
\end{lemma}

We shall now verify the inequalities \eqref{equiv_inequality}.  Toward this end, we restate it in the following lemma.
%The next lemma provides conditions to verify the optimality of the strategy with value function $v_{b^*}$.
\begin{lemma}
\label{tussenlemma} The inequalities \eqref{equiv_inequality} for $v_{\hat{{\ell}}} = v_{b^*}$ hold if and only if
\begin{equation}
\label{equiv_inequality2}
\begin{cases}
v_{b^*}'(x)\geq \beta & \text{if $x> b^*$}, \\
v_{b^*}'(x)\leq \beta & \text{if $x \leq b^*$}.
\end{cases}
\end{equation}
\end{lemma}
\begin{proof}[\textbf{Proof}]
(i) We shall first prove that the following equalities hold for $v_{b^*}$:
\begin{equation}
\label{generators}
\begin{split}
(\Gamma-q) v_{b^*}(x)+h(x)=0  & \quad \text{for $x\leq b^*$},\\
(\Gamma-q) v_{b^*}(x) -\delta (v_{b^*}'(x)-\beta)+h(x)  =0 & \quad  \text{for $x>b^*$}.
\end{split}
\end{equation}
%By (\ref{vf1}) we have that the value function is given for $x\leq b^*$ by:
%\begin{equation}
%v_{b^*}(x)=\frac{ \mathrm{e}^{\Phi(q)(x-b^*)}}{\Phi(q)}\left[h(b^*)W^{(q)}(0)+\red{\beta}+\int_{-\infty}^0h(y+b^*)W^{(q)\prime}(-y)dy\right]-\int_{-\infty}^xh(y)W^{(q)}(x-y)dy.\notag
%\end{equation}
%So if we denote by $\tau_{b^*}=\inf\{t>0:X_t>b^*\}$ we know that
%\[
% \mathrm{e}^{-q(\tau_{b^*}\wedge t)} \mathrm{e}^{-\Phi(q)X_{\tau_{b^*}\wedge t}}\qquad\text{$t\geq0$,}
%\]
%is a $\mathbb{P}_x$-martingale, this implies by Lemma 4 of \cite{Avram_et_al_2007} that
%\[
%(\Gamma-q) \mathrm{e}^{\Phi(q)(x-b^*)}=0.
%\]
%On the other hand by Lemma 4.5 of \cite{Egami-Yamazaki-2010-1},
%\[
%(\Gamma-q)\int_{-\infty}^xh(y)W^{(q)}(x-y)dy=h(x).
%\]
Here we give a proof of the second equality of \eqref{generators};  for the first equality a similar (and even simpler) argument holds.   We also focus on the case $b^* \in (-\infty, \infty)$; the proof for the cases $b^* =-\infty$ and $b^*=\infty$ can be obtained by modifying the hitting time $\kappa$ defined below.

Fix $x > b^*$ and recall \eqref{def_Y} for the definition of $Y$.   Define the stochastic process
\begin{equation*}
M_t:= \mathrm{e}^{-q(t\wedge\kappa)}v_{b^*}(Y_{t\wedge\kappa})+\int_0^{t\wedge\kappa} \mathrm{e}^{-qs}(h(Y_s)+\beta\delta) \diff s, \quad t\geq 0,
\end{equation*}
where  $\kappa :=\inf\{t>0:Y_t \notin [b^*, N] \}$ with $N > x$.
%Then it is clear that $\kappa_{b^*}^-$ coincides with the hitting time $\inf\{t>0:U_t^{b^*}<b^*\}$ $\p_x$-a.s. and $\{Y_t,0\leq t\leq \kappa_{b^*}^-\}$ equals $\{U^{b^*}_t,0\leq t\leq T_{b^*}^-\}$.

%Let $\mathbf{P}_x$ be the law of $Y$ under which $Y_0 = x$ and $\mathbf{E}_x$ the expectation operator.
In order to show the second equality of \eqref{generators}, it suffices to show that $M$ is a $\p_x$-martingale.   Indeed, it holds by
the martingale property and Ito's formula (thanks to Lemma \ref{lemma_sufficiently_smooth}, which guarantees that  $v_{b^*}$ is sufficiently smooth) noticing that the generator of $Y$ is given by $\Gamma_Y f := \Gamma f - \delta f'$.

%here $\tau_{b^*}^-$ stands for $\tau_{b^*}^-=\inf\{t>0:Y_t<b^*\}$ and $\mathbf{P}_x$ is the law of $Y$ when $Y_0 = x$.
For $x>{b^*}$ and $t > 0$, the strong Markov property of $Y$ gives
\begin{equation*}
\begin{split}
\E_x & \left.\left[ \mathrm{e}^{-q\kappa}v_{b^*}(Y_\kappa)+\int_0^\kappa  \mathrm{e}^{-qs}(h(Y_s)+ \beta \delta) \diff s \right| \mathcal{F}_t \right] \\
%= &  \mathbf{1}_{\{t<\kappa\}} \E_x   \left.\left[ \mathrm{e}^{-q\kappa}v_{b^*}(Y_\kappa)+\int_0^\kappa  \mathrm{e}^{-qs}(h(Y_s)+\beta \delta) \diff s \right| \mathcal{F}_t \right] + \mathbf{1}_{\{t\geq \kappa \}} M_t\\
= & \mathbf{1}_{\{t<\kappa\}}\left\{\mathrm{e}^{-qt}\E_{Y_t} \left[ \mathrm{e}^{-q\kappa}v_{b^*}(Y_{\kappa})+\int_0^{\kappa} \mathrm{e}^{-qs}(h(Y_s)+\beta\delta)\diff s \right]+\int_0^t \mathrm{e}^{-qs}(h(Y_s)+\beta\delta) \diff s \right\}
   + \mathbf{1}_{\{t\geq\kappa\}} M_t.
%\mathbf{1}_{\{t\geq\tau_{b^*}^-\}} \mathrm{e}^{-q\tau_{b^*}^-} \left( v_{b^*}(Y_{\tau_{b^*}^-})-\frac{\delta}{q}  \right)
\end{split}
\end{equation*}
%Here $\mathbf{E}_x$ denotes expectation with respect to $\mathbf{P}_x$.
Because $U_s^{b^*} = Y_s$ for all $0 \leq s \leq \kappa$, $\p_x$-a.s.,
\begin{equation*}
\begin{split}
\mathbf{1}_{\{t<\kappa \}} M_t %= & \mathbf{1}_{\{t<\kappa\}}\left\{\mathrm{e}^{-qt} \mathbb{E}_{Y_t} \left[\int_0^{\infty} \mathrm{e}^{-qs}(h(U_s^{b^*})+\beta \delta\mathbf{1}_{\{U^{b^*}_s>b^*\}}) \mathrm{d}s\right]+\int_0^t \mathrm{e}^{-qs}(h(Y_s)+\beta \delta)\diff s\right\} \\
%= & \mathbf{1}_{\{t<\kappa \}}\bigg\{\mathrm{e}^{-qt} \mathbb{E}_{Y_t} \left[ \int_{\kappa}^{\infty} \mathrm{e}^{-qs}(h(U_s^{b^*})+\beta\delta\mathbf{1}_{\{U^{b^*}_s>b^*\}}) \mathrm{d}s+\int_0^{\kappa}e^{-qs}(h(Y_s)+\beta\delta) \mathrm{d}s\right] \\
%& \qquad +\int_0^te^{-qs}(h(Y_s)+\beta\delta)\diff s\bigg\}\\
= & \mathbf{1}_{\{t<\kappa\}}\left\{\mathrm{e}^{-qt}\E_{Y_t} \left[ \mathrm{e}^{-q\kappa} v_{b^*}(Y_{\kappa})+\int_0^{\kappa} \mathrm{e}^{-qs}(h(Y_s)+\beta\delta) \mathrm{d}s\right]+\int_0^t \mathrm{e}^{-qs}(h(Y_s)+\beta \delta)\diff s\right\}.
\end{split}
\end{equation*}
Putting the pieces together, we conclude that  $M_t = \E_x \left.\left[ \mathrm{e}^{-q\kappa}v_{b^*}(Y_\kappa)+\int_0^\kappa  \mathrm{e}^{-qs}(h(Y_s)+ \beta \delta) \diff s \right| \mathcal{F}_t \right]$, $t \geq 0$,
which is a martingale.

(ii) Using the identities \eqref{generators}, we shall now complete the proof of this lemma.

Suppose first that  \eqref{equiv_inequality2} holds.  If $v_{b^*}'(x)<\beta$, then \eqref{equiv_inequality2} implies $x\leq b^*$ and so, by \eqref{generators}, $(\Gamma-q) v_{b^*}(x)+h(x)=0$. If $v_{b^*}'(x)\geq \beta$, then \eqref{equiv_inequality2} implies $x> b^*$ and so, by \eqref{generators}, $(\Gamma-q) v_{b^*}(x) -\delta (v_{b^*}'(x)- \beta) + h(x)=0$. If $v_{b^*}'(x)=\beta$, then we have, by \eqref{generators}, $(\Gamma-q) v_{b^*}(x)+h(x)=0$. Hence \eqref{equiv_inequality} holds with $v_{\hat{\ell}} = v_{b^*}$.

To conclude the proof, suppose now that \eqref{equiv_inequality} holds with $v_{\hat{\ell}} = v_{b^*}$. For the case $x\leq b^*$, suppose for contradiction that $v_{b^*}'(x)> \beta$. Then   \eqref{equiv_inequality} and \eqref{generators} give $-\delta (v_{b^*}'(x)- \beta) \geq0$ which implies $v_{b^*}'(x)\leq \beta$ and hence forms a contradiction. Therefore we deduce  $v_{b^*}'(x)\leq \beta$. Similarly, for the case $x>b^*$,  suppose $v_{b^*}'(x)< \beta$. Then   \eqref{equiv_inequality} and \eqref{generators} give $\delta (v_{b^*}'(x)- \beta )\geq0$ which implies $v_{b^*}'(x)\geq \beta$ and hence forms a contradiction. Therefore we deduce  $v_{b^*}'(x)\geq \beta$.

%\red{about the case $b^* =-\infty, \infty$}
\end{proof}

\begin{lemma} \label{lemma_convexity_v}
The function $v_{b^*}$  is convex.
\end{lemma}
\begin{proof}
By \eqref{v_prime_equals_expectation}, it is sufficient to show that the mapping $x \mapsto \E_x \left[ \int_0^\infty  \mathrm{e}^{-qt} h'(U_t^{b^*}) \diff t\right]$ is nondecreasing.

(i) Suppose that $-\infty < b^* < \infty$.    Notice that the law of $U_t^{b^*}$ under $\p_x$ is the same as $U_t^{b^*,x}$ under $\p$ where $U_t^{b^*,x}$ is the unique solution to the SDE: $U_t^{b^*,x} = x+ X_t - \delta \int_0^t \mathbf{1}_{\{U_s^{b^*,x} > b^* \}} \diff s$.
By the convexity of $h$  its derivative $h'$ is nondecreasing, and hence it is sufficient to show that, for any fixed $y > x$, $U_t^{b^*,y} \geq U_t^{b^*,x}$ $\p$-a.s. for all $t \geq 0$.  %Here, we show for the case $y > x > b^*$; by modifying the arguments slightly similar results hold for $b^* > y > x$. \red{should do for both.}

First, let us  suppose that $X$ is of bounded variation. We define a sequence of increasing stopping times $0 < \tau_1 < \tau_2 < \cdots$ representing the times at which either $U_t^{b^*,y}$ or $U_t^{b^*,x}$ crosses the level $b^*$.  For example, if $y > x > b^*$, then $\tau_1 = \inf \{t > 0: U_t^{b^*,x} < b^*\}$ and
$\tau_2 = \inf \{t > \tau_1: U_t^{b^*,y} < b^* \textrm{ or } U_t^{b^*,x} > b^*\}$.

By induction, we shall show that $U_t^{b^*, y} \geq U_t^{b^*, x}$ for any $t \in [\tau_{n-1}, \tau_{n}]$ for each $n \geq 0$ with $\tau_0 = \tau_{-1} := 0$ by convention. The base case ($n=0$) is already clear, because $U_{\tau_0}^{b^*, y} = y > x = U_{\tau_0}^{b^*, x}$.

%for every $t \in [0, \tau_1]$, $U_t^{b^*, y} = X_t + y - \delta t \geq X_t + x - \delta t = U_t^{b^*, x}$.

 Now suppose that the hypothesis holds for some $n \geq 0$; this implies $U_{\tau_n}^{b^*, y} \geq U_{\tau_n}^{b^*, x}$.
 Then, there are three possible scenarios (1) $U_{\tau_n}^{b^*, y} \geq U_{\tau_n}^{b^*, x} \geq b^*$, (2) $U_{\tau_n}^{b^*, y} \geq b^* \geq U_{\tau_n}^{b^*, x}$, and (3) $b^* \geq U_{\tau_n}^{b^*, y} \geq U_{\tau_n}^{b^*, x}$.

 For the case (1), it is necessary that  $U_t^{b^*, y} = U_{\tau_n}^{b^*, y} + (X_t - X_{\tau_n})  - \delta (t - \tau_n) \geq U_{\tau_n}^{b^*, x} + (X_t - X_{\tau_n})  - \delta (t - \tau_n) = U_t^{b^*, x}$ for $\tau_n \leq t \leq \tau_{n+1} = \inf \{ t > \tau_n: U_{t}^{b^*, x}  < b^* \}$.

 For the case (2),  we have $\tau_{n+1} := \overline{\tau}_{n+1}\wedge \underline{\tau}_{n+1}$,
where $\overline{\tau}_{n+1} := \inf \{t > \tau_n: U_t^{b^*,y} < b^*\}$ and $\underline{\tau}_{n+1} := \inf \{t > \tau_n: U_t^{b^*,x} > b^*\}$.

On the event $\{\overline{\tau}_{n+1} < \underline{\tau}_{n+1} \}$,  $U_{t}^{b^*,y} \geq b^* \geq U_{t}^{b^*,x}$ for all $t \in [\tau_n, \tau_{n+1})$ and $U_{\tau_{n+1}}^{b^*,y} = U_{\tau_{n+1}-}^{b^*,y} + \Delta X_{\tau_{n+1}} \geq  U_{\tau_{n+1}-}^{b^*,x} + \Delta X_{\tau_{n+1}} = U_{\tau_{n+1}}^{b^*,x}$.
On the event $\{\overline{\tau}_{n+1} > \underline{\tau}_{n+1} \}$, $U_{t}^{b^*,y} \geq b^* \geq U_{t}^{b^*,x}$ for all $t \in [\tau_{n}, \tau_{n+1}]$.

 For the case (3), we have  $U_t^{b^*, y} = U_{\tau_n}^{b^*, y} + (X_t - X_{\tau_n})   \geq U_{\tau_n}^{b^*, x} + (X_t - X_{\tau_n}) = U_t^{b^*, x}$ for $\tau_n \leq t \leq \tau_{n+1} = \inf \{ t > \tau_n: U_{\tau_n}^{b^*, y}  > b^* \}$.

 Hence for all cases $U_t^{b^*, y} \geq U_t^{b^*, x}$  for all $t \in (\tau_n, \tau_{n+1}]$. By mathematical induction, the inequality holds for all $t  \geq 0$.

 For the case of unbounded variation, recall, as in Lemma 12 of \cite{kyprianou2010refracted},  that there is a sequence of refracted \lev processes of bounded variation that converges a.s.\ to the desired  refracted \lev process.  Hence, we can obtain the same inequality $U_t^{b^*,y} \geq U_t^{b^*,x}$ $\p$-a.s. by taking the limit.

 (ii) Suppose $b^*=\infty$ or $b^*=-\infty$.
In view of \eqref{v_infty_by_expectation}, because $v_\infty(x) = \E [\int_0^\infty  \mathrm{e}^{-qt} h(X_t+x) \diff t ]$, the convexity of $h$ gives $v_\infty'(x) = \E [\int_0^\infty  \mathrm{e}^{-qt} h'(X_t+x) \diff t ] = \E_x [\int_0^\infty  \mathrm{e}^{-qt} h'(X_t) \diff t ]$. Similarly, $v_{-\infty}'(x) = \E _x [\int_0^\infty  \mathrm{e}^{-qt} h'(Y_t) \diff t ]$. The convexity of $h$ now shows that  $x \mapsto \E_x \left[ \int_0^\infty  \mathrm{e}^{-qt} h'(U_t^{b^*}) \diff t\right]$ is nondecreasing.
\end{proof}

\begin{proposition} \label{prop_second_condition}
The function $v_{b^*}$ satisfies \eqref{equiv_inequality}.
\end{proposition}
\begin{proof}
In view of Lemma \ref{tussenlemma}, we show that $v_{b^*}$ satisfies \eqref{equiv_inequality2}.
%Below, we first prove for the case  $-\infty < b^* < \infty$ and then for the case $b^* = -\infty$ and $b^* = \infty$.

(i) Suppose $-\infty < b^* < \infty$.   By setting $x = b^*$ in \eqref{v_prime_equals_expectation},
%By Proposition \ref{prop_derivative_all} and $I(b^*) = 0$, we have
%\begin{align*}
%v_{b^*}'(x) = \E_x \Big[ \int_0^\infty e^{-qt} h'(U_t^{b^*}) \diff t \Big], \quad x \in \R.
%\end{align*}
we have
\begin{align*}
v_{b^*}'(b^*)   &
= \int_0^\infty h'(y+b^*) \frac {\varphi(q) - \Phi(q)} {\delta \Phi(q)}  \mathrm{e}^{-\varphi(q) y} \diff y \\
&\qquad + \int_{-\infty}^0 h'(y+b^*) \Big\{ \frac {\varphi(q) - \Phi(q)} {\Phi(q)} \int_0^\infty  \mathrm{e}^{-\varphi(q) z} W^{(q)\prime} (z-y) \diff z -  W^{(q)}(-y)
\Big\} \diff y \\
&=  I(b^*) \frac  {\varphi(q)} {\delta \Phi(q)} + \beta = \beta.
\end{align*}
This together with Lemma \ref{lemma_convexity_v} shows \eqref{equiv_inequality2}.

(ii) Suppose $b^*=\infty$ or $b^*=-\infty$.
%\begin{align*}
%\E [\int_0^\infty e^{-qt} \frac {h(X_t+x+\varepsilon) - h(X_t+x)} \varepsilon  \diff t ] \geq \E [\int_0^\infty e^{-qt} h'(X_t+x)  \diff t ].
%\end{align*}
%we have
%\begin{align*}
%\liminf_{\varepsilon \rightarrow 0}\E [\int_0^\infty e^{-qt} \frac {h(X_t+x+\varepsilon) - h(X_t+x)} \varepsilon  \diff t ] \geq \E [\int_0^\infty e^{-qt} h'(X_t+x)  \diff t ].
%\end{align*}
%We can show that
%\begin{align*}
%v_\infty'(x) = \E_x \Big[\int_0^\infty e^{-qt} h'(X_t) \diff t \Big] \quad \textrm{and} \quad v_{-\infty}'(x) = \E _x\Big[\int_0^\infty e^{-qt} h'(Y_t) \diff t \Big].
%\end{align*}
For the case $b^*=\infty$, by the convexity of $h$, monotone convergence gives
$\lim_{x \uparrow \infty} v_\infty'(x) = \lim_{x \uparrow \infty}  \E  [\int_0^\infty  \mathrm{e}^{-qt} h'(X_t+x) \diff t ] = {h'(\infty)} /q \leq \beta,$ where the last inequality holds by Lemma \ref{lemma_classification}.
Similarly, for the case $b^* = -\infty$, $\lim_{x \downarrow -\infty} v_{-\infty}'(x) = \lim_{x \downarrow -\infty}  \E [\int_0^\infty  \mathrm{e}^{-qt} h'(Y_t+x) \diff t ] = {h'(-\infty)} / q \geq \beta.$
These bounds together with Lemma \ref{lemma_convexity_v} show \eqref{equiv_inequality2}.

\end{proof}

By Lemmas  \ref{lemma_sufficiently_smooth} and \ref{tussenlemma} and Proposition \ref{prop_second_condition},  we now have the main result of the paper.

\begin{theorem} The strategy ${\ell}^{b^*}$ is optimal and the value function is given by $v(x) = v_{b^*}(x)$ for all $x \in \R$.
\end{theorem}

We conclude this section with the remark on how the convexity of $h$ is important in deriving our results.
\begin{enumerate}
\item Note that the function $I$ as in  \eqref{def_I_new} is monotone due to the convexity of $h$. Otherwise, there may be multiple $b^*$ such that $I(b^*)$ vanishes; the form of the optimal strategy will be much more complicated.
\item  Thanks to \eqref{v_prime_equals_expectation} and the monotonicity of $h'$, $v'_{b^*}$ is also monotone. This step is necessary in showing the sufficient condition for optimality \eqref{equiv_inequality2}.
\item The convexity of $h$ also helps in technical details; indeed by the monotonicity of $h$, we can exchange limits over integrals in various parts.
\item
 The uniqueness of the optimal strategy  ${\ell}^{b^*}$ is not guaranteed in the complete set of admissible strategies $\Pi_{\delta}$, but it holds when we restrict to the set of refracted strategies when $h$ is not affine.
\end{enumerate}

\section{Convergence to Reflection Strategies}  \label{section_convergence}

Recall from Remark \ref{remark_dual} that
\begin{align} \label{v_tilde_delta}
\begin{split}
\tilde{v}(x; \delta) &:= \inf_{{\ell} \in \Pi_\delta}\E_x \Big[\int_0^\infty  \mathrm{e}^{-qt} (h (Y_t + L_t) + \tilde{\beta} \ell_t ) \diff t \Big] = v (x; \delta, - \tilde{\beta})  + \frac {\tilde{\beta} \delta} q,
\end{split}
\end{align}
where $v (x; \delta, - \tilde{\beta})$ is the value function \eqref{def_value_function} obtained in the previous sections with $X_t$ replaced with $X_t^{(\delta)} := Y_t + \delta t$ and $\beta$ with $-\tilde{\beta}$.
 In this section, we fix the process $Y$ (with its Laplace exponent $\psi_Y(\cdot)$, $\varphi(q) := \sup \{ \lambda \geq 0: \psi_Y(\lambda) = q\}$ and scale function $\mathbb{W}(\cdot)$) and then study the asymptotics as $\delta \uparrow \infty$.

When the absolutely continuous assumption on $L$ is relaxed, the problem defined in  \eqref{v_tilde_delta} becomes a classical singular control problem.  Let $\Pi_\infty$ be the set of admissible strategies consisting of all right-continuous, nondecreasing and adapted processes $L$ with $L_{0-} = 0$.  It is known as in Yamazaki \cite{Yamazaki_2013} under Assumption \ref{assump_f_g} defined below that, for all $x \in \R$,
\begin{align} \label{v_tilde_infty}
\begin{split}
\tilde{v}(x; \infty) &:= \inf_{{\ell} \in \Pi_\infty}\E_x \Big[ \int_{[0,\infty)}  \mathrm{e}^{-qt}  ( h (Y_t + L_{t}) \diff t  + \tilde{\beta} \diff L_{t} ) \Big]  \\
&=  -  \tilde{\beta} \left( \overline{\mathbb{Z}}^{(q)}(x-{b^*(\infty)}) + \frac {\psi_Y'(0+)} q  \right) - \int_{b^*(\infty)}^x \mathbb{W}^{(q)}(x-y) h(y) \diff y\\ &\qquad +
\mathbb{Z}^{(q)}(x-{b^*(\infty)}) \Big( \frac {\varphi(q)} q \int_0^\infty  \mathrm{e}^{-\varphi(q) y} h(y+b^*(\infty)) \diff y + \frac  {\tilde{\beta}} {\varphi(q)} \Big),
\end{split}
\end{align}
where $\mathbb{Z}^{(q)} (z):= 1 + q\int_0^z \mathbb{W}^{(q)}(y) \diff y$ and $\overline{\mathbb{Z}}^{(q)}(z) := \int_0^z \mathbb{Z}^{(q)} (y) \diff y$ for all $z \in \R$.
The infimum  is attained by the \emph{reflected \lev process} $Y_t + L^{b^*(\infty)}_t$ with
\begin{align*}
L^{b^*(\infty)}_t := \sup_{0 \leq t' \leq t} ((b^*(\infty))-Y_{t'}) \vee 0, \quad t \geq 0.
\end{align*}
The lower boundary $b^*(\infty)$ is defined as the unique root of $I_\infty(b) = 0$ where
\begin{align}
I_\infty(b) := \int_0^\infty  h'(y+b)   \mathrm{e}^{- \varphi(q) y}\diff y  +  \tilde{\beta} \frac {q} {\varphi(q)}, \quad b \in \R. \label{def_I_infty}
\end{align}
%\begin{align*}
%\frac {\Phi(q)} q \Psi(a_0;f) + \frac  {\tilde{\beta}} {\varphi(q)}  = \frac {f(a_0)} q.
%\end{align*}
%The value function is given by
%\begin{multline} \label{value_function_no_cost_case}
%\tilde{v}(x; \infty)
%%&= C \left[ - \overline{Z}^{(q)}(x-a_0) - \frac \mu q + \frac 1 {\Phi(q)} Z^{(q)} (x-a_0) \right] +
%%Z^{(q)}(x-a_0) \frac {\Phi(q)} q \Psi(a_0;f) - \varphi^x_{a_0} (f) \\
%=  -  \tilde{\beta} \left( \overline{Z}^{(q)}(x-{b^*(\infty)}) + \frac {\psi_Y'(0+)} q  \right) - \int_{b^*(\infty)}^x \mathbb{W}^{(q)}(x-y) h(y) \diff y\\ +
%Z^{(q)}(x-{b^*(\infty)}) \left( \frac {\varphi(q)} q \int_0^\infty  \mathrm{e}^{-\varphi(q)(q) y} h(y+b^*(\infty)) \diff y + \frac  {\tilde{\beta}} {\Phi(q)} \right).
%\end{multline}
Our objective in this section is to show the convergences of $b^*(\delta)$ to $b^*(\infty)$ and $\tilde{v}(x; \delta)$ to $\tilde{v}(x; \infty)$ as $\delta \uparrow \infty$.

Throughout, we assume that $Y$ is a spectrally negative \lev process and satisfies Assumption
\ref{assump_levy_measure} (2), which means that  there exists $\bar{\theta} > 0$ such that $\exp \psi_Y(-\bar{\theta}) = \E [\exp (-\bar{\theta} Y_1)] < \infty$. In addition, we further assume the condition postulated in Yamazaki \cite{Yamazaki_2013}:

\begin{assump}  \label{assump_f_g}
(1) For some $a \in \R$, the function $h(x) + \tilde{\beta}  q x$ is  decreasing on $(-\infty, a)$ and increasing on $(a,\infty)$.
(2) There exist a $c_0 > 0$ and an $x_0 \geq a$ such that $h'(x) + \tilde{\beta} q \geq c_0$ for a.e.\ $x \geq x_0$.
\end{assump}

We use the explicitly obtained expression of \eqref{v_tilde_delta}, for each $\delta > 0$, with
\begin{align*}
\tilde{\beta} = - \beta,
\end{align*}
and take the limit as $\delta \uparrow \infty$.
In order to emphasize the dependence on $\delta$, let us denote, for all $\delta >0$, $\psi_\delta(s) := \psi_Y(s) + \delta s$ as the Laplace exponent of the spectrally negative \lev process $X^{(\delta)}$.  The scale functions $W_\delta^{(q)}$ and $\Theta_\delta^{(\delta)}$ are defined in an obvious way.  Also, let $\Phi_\delta(q) := \sup \{ \lambda \geq 0: \psi_\delta(\lambda) = q\}$  and
\begin{multline} \label{def_I_delta_b}
I_\delta(b) :=  \frac {\varphi(q)-\Phi_\delta(q)} {\varphi(q)} \int_0^\infty  h'(y+b)   \mathrm{e}^{- \varphi(q) y}\diff y   + \delta \Big[ \int_{-\infty}^0 h'(y+b) \int_0^\infty  \mathrm{e}^{- \varphi(q) z} \Theta_\delta^{(q)} (z-y) \diff z  \diff y -  \beta \frac {\Phi_\delta(q)} {\varphi(q)} \Big].
\end{multline}
Assumption  \ref{assump_f_g} guarantees that $h'(-\infty) -\beta q < 0$ and $h'(\infty) -\beta q > 0$. Hence, by Lemma \ref{lemma_classification}, $I_\delta(\infty) > 0$ and $I_\delta(-\infty)  < 0$, implying that there  always exists a root $b^*(\delta) \in (-\infty, \infty)$.  As in the discussion given in Section \ref{subsection_candidate}, Assumption  \ref{assump_f_g} guarantees that the root $b^*(\delta)$ is unique.  We also let $U^{(\delta), b^*(\delta)}$ be the corresponding optimally controlled process. In other words,  it is a refracted \lev process defined as the solution to the SDE: $\diff U_t^{(\delta), b^*(\delta)}  = \diff X_t^{(\delta)} - \delta \mathbf{1}_{\{ U^{(\delta), b^*(\delta)} > b^*(\delta)\}} \diff t$, $t \geq 0$.

Note that
 we have $\Phi_\delta (q) \xrightarrow{\delta \uparrow \infty} 0$; because $q  = \psi_Y(\Phi_\delta(q)) + \delta \Phi_\delta(q)$ and $\psi_Y$ is continuous on $[0,\infty)$ and vanishes at zero,
\begin{align}
\delta \Phi_\delta(q) \xrightarrow{\delta \uparrow \infty} q.  \label{convergence_product_delta_Phi}
\end{align}
% \begin{remark} \label{convergence_scale_delta}
 By (8.20) of \cite{Kyprianou_2006}, for any $\theta > \varphi(q)$ (which is uniformly larger than $\Phi_\delta(q)$),
 %\begin{align*}
$\int_{[0, \infty)}   \mathrm{e}^{-\theta x} \delta W^{(q)}_\delta(\diff x) = {\delta \theta} / {(\psi_Y(\theta)+\delta \theta -q)} \xrightarrow{\delta \uparrow \infty} 1$.
%\end{align*}
Hence continuity theorem gives that
%$\delta W^{(q)}_\delta(\diff x) \xrightarrow{\delta \uparrow \infty} \delta_0(\diff x)$ and hence
\begin{align}
\delta W^{(q)}_\delta(x) = \delta  W^{(q)}_\delta[0,x]  \xrightarrow{\delta \uparrow \infty} 1, \quad x > 0. \label{convergence_scale_delta}
\end{align}
% Consequently, for all $y < 0$,
%\begin{align*}
% \lim_{\delta \rightarrow \infty} \int_0^\infty  \mathrm{e}^{- \varphi(q) z} \delta \Theta_\delta^{(q)} (z-y) \diff z
% =  \lim_{\delta \rightarrow \infty} \int_0^\infty  \mathrm{e}^{- \varphi(q) z} (\delta W^{(q)\prime}_\delta (z-y) - \Phi_\delta(q) \delta W^{(q)}_\delta (z-y) ) \diff z
%= \delta \int_{[0, \varepsilon)} (W^{(q)}_\delta (\diff x) - \Phi_\delta(q) W^{(q)}_\delta (x) \diff x)
%\xrightarrow{\delta \uparrow \infty} 0.
%\end{align*}
%This shows the claim.
%\end{remark}

\subsection{Convergence of optimal thresholds}  We first show the convergence of  $b^*(\delta)$ to   $b^*(\infty)$. Toward this end, we first show the following.

\begin{lemma} \label{confergence_I_b}
For all $b \in \R$, $I_\delta(b) \xrightarrow{\delta \uparrow \infty}  I_\infty(b)$. %\red{does it have to be uniform conv?}
\end{lemma}
\begin{proof}

It is clear that $\frac {\varphi(q)-\Phi_\delta(q)} {\varphi(q)} \int_0^\infty  h'(y+b)   \mathrm{e}^{- \varphi(q) y}\diff y \xrightarrow{\delta \uparrow \infty} \int_0^\infty  h'(y+b)   \mathrm{e}^{- \varphi(q) y}\diff y$.  Hence, it is left to show that
%\begin{align*}
$\int_{-\infty}^0 |h'(y+b)| \int_0^\infty  \mathrm{e}^{- \varphi(q) z} \delta \Theta_\delta^{(q)} (z-y) \diff z  \diff y \xrightarrow{\delta \uparrow \infty} 0$.
%\end{align*}
%Because $h'$ is nondecreasing, the supremum on the left hand side is bounded by
%\begin{align*}
%\sup_{b \in [\underline{b}, \overline{b}]}\int_{-\infty}^0 (|h'(y+\underline{b})|+|h'(y+\overline{b})|) \int_0^\infty  \mathrm{e}^{- \varphi(q) z} \delta \Theta_\delta^{(q)} (z-y) \diff z  \diff y.
%\end{align*}
%We shall show that this vanishes in the limit.

Recalling  Assumption
\ref{assump_levy_measure} (2), we choose $0 < \epsilon < \bar{\theta}$. For any $y < 0$,
\begin{align*}
\int_0^\infty  \mathrm{e}^{- \varphi(q) z} \delta \Theta_\delta^{(q)} (z-y) \diff z
% =   \mathrm{e}^{\epsilon y}\int_0^\infty  \mathrm{e}^{- (\varphi(q) +\epsilon) z}    \mathrm{e}^{\epsilon (z-y)} \delta\Theta_\delta^{(q)} (z-y) \diff z  \\
=  \mathrm{e}^{\epsilon y}\int_{-y}^\infty  \mathrm{e}^{- (\varphi(q) +\epsilon) (w+y)}   \mathrm{e}^{\epsilon w} \delta \Theta_\delta^{(q)} (w) \diff w   \leq    \mathrm{e}^{\epsilon y}\int_{-y}^\infty  \mathrm{e}^{\epsilon w} \delta \Theta_\delta^{(q)} (w) \diff w,\end{align*}
which is bounded from above by $\exp(\epsilon y)$ times a constant because (with $\underline{X}^{(\delta)}$ and $\underline{Y}$ the running infimum processes of $X^{(\delta)}$ and $Y$, respectively)
\begin{align*}
\int_{-y}^\infty   \mathrm{e}^{\epsilon w} \delta \Theta_\delta^{(q)} (w) \diff w \leq \delta \frac {\Phi_\delta(q)} q \E[ \mathrm{e}^{\epsilon |\underline{X}_{\EQ}^{(\delta)}|}] \leq  \delta \frac {\Phi_\delta(q)} q \E[ \mathrm{e}^{\epsilon |\underline{Y}_{\EQ}|}]
%=\frac {\delta (-\epsilon - \Phi_\delta(q))} {\psi_Y(-\epsilon) - \delta \epsilon -q},
\end{align*}
which is bounded in $\delta$ by  \eqref{finiteness_laplace_negagtive} (with $X$ replaced with $Y$) and   \eqref{convergence_product_delta_Phi}. This, together with $\int_{-\infty}^0 |h'(y+b)|   \mathrm{e}^{\epsilon y} \diff y < \infty$ by Assumption \ref{assump_h}, allows us to apply the dominated convergence theorem and
\begin{align*}
\lim_{\delta \rightarrow \infty}\int_{-\infty}^0 |h'(y+b)|  \int_0^\infty  \mathrm{e}^{- \varphi(q) z} \delta \Theta_\delta^{(q)} (z-y) \diff z  \diff y
%\\ &= \delta  \int_{-\infty}^0 |h'(y+b)|  \mathrm{e}^{\epsilon y}\int_0^\infty  \mathrm{e}^{- (\varphi(q) +\epsilon) z}   \mathrm{e}^{\epsilon (z-y)} \Theta_\delta^{(q)} (z-y) \diff z  \diff y \\
%&= \delta  \int_{-\infty}^0 |h'(y+b)|  \mathrm{e}^{\epsilon y}\int_{-y}^\infty  \mathrm{e}^{- (\varphi(q) +\epsilon) (w+y)}   \mathrm{e}^{\epsilon w} \Theta_\delta^{(q)} (w) \diff w  \diff y \\
= \int_{-\infty}^0 |h'(y+b)|  \lim_{\delta \rightarrow \infty} \int_{-y}^\infty  \mathrm{e}^{- \varphi(q) (z+y)} \delta \Theta_\delta^{(q)} (z) \diff z  \diff y. 
%&\leq \Big( \int_{-\infty}^0 |h'(y+b)|  \mathrm{e}^{\epsilon y} \diff y \Big) \Big(   \int_{0}^\infty   \mathrm{e}^{\epsilon w} \delta \Theta_\delta^{(q)} (w) \diff w \Big).
\end{align*}
Again, by Remark \ref{remark_wiener_hopf} (1) and \eqref{convergence_product_delta_Phi}, for any $\theta > 0$,
\begin{align*}
\E [ \mathrm{e}^{\theta \underline{X}^{(\delta)}_{\EQ}}]= \frac q {\delta \Phi_\delta(q)}\frac {\delta (\theta - \Phi_\delta(q))} {\psi_Y(\theta) + \delta \theta -q} \xrightarrow{\delta \uparrow \infty} 1.
\end{align*}
Hence $\underline{X}_{\EQ}^{(\delta)}$ converges to zero in distribution as $\delta \uparrow \infty$.
Using this and \eqref{convergence_product_delta_Phi}, for all $y < 0$,
\begin{align*}
 \int_{-y}^\infty  \mathrm{e}^{- \varphi(q) (z+y)} \delta \Theta_\delta^{(q)} (z) \diff z  \leq  \int_{-y}^\infty \delta \Theta_\delta^{(q)} (z) \diff z = \delta \frac {\Phi_\delta(q)} q \p \{ \underline{X}_{\EQ}^{(\delta)} < y\} \xrightarrow{\delta \uparrow \infty} 0.
\end{align*}
This shows the claim.

%By Wiener-Hopf factorization and Assumption
%\ref{assump_levy_measure} (2),
%\begin{align*}
% \int_{-y}^\infty   \mathrm{e}^{\epsilon w} \delta \Theta_\delta^{(q)} (w) \diff w = \delta \frac {\Phi_\delta(q)} q \ \mathrm{e}^\delta[e^{\epsilon \underline{X}_{\EQ}} 1_{\{\underline{X}_{\EQ} < y \}}] =\frac {\delta (\epsilon - \Phi_\delta(q))} {\psi_Y(\epsilon) + \delta \epsilon -q} -  \delta \frac {\Phi_\delta(q)} q \E^\delta[e^{\epsilon \underline{X}_{\EQ}} 1_{\{\underline{X}_{\EQ} > y \}}] \\
% \leq \frac {\delta (\epsilon - \Phi_\delta(q))} {\psi_Y(\epsilon) + \delta \epsilon -q} -  \delta \frac {\Phi_\delta(q)} q \E^\delta[e^{-\epsilon \varepsilon} 1_{\{\underline{X}_{\EQ} > -\varepsilon \}}]
%\end{align*}
%Here
%\begin{align*}
%\frac {\delta (\epsilon - \Phi_\delta(q))} {\psi(\epsilon) + \delta \epsilon -q}  \xrightarrow{\delta \uparrow \infty} 1.
%\end{align*}
%In addition, because

 \end{proof}

By Lemma \ref{confergence_I_b}, we now show the convergence of $b^*(\delta)$.
\begin{proposition} \label{convergence_b_star}We have $b^*(\delta) \xrightarrow{\delta \uparrow \infty} b^*(\infty)$.\end{proposition}
\begin{proof} Let $\underline{\delta} > 0$ be large enough such that $\Phi_\delta (q) \leq 1/2$ for all $\delta \geq \underline{\delta}$.

By differentiating \eqref{def_I_infty},
\begin{align}
I_\infty'(b) =  \int_0^\infty  h''(y+b)   \mathrm{e}^{- \varphi(q) y}\diff y + \sum_{y \geq 0}   \Delta h'(y+b)   \mathrm{e}^{- \varphi(q) y}, \quad b \in \R. \label{I_infty_derivative}
\end{align}
In addition, by \eqref{def_I_delta_b}  and because $h'$ is nondecreasing,  we have a lower bound:
\begin{align}
I_\delta'(b) \geq  \frac 1 2 I_\infty'(b), \quad b \in \R, \; \delta \geq \underline{\delta}. \label{I_delta_bound_half} \end{align}
Notice that $I_\infty(b)$  is strictly increasing on $(-\infty, \tilde{b})$ for some $\tilde{b} \in (-\infty, \infty]$. When $\tilde{b} < \infty$, we must have $I_\infty (\tilde{b}) > 0$; otherwise Assumption \ref{assump_f_g}
 (2) is violated.
 This means that we can choose some $\bar{b}  \in (b^*(\infty), \tilde{b})$ such that $I'_\delta (\bar{b}) > 0$. In addition, for all $b < \bar{b}$ we can obtain from \eqref{I_infty_derivative} a lower bound $I_\infty'(b) \geq e^{-\varphi(q) (\bar{b}-b)} I_\infty'(\bar{b})$.
%  Fix $\underline{b} < b^*(\infty)$, for any $\underline{b} < b < \bar{b}$, $\frac 1 2 I_\infty'(b) \geq e^{-\varphi(q) (\bar{b}-\underline{b})}\frac 1 2 I_\infty'(\bar{b})$.
% \begin{align*}
%\frac 1 2 I_\infty'(b) =  \frac 1 2 \Big[ e^{\varphi(q) b}\int_b^\infty  h''(y)   \mathrm{e}^{- \varphi(q) y}\diff y + \sum_{y \geq b}   (h'(y+)- h'(y-))   \mathrm{e}^{- \varphi(q) (y-b)} \Big] \\
%\geq  \frac 1 2 \Big[ e^{\varphi(q) \underline{b}}\int_{\bar{b}}^\infty  h''(y)   \mathrm{e}^{- \varphi(q) y}\diff y + \sum_{y \geq \bar{b}}   (h'(y+)- h'(y-))   \mathrm{e}^{- \varphi(q) (y-\underline{b})} \Big] \\
%= e^{-\varphi(q) (\bar{b}-\underline{b})}\frac 1 2 I_\infty'(b) \end{align*}
 Hence we can choose a sufficiently small $\bar{\varepsilon} > 0$ and $\alpha > 0$ such that $I_\infty'(b) \geq \alpha$ on $(b^*(\infty)- \bar{\varepsilon}, \, b^*(\infty)+\bar{\varepsilon})$; this together with  \eqref{I_delta_bound_half} gives
 \begin{align}
\inf_{\delta \geq \underline{\delta}, \, b \in (b^*(\infty)- \bar{\varepsilon}, b^*(\infty)+\bar{\varepsilon})} I_\delta'(b) \geq \frac \alpha 2 > 0. \label{I_lower_bound_inf}
\end{align}
 % uniformly on $\delta \geq \underline{\delta}$ and $b \in (b^*(\infty)- \bar{\varepsilon}, b^*(\infty)+\bar{\varepsilon})$.

Fix any $0 < \varepsilon < \bar{\varepsilon}$. By Lemma \ref{confergence_I_b}, we can choose a sufficiently large $\overline{\delta} > \underline{\delta}$ such that $|I_\infty(b^*(\infty))- I_\delta(b^*(\infty))| = |I_\delta(b^*(\infty))|< \varepsilon \alpha/2$  for any $\delta \geq \overline{\delta}$.  Then, by  \eqref{I_lower_bound_inf},  we must have for any $\delta \geq \overline{\delta}$ that $I_\delta(b^*(\infty)-\varepsilon) < I_\delta(b^*(\infty)) - \varepsilon \alpha/2 < 0$; similarly, $I_\delta(b^*(\infty)+\varepsilon) > I_\delta(b^*(\infty)) + \varepsilon \alpha/2 > 0$.  This means, together with the monotonicity of $I_\delta$, that $b^*(\infty) - \varepsilon < b^*(\delta) < b^*(\infty) + \varepsilon$.  Because $\varepsilon$ was chosen arbitrarily, this completes the proof.

%
%
%
%First show that $b^*(\delta)$ are bounded in a compact set for a sufficient large $\delta$.
%
%By triangular inequality and $I_\infty(b^*(\infty)) = I_\delta(b^*(\delta))=0$,
%\begin{align*}
%|I_\infty(b^*(\infty)) - I_\infty(b^*(\delta))| \leq |I_\delta(b^* (\delta)) - I_\infty(b^*(\delta))|.
%\end{align*}
%The boundedness of $b^*(\delta)$ and the uniform convergence in Lemma  \ref{confergence_I_b} give
%$\lim_{\delta \rightarrow \infty}|I_\infty(b^*(\infty)) - I_\infty(b^*(\delta))| =0$.
%
%Suppose for contradiction that the convergence does not hold; then there is $\epsilon  > 0$ and a subsequence such that $\lim_{\delta' \rightarrow \infty}b^*(\delta') > b^*(\infty) + \epsilon$ or $\lim_{\delta' \rightarrow \infty}b^*(\delta') < b^*(\infty) - \epsilon$.  The continuity and monotonicity of $I$ gives $I_\infty(b^*(\infty)) - I_\infty(b^*(\infty)+\epsilon) =0$ and $I_\infty(b^*(\infty)) - I_\infty(b^*(\infty)-\epsilon) =0$, respectively, which contradicts because $I_\infty$ is strictly increasing on $(-\infty, M)$.  Hence the claim holds by contradiction.
\end{proof}

\subsection{Convergence of value functions}
%
%In view of \eqref{v_tilde_delta},
%
%We have
%\begin{align*}
%\inf_{\pi }\E_x \Big[\int_0^\infty e^{-qt} (h (X_t^{(\delta)} - \tilde{L}_t^\pi) - \tilde{\beta} \tilde{\ell}_t ) \diff t \Big] = \E_x \Big[\int_0^\infty e^{-qt} (U_t^{(\delta), b^*(\delta)}) - \tilde{\beta} \delta  1_{\{ U_t^{(\delta), b^*(\delta)} > b^*(\delta)\}}) \diff t \Big],
%\end{align*}
%where the optimally controlled process $U_t^{(\delta), b^*(\delta)}$ is a refracted \lev process that moves like $Y$ above $b^*(\delta)$ and like $X_t^{(\delta)}$ below $b^*(\delta)$.

With the help of Proposition \ref{convergence_b_star}, we show the convergence of $\tilde{v}(x; \delta)$ to $\tilde{v}(x; \infty)$. More precisely, we show the following theorem.

\begin{theorem} \label{theorem_convergence}Uniformly in $x$ in compacts,
$\tilde{v}(x; \delta) \xrightarrow{\delta \uparrow \infty} \tilde{v}(x; \infty)$.
\end{theorem}
In order to show this theorem, it is sufficient to show the pointwise convergence. Indeed, by the definition of $\tilde{v}(x; \delta)$ in  \eqref{v_tilde_delta} as the infimum of the NPV and because the set $\Pi_\delta$ is monotonically nondecreasing in $\delta$,  $\tilde{v}(x; \delta)$ is nonincreasing in $\delta$. Hence, pointwise convergence implies the uniform convergence on compacts by Dini's theorem.

We shall first show for $x < b^*(\infty)$ and then extend the result for $x \geq b^*(\infty)$.
\begin{lemma} \label{lemma_convergence_below_b}For any $x < b^*(\infty)$, we have
$\tilde{v}(x; \delta) \xrightarrow{\delta \uparrow \infty} \tilde{v}(x; \infty)$.
\end{lemma}
\begin{proof}
%As discussed above, $\tilde{v}(x; \delta)$ is decreasing in $\delta$. In addition, they are uniformly bounded from below by $\tilde{v}(x; \infty)$ because $\Pi_\delta \subset \Pi_\infty$. In view of this, it is sufficient to show that, for any $\varepsilon > 0$,
%\begin{align}
%\lim_{\delta \uparrow \infty} \tilde{v}(x; \delta) \leq \tilde{v}(x; \infty) + \varepsilon. \label{convergence_epsilon_adjusted}
%\end{align}

By Proposition \ref{convergence_b_star}, we can choose sufficiently large $\underline{\delta}$ such that $x < b^*(\delta)$ for all $\delta \geq \underline{\delta}$. Let $h(-\infty) := \lim_{b \downarrow -\infty} h(b) \in [-\infty, \infty]$, which exists by the convexity of $h$.

Fix any $\varepsilon > 0$.

(i) Suppose $h(-\infty) = \infty$. By the convexity of $h$ (together with Assumption  \ref{assump_h}) and because \eqref{finiteness_laplace_negagtive} holds when $X$ is replaced with $Y$, we can choose a sufficiently small $-M$ such that $h$ is decreasing and positive on $(-\infty,-M]$ and
\begin{align*}
&\E_x \Big[\int_0^\infty  \mathrm{e}^{-qt} |h (\underline{Y}_t)| \mathbf{1}_{\{ Y_t < -M\}}\diff t \Big] = \E_x \Big[\int_0^\infty  \mathrm{e}^{-qt} h (\underline{Y}_t) \mathbf{1}_{\{ Y_t < -M\}}\diff t \Big]< \varepsilon.
\end{align*}
Because $\underline{Y}_t \leq U_t^{(\delta), b^*(\delta)}$ a.s. for any $\delta \geq \underline{\delta}$, \begin{align}
&\sup_{\delta \geq  \underline{\delta}}\E_x \Big[\int_0^\infty  \mathrm{e}^{-qt} |h (U_t^{(\delta), b^*(\delta)})| \mathbf{1}_{\{ U_t^{(\delta), b^*(\delta)} < -M\}}\diff t \Big] < \varepsilon. \label{unif_bound_below_M_U}
\end{align}
%uniformly in $\delta \geq 0$.

(ii) Suppose $h(-\infty) \in (-\infty, \infty)$, then $h$ is bounded on the negative half line and we can choose $-M$ such that \eqref{unif_bound_below_M_U} holds.

(iii) Suppose $h(-\infty) = -\infty$. Then, we can choose $-M$ such that $h$ is negative on $(-\infty, -M]$ and hence $\E_x [\int_0^\infty  \mathrm{e}^{-qt} h (U_t^{(\delta), b^*(\delta)}) \mathbf{1}_{\{ U_t^{(\delta), b^*(\delta)}< -M\}}\diff t ] \leq 0$.
In view of these, it is sufficient to show for all $-M \in \R$
\begin{align*}
\tilde{v}(x; \delta, -M) := \tilde{v}(x; \delta) - \E_x \Big[\int_0^\infty  \mathrm{e}^{-qt} h (U_t^{(\delta), b^*(\delta)}) \mathbf{1}_{\{ U_t^{(\delta), b^*(\delta)} < -M\}}\diff t \Big] \xrightarrow{\delta \uparrow \infty}  \tilde{v}(x; \infty).
\end{align*}
Indeed, this implies for the cases (i) and (ii) that $\lim_{\delta \uparrow \infty} |\tilde{v}(x; \delta) -\tilde{v}(x; \infty)| \leq \varepsilon$ for any $\varepsilon >0$. For the case (iii), it implies $\limsup_{\delta \uparrow \infty} \tilde{v}(x; \delta) \leq \tilde{v}(x; \infty)$ (and  $\liminf_{\delta \uparrow \infty} \tilde{v}(x; \delta) \geq \tilde{v}(x; \infty)$ holds because $\Pi_\delta \subset \Pi_\infty$ in view of \eqref{v_tilde_delta}).

%By Lemma \ref{convergence_b_star}, we can choose $\underline{\delta} > 0$ such that $b^*(\delta) > x$ for all $\delta \geq \underline{\delta}$.
Fix $\delta \geq \underline{\delta}$. We can write
\begin{align*}
&\tilde{v}(x; \delta, -M)
%&=\int_b^\infty (h(y)+ \beta \delta)  r_b(x, y) \diff y + \int_{-\infty}^b h(y) r_b(x,y) \diff y \\
=    \mathrm{e}^{\Phi_\delta(q) (x  - {b^*(\delta)})} \frac {\varphi(q)-\Phi_\delta(q)} {\delta \Phi_\delta(q)} \Big[  \int_0^\infty h(y+{b^*(\delta)})   \mathrm{e}^{- \varphi(q) y} \diff y +   \frac {\beta \delta}  {\varphi(q)} \Big]  - \frac {\beta \delta} q \\ &+ \int_{-M-b^*(\delta)}^0 h(y+{b^*(\delta)}) \Big[  \mathrm{e}^{\Phi_\delta(q)(x-{b^*(\delta)})} \frac {\varphi(q)-\Phi_\delta(q)} {\Phi_\delta(q)} \int_0^\infty  \mathrm{e}^{-\varphi(q)z} W^{(q)\prime}_\delta (z-y) \diff z  - W^{(q)}_\delta(x-{b^*(\delta)}-y) \Big] \diff y.
%v^{(2)}(x; \delta, -M)  &:=\int_{0}^\infty (h(y+b^*(\delta)) + \beta \delta) \Big\{   \mathrm{e}^{-\varphi(q)y}M(x; b^*(\delta)) - \mathbb{W}^{(q)}(x+b^*(\delta)-y) \Big\} \diff y \\
%&+ \delta \int_{-M-b^*(\delta)}^{0} h(y+b^*(\delta)) \Big\{ M(x; b^*(\delta)) \int_{0}^\infty e^{-\varphi(q) z} W^{(q)\prime}(z-y) \diff z \\ & \qquad - \int_{y+b^*(\delta)}^x \mathbb{W}^{(q)}(x-z) W^{(q)\prime} (z-y+b^*(\delta)-y) \diff z \Big\} \diff y.
\end{align*}
%In addition, because $\underline{X}$
%
%We choose $\epsilon$ to be $\int_{-\infty}^{-M} h(y) ...$ which decreases in $\delta$ where $-M$ is such that $h$ is decreasing and positive below $-M$.
%
%
%Let us consider
%\begin{align*}
%&\E_x \Big[\int_0^\infty e^{-qt} h (U_t^b ) 1_{\{ U_t^{b^*} < -M\}}\diff t \Big] \\
%&=  \int_{-\infty}^{-M-b^*(\delta)} h(y+{b^*(\delta)}) \Big[ e^{\Phi_\delta(q)(x-{b^*(\delta)})} \frac {\varphi(q)-\Phi_\delta(q)} {\Phi_\delta(q)} \int_0^\infty e^{-\varphi(q)z} W^{(q)\prime}_\delta (z-y) \diff z  - W^{(q)}_\delta(x-{b^*(\delta)}-y) \Big] \diff y.
%\end{align*}
%Define
%\begin{align*}
%v_{b^*(\delta)}^{(1)}(x; \delta)
%%&=\int_b^\infty (h(y)+ \beta \delta)  r_b(x, y) \diff y + \int_{-\infty}^b h(y) r_b(x,y) \diff y \\
%&:=   e^{\Phi_\delta(q) (x  - {b^*(\delta)})} \frac {\varphi(q)-\Phi_\delta(q)} {\delta \Phi_\delta(q)} \Big[  \int_0^\infty h(y+{b^*(\delta)})  e^{- \varphi(q) y} \diff y +   \frac {\beta \delta}  {\varphi(q)} \Big]  \\ &+ \int_{-\infty}^0 h(y+{b^*(\delta)}) \Big[ e^{\Phi_\delta(q)(x-{b^*(\delta)})} \frac {\varphi(q)-\Phi_\delta(q)} {\Phi_\delta(q)} \int_0^\infty e^{-\varphi(q)z} W^{(q)\prime}_\delta (z-y) \diff z \\ & \qquad - W^{(q)}_\delta(x-{b^*(\delta)}-y) \Big] \diff y - \frac {\beta \delta} q.
%\end{align*}
Regarding the first line on the right hand side, it equals
\begin{align*}
%& \mathrm{e}^{\Phi_\delta(q)(x-{b^*(\delta)})}\frac {\varphi(q)-\Phi_\delta(q)} {\delta \Phi_\delta(q)} \Big[  \int_0^\infty h(y+b^*(\delta))   \mathrm{e}^{- \varphi(q) y} \diff y +   \frac {\beta \delta}  {\varphi(q)} \Big] - \frac {\beta \delta} q \\
& \mathrm{e}^{\Phi_\delta(q)(x-{b^*(\delta)})} \Big[ \frac {\varphi(q)-\Phi_\delta(q)} {\delta \Phi_\delta(q)}  \int_0^\infty h(y+b^*(\delta))   \mathrm{e}^{- \varphi(q) y} \diff y +    \beta \Big( \frac 1 {\Phi_\delta(q)} - \frac {\delta} q - \frac 1 {\varphi(q)}  \Big) \Big]  - \frac {\beta \delta} q (1- \mathrm{e}^{\Phi_\delta(q)(x-{b^*(\delta)})})\\
%&=  \frac {\varphi(q)-\Phi_\delta(q)} {\delta \Phi_\delta(q)}  \int_0^\infty h(y+{b^*(\delta)})   \mathrm{e}^{- \varphi(q) y} \diff y + \beta \Big( \frac 1 {\Phi_\delta(q)} - \frac {\delta} q - \frac 1 {\varphi(q)} - \frac {\delta} q (1- \mathrm{e}^{\Phi_\delta(q)(x-{b^*(\delta)})}) \Big) \\
&\xrightarrow{\delta \uparrow \infty}  \tilde{v}(x; \infty)   = \beta \left( x-{b^*(\infty)} + \frac {\psi_Y'(0+)} q  \right) +
 \frac {\varphi(q)} q \int_0^\infty  \mathrm{e}^{-\varphi(q) y} h(y+b^*(\infty)) \diff y - \frac  {\beta} {\varphi(q)}.
\end{align*}
To see how the convergence holds, note, by \eqref{convergence_product_delta_Phi},
\begin{align}
 \mathrm{e}^{\Phi_\delta(q)(x-{b^*(\delta)})} \frac {\varphi(q)-\Phi_\delta(q)} {\delta \Phi_\delta(q)} \xrightarrow{\delta \uparrow \infty} \frac {\varphi(q)} {q}. \label{convergence_Phi_fraction}
\end{align}
In addition,
\begin{align*}
\frac 1 {\Phi_\delta(q)} - \frac {\delta} q = \frac {q-\delta \Phi_\delta(q)} {\Phi_\delta(q) q} =  \frac {\psi_Y(\Phi_\delta(q))} {\Phi_\delta(q) q} \xrightarrow{\delta \uparrow \infty} \frac {\psi_Y'(0+)} q
\end{align*}
and the Taylor expansion and the convergence of $b^* (\delta)$ gives
\begin{align*}
\delta (1-  \mathrm{e}^{\Phi_\delta(q)(x-b^*(\delta))}) = \delta [(x-b^*(\delta)) \Phi_\delta(q) + o(\Phi_\delta(q))] \xrightarrow{\delta \uparrow \infty} q (x-b^*(\infty)).
\end{align*}

On the other hand, $-M-b^*(\delta)$ converges and hence is bounded in $\delta$ from below by, say $-N < 0$; hence it is now left to show that
\begin{multline} \label{integral_from_minus_N_0}
 \int_{-N}^0 |h(y+b^*(\delta))| \Big[  \mathrm{e}^{\Phi_\delta(q)(x-b^*(\delta))} \frac {\varphi(q)-\Phi_\delta(q)} {\Phi_\delta(q)} \int_0^\infty  \mathrm{e}^{-\varphi(q)z} W_\delta^{(q)\prime} (z-y) \diff z 
 - W_\delta^{(q)}(x-b^*(\delta)-y) \Big] \diff y
\xrightarrow{\delta \uparrow \infty} 0.
\end{multline}
%\green{[As $\Big[  \mathrm{e}^{\Phi_\delta(q)(x-{b^*(\delta)})} \frac {\varphi(q)-\Phi_\delta(q)} {\Phi_\delta(q)} \int_0^\infty  \mathrm{e}^{-\varphi(q)z} W^{(q)\prime}_\delta (z-y) %\diff z  - W^{(q)}_\delta(x-{b^*(\delta)}-y) \Big]$ may fail to be positive when $x > b^*$ (not here, but the case $x > b^*$ is using this), hence I changed this way]}
Here, $\sup_{-N \leq y \leq 0, \delta \geq \underline{\delta}} |h'(y+{b^*(\delta)})| < \infty$ and
\begin{align}
\label{want_to_be_zero_1}
\begin{split}
0 &\leq \int_{-N}^0  \Big[  \mathrm{e}^{\Phi_\delta(q)(x-{b^*(\delta)})} \frac {\varphi(q)-\Phi_\delta(q)} {\Phi_\delta(q)} \int_0^\infty  \mathrm{e}^{-\varphi(q)z} W^{(q)\prime}_\delta (z-y) \diff z   \Big] \diff y \\
%&=    \mathrm{e}^{\Phi_\delta(q)(x-{b^*(\delta)})} \frac {\varphi(q)-\Phi_\delta(q)} {\Phi_\delta(q)} \int_0^\infty  \mathrm{e}^{-\varphi(q)z} \int_{-N}^0 W^{(q)\prime}_\delta (z-y) \diff y \diff z  \\
&=   \mathrm{e}^{\Phi_\delta(q)(x-{b^*(\delta)})} \frac {\varphi(q)-\Phi_\delta(q)} {\delta \Phi_\delta(q)}\delta  \int_0^\infty  \mathrm{e}^{-\varphi(q)z}  ( W^{(q)}_\delta (z+N)- W^{(q)}_\delta (z) )\diff z.
\end{split}
\end{align}
Here we recall the convergence  \eqref{convergence_Phi_fraction}. In addition,
\begin{align} \label{want_to_be_zero}
\begin{split}
&\delta  \int_0^\infty  \mathrm{e}^{-\varphi(q)z}  ( W^{(q)}_\delta (z+N)- W^{(q)}_\delta (z) )\diff z \\
% &= e^{\varphi(q) N}\int_{N}^\infty e^{-\varphi(q)z} \delta W^{(q)}_\delta (z) \diff z - \frac {\delta} {\psi_Y(\varphi(q)) + \delta \varphi(q) - q} \\
&= \delta  \mathrm{e}^{\varphi(q) N}\Big( \int_0^\infty  \mathrm{e}^{-\varphi(q)z} W^{(q)}_\delta (z) \diff z - \int^{N}_0  \mathrm{e}^{-\varphi(q)z} W^{(q)}_\delta (z) \diff z\Big)  - \delta  \int_0^\infty  \mathrm{e}^{-\varphi(q)z}  W^{(q)}_\delta (z) \diff z \\
&= ( \mathrm{e}^{\varphi(q) N}-1) \frac {\delta} {\psi_Y(\varphi(q)) + \delta \varphi(q) - q} -  \mathrm{e}^{\varphi(q) N} \int^{N}_0  \mathrm{e}^{-\varphi(q)z} \delta W^{(q)}_\delta (z) \diff z
\xrightarrow{\delta \uparrow \infty} 0.
\end{split}
\end{align}
To see how the convergence holds, $\delta \sup_{0 \leq z \leq N}W^{(q)}_\delta (z) = \delta W^{(q)}_\delta (N)$, which is bounded in $\delta$
by \eqref{convergence_scale_delta}.  Hence
dominated convergence and \eqref{convergence_scale_delta} give $\lim_{\delta \uparrow \infty}\int^{N}_0  \mathrm{e}^{-\varphi(q)z} \delta W^{(q)}_\delta (z) \diff z = \int^{N}_0  \mathrm{e}^{-\varphi(q)z} \diff z = (1 -  \mathrm{e}^{-\varphi(q) N}) / \varphi(q)$. On the other hand, $W^{(q)}_\delta(x-{b^*(\delta)}-y) \rightarrow 0$ uniformly on $y \in [-N, 0]$ by the convergence of $b^*(\delta)$ and \eqref{convergence_scale_delta}.
Hence \eqref{integral_from_minus_N_0} holds, as desired.

%
%This shows that
%\begin{align*}
%\tilde{v}(x; \delta, -M) \rightarrow   \frac {\varphi(q)} {q}  \int_0^\infty h(y+{b^*(\infty)})  e^{- \varphi(q) y} \diff y + \frac \beta q \Big( {\psi_Y'(0+)}  - \frac 1 {\varphi(q)} \Big) + \beta (x-{b^*(\infty)}).
%\end{align*}
%
%Regarding the first term of $v^{(2)}(x; \delta, -M)$,
%\begin{multline*}
%\int_{0}^\infty h(y+ b^*(\delta)) \Big\{  e^{-\varphi(q)y}M(x; b^*(\delta)) - \mathbb{W}^{(q)}(x- b^*(\delta)-y) \Big\} \diff y \\
%\xrightarrow{\delta \uparrow \infty}\int_{0}^\infty h(y+ b^*(\infty)) \Big\{  e^{-\varphi(q)y}\varphi(q)  \overline{\mathbb{W}}^{(q)}(x- b^*(\infty)) - \mathbb{W}^{(q)}(x- b^*(\infty)-y) \Big\} \diff y.
%\end{multline*}
%Here
%
%and
%\begin{align*}
%I_\infty(b) := \int_0^\infty  h'(y+b)  e^{- \varphi(q) y}\diff y   -  \beta \frac {q} {\varphi(q)} \\
%= h(b) + \int_0^\infty \varphi(q) h(y+b) e^{-\varphi(q)y} \diff y - \beta \frac {q} {\varphi(q)}
%\end{align*}
%First
%
%On the other hand,
%\begin{align*}
%v_b^{(2)}(x) &:=\int_{0}^\infty (h(y+b) + \beta \delta) \Big\{  e^{-\varphi(q)y}M(x; b) - \mathbb{W}^{(q)}(x-b-y) \Big\} \diff y \\
%&+ \delta \int_{-\infty}^{0} h(y+b) \Big\{ M(x; b) \int_{0}^\infty e^{-\varphi(q) z} W^{(q)\prime}(z-y) \diff z - \int_{b}^x \mathbb{W}^{(q)}(x-z) W^{(q)\prime} (z-b-y) \diff z \Big\} \diff y.
%\end{align*}

\end{proof}

In order to extend the result to the case $x \geq b^*(\infty)$, we shall show that the derivative converges.  By \eqref{v_prime_equals_expectation},
\begin{align}
\tilde{v}'(x; \delta) = \E_x \Big[ \int_0^\infty  \mathrm{e}^{-qt} h'(U_t^{(\delta), b^* (\delta)}) \diff t \Big], \quad x \in \R.  \label{v_b_derivative_for_delta}
\end{align}
On the other hand, by \cite{Yamazaki_2013}, the derivative of $\tilde{v}(x; \infty)$ simplifies and
\begin{align*}
\tilde{v}'(x; \infty)=\beta - \int_{b^*(\infty)}^x (h'(y)-\beta q)  \mathbb{W}^{(q)}(x-y) \diff y, \quad x \in \R.
\end{align*}
\begin{lemma} \label{lemma_convergence_derivative_v}
For every $x \in \R$, we have $\tilde{v}'(x; \delta) \xrightarrow{\delta \uparrow \infty} \tilde{v}'(x; \infty)$ for all $x \neq b^*(\infty)$.
\end{lemma}
\begin{proof}
Fix any $\underline{\delta} > 0$. We first show, for any $\varepsilon > 0$, that we can choose sufficiently small $-M \in \R$ such that
\begin{align}
\sup_{\delta > \underline{\delta}}\E_x \Big[\int_0^\infty  \mathrm{e}^{-qt} |h' (U_t^{(\delta), b^*(\delta)})| \mathbf{1}_{\{ U_t^{(\delta), b^*(\delta)} < -M\}}\diff t \Big] < \varepsilon. \label{bound_below_M_2}
\end{align}
(i) Suppose $h'(-\infty) < 0$. By the convexity of $h$ and \eqref{eq_integrability1}, we can choose a sufficiently small $-M$ such that, on $(-\infty,-M]$, $h'$ is negative (and hence $|h'|$ is decreasing) and $\E_x [\int_0^\infty  \mathrm{e}^{-qt} |h' (Y_t ) | \mathbf{1}_{\{ Y_t < -M\}}\diff t ] < \varepsilon$.
By $Y \leq U^{(\delta), b^*(\delta)}$ for all $\delta \geq \underline{\delta}$, we have \eqref{bound_below_M_2}.
(ii) Suppose $h'(-\infty) \geq 0$, then the convexity of $h$ implies that $h'$ is uniformly bounded by a constant on the negative half line.  Again, by $Y \leq U^{(\delta), b^*(\delta)}$, \eqref{bound_below_M_2} holds for sufficiently small $-M$.

For any $\varepsilon > 0$, with $-M$ such that  \eqref{bound_below_M_2} holds, the relation \eqref{v_b_derivative_for_delta} implies
\begin{align*}
&\sup_{\delta > \underline{\delta}}\Big| \tilde{v}'(x; \delta) - \E_x \Big[\int_0^\infty  \mathrm{e}^{-qt} h' (U_t^{(\delta), b^*(\delta)}) \mathbf{1}_{\{ U_t^{(\delta), b^*(\delta)} \geq -M\}}\diff t \Big] \Big| < \varepsilon.
\end{align*}
Hence, for the proof of this lemma, it is sufficient to show that for all small $-M \in \R$
\begin{align*}
\tilde{v}'(x; \delta, -M) := \E_x \Big[\int_0^\infty  \mathrm{e}^{-qt} h' (U_t^{(\delta), b^*(\delta)}) \mathbf{1}_{\{ U_t^{(\delta), b^*(\delta)} \geq -M\}}\diff t \Big] \xrightarrow{\delta \uparrow \infty} \tilde{v}'(x; \infty).
\end{align*}
% because $\varepsilon > 0$ can be chosen arbitrarily, this implies the convergence.

 By \eqref{r_expression} and  \eqref{v_b_derivative_for_delta}, we can write $\tilde{v}'(x; \delta, -M) = \tilde{v}^{(1)\prime}(x; \delta, -M) + \mathbf{1}_{\{x > b^*(\delta)\}}\tilde{v}^{(2)\prime}(x; \delta, -M)$
where
\begin{align*}
&\tilde{v}^{(1)\prime}(x; \delta, -M)
%&=\int_b^\infty (h(y)+ \beta \delta)  r_b(x, y) \diff y + \int_{-\infty}^b h(y) r_b(x,y) \diff y \\
:=    \mathrm{e}^{\Phi_\delta(q) (x  - {b^*(\delta)})} \frac {\varphi(q)-\Phi_\delta(q)} {\delta \Phi_\delta(q)} \int_0^\infty h'(y+{b^*(\delta)})   \mathrm{e}^{- \varphi(q) y} \diff y   \\ &+ \int_{-M-b^*(\delta)}^0 h'(y+{b^*(\delta)}) \Big[  \mathrm{e}^{\Phi_\delta(q)(x-{b^*(\delta)})} \frac {\varphi(q)-\Phi_\delta(q)} {\Phi_\delta(q)} \int_0^\infty  \mathrm{e}^{-\varphi(q)z} W^{(q)\prime}_\delta (z-y) \diff z  - W^{(q)}_\delta(x-{b^*(\delta)}-y) \Big] \diff y, \\
&\tilde{v}^{(2)\prime}(x; \delta, -M)  :=\int_{0}^\infty h'(y+b^*(\delta)) \Big\{   \mathrm{e}^{-\varphi(q)y}M_\delta(x; b^*(\delta)) - \mathbb{W}^{(q)}(x+b^*(\delta)-y) \Big\} \diff y \\
&+ \delta \int_{-M-b^*(\delta)}^{0} h'(y+b^*(\delta)) \Big\{ M_\delta(x; b^*(\delta)) \int_{0}^\infty  \mathrm{e}^{-\varphi(q) z} W_\delta^{(q)\prime}(z-y) \diff z \\ & \qquad - \int_{b^*(\delta)}^x \mathbb{W}^{(q)}(x-z) W_\delta^{(q)\prime} (z-y-b^*(\delta)) \diff z \Big\} \diff y,
\end{align*}
where $M_\delta(x;b^*(\delta)) :=  (\varphi(q)-\Phi_\delta(q))   \mathrm{e}^{-\Phi_\delta(q) b^* (\delta)} \int_{b^*(\delta)}^x  \mathrm{e}^{\Phi_\delta(q)z} \mathbb{W}^{(q)}(x-z) \diff z$.

Regarding $\tilde{v}^{(1)\prime}(x; \delta, -M)$, by the choice of $b^*(\infty)$ such that \eqref{def_I_infty} vanishes,
\begin{align*}
 \mathrm{e}^{\Phi_\delta(q) (x  - {b^*(\delta)})} \frac {\varphi(q)-\Phi_\delta(q)} {\delta \Phi_\delta(q)} \int_0^\infty h'(y+{b^*(\delta)})   \mathrm{e}^{- \varphi(q) y} \diff y \xrightarrow{\delta \uparrow \infty} \frac {\varphi(q)} {q} \int_0^\infty h'(y+{b^*(\infty)})   \mathrm{e}^{- \varphi(q) y} \diff y = \beta.
\end{align*}
The remaining term of $\tilde{v}^{(1)\prime}(x; \delta, -M)$ vanishes in the limit.  Indeed,  we can choose sufficiently small $-N < 0$ such that $-M - b^*(\delta) \geq -N$ uniformly in $\delta \geq \underline{\delta}$. Clearly, $\sup_{-N \leq y \leq 0, \delta > \underline{\delta}}|h'(y+{b^*(\delta)})| < \infty$, and we have shown that \eqref{want_to_be_zero_1} vanishes in the limit by \eqref{want_to_be_zero}. Hence,
\begin{align}
\tilde{v}^{(1)\prime}(x; \delta, -M) \xrightarrow{\delta \uparrow \infty} \beta. \label{convergence_v_1_prime}
\end{align}

It is now left to show the convergence $\tilde{v}^{(2)\prime}(x; \delta, -M) 1_{\{ x > b^*(\delta)\}} \xrightarrow{\delta \uparrow \infty}\tilde{v}^{(2)\prime}(x; \infty, -M) 1_{\{ x > b^*(\infty)\}}$.
 First, Proposition \ref{convergence_b_star} gives $1_{\{x > b^*(\delta)\}} \xrightarrow{\delta \uparrow \infty} 1_{\{x > b^*(\infty)\}}$ for $x \neq b^*(\infty)$. Hence, it remains  to show that $\tilde{v}^{(2)\prime}(x; \delta, -M)$ converges to $\tilde{v}^{(2)\prime}(x; \infty, -M)$ for $x > b^*(\infty)$.

 By Lemma  \ref{convergence_b_star}, we can set $\underline{\delta}$ large enough such that $x > b^*(\delta)$ for all $\delta \geq \underline{\delta}$.  For any such $\delta \geq \underline{\delta}$, we have the convergence
\begin{align*}
&\int_{0}^\infty h'(y+b^*(\delta)) \Big\{   \mathrm{e}^{-\varphi(q)y}M_\delta(x; b^*(\delta)) - \mathbb{W}^{(q)}(x-b^*(\delta)-y) \Big\} \diff y \\
&\xrightarrow{\delta \uparrow \infty}\int_{0}^\infty h'(y+b^*(\infty)) \Big\{   \mathrm{e}^{-\varphi(q)y}\varphi(q)  \int_0^{x-b^*(\infty)}\mathbb{W}^{(q)}(z) \diff z - \mathbb{W}^{(q)}(x-b^*(\infty)-y) \Big\} \diff y \\
&=-  \int_{b^*(\infty)}^x (h'(y)-\beta q)  \mathbb{W}^{(q)}(x-y) \diff y,
\end{align*}
where we used the choice of $b^*(\infty)$ that makes \eqref{def_I_infty} zero.
We shall now show that the remaining term vanishes in the limit and hence
\begin{align}
\tilde{v}^{(2)\prime}(x; \delta) \xrightarrow{\delta \uparrow \infty} - \int_{b^*(\infty)}^x (h'(y)-\beta q)  \mathbb{W}^{(q)}(x-y) \diff y.  \label{convergence_v_2_prime}
\end{align}
For any $\delta \geq \underline{\delta}$, Fubini's theorem gives
\begin{align*}
&\delta \int_{-M-b^*(\delta)}^{0} \Big\{ M_\delta(x; b^*(\delta)) \int_{0}^\infty  \mathrm{e}^{-\varphi(q) z} W_\delta^{(q)\prime}(z-y) \diff z  - \int_{b^*(\delta)}^x \mathbb{W}^{(q)}(x-z) W_\delta^{(q)\prime} (z-b^*(\delta)-y) \diff z \Big\} \diff y \\
&= \delta  \Big\{ M(x; b^*(\delta)) \int_{0}^\infty  \mathrm{e}^{-\varphi(q) z} (W_\delta^{(q)}(z+M + b^*(\delta)) - W_{\delta}^{(q)}(z))\diff z \\ & \qquad - \int_{b^*(\delta)}^x \mathbb{W}^{(q)}(x-z) (W_\delta^{(q)} (z+M) - W_\delta^{(q)} (z-b^*(\delta))) \diff z \Big\}.
%&\leq \delta  \Big\{ M_\delta(x; b^*(\delta)) \int_{0}^\infty  \mathrm{e}^{-\varphi(q) z} (W_\delta^{(q)}(z+M + b^*(\delta)) - W_\delta^{(q)}(z))\diff z \\ & \qquad + \mathbb{W}^{(q)}(x-b^*(\delta)) \int_{b^*(\delta)}^x  (W_\delta^{(q)} (z+M) - W_\delta^{(q)} (z-b^*(\delta))) \diff z \Big\} \diff y.
\end{align*}
Here $\delta \int_{0}^\infty  \mathrm{e}^{-\varphi(q) z} (W_\delta^{(q)}(z+M + b^*(\delta)) - W_{\delta}^{(q)}(z))\diff z \leq \delta \int_{0}^\infty  \mathrm{e}^{-\varphi(q) z} (W_\delta^{(q)}(z+M + x) - W^{(q)}_{\delta}(z))\diff z$,
which vanishes as $\delta \uparrow \infty$ by \eqref{want_to_be_zero}.  On the other hand,  with $\underline{b}^*:= \inf_{\delta \geq \underline{\delta}} b^*(\delta) > -\infty$,
we have $\sup_{b^*(\delta) \leq z \leq x}\mathbb{W}^{(q)}(x-z) = \mathbb{W}^{(q)}(x-b^*(\delta)) \leq \mathbb{W}^{(q)}(x-\underline{b}^*) < \infty$, and
\begin{align*}
0 \leq \delta \int_{b^*(\delta)}^x  (W_\delta^{(q)} (z+M) - W_\delta^{(q)} (z-b^*(\delta))) \diff z =   \delta \int_{0}^{x-b^*(\delta)}  (W_\delta^{(q)} (z+M + b^*(\delta)) - W_\delta^{(q)} (z)) \diff z  \\ \leq \delta \int_{0}^{x- \underline{b}^*}  (W_\delta^{(q)} (z+M +x) - W_\delta^{(q)} (z)) \diff z,
\end{align*}
which vanishes in the limit as $\delta \uparrow \infty$ by dominated convergence and \eqref{convergence_scale_delta}. Hence, \eqref{convergence_v_2_prime} holds. 
%Combining  \eqref{convergence_v_1_prime} and  \eqref{convergence_v_2_prime}, we have the claim.
\end{proof}

\begin{proof}[Proof of Theorem \ref{theorem_convergence}]
%In order to complete the proof of this theorem, it is sufficient to show that the pointwise convergence also holds for $x \geq b^*(\infty)$. This and Lemma \ref{lemma_convergence_below_b} show the pointwise convergence on $\R$; because $\tilde{v}(x; \delta)$ is monotone in $\delta$, Dini's theorem then completes the proof.

%By the definition of $\tilde{v}(x; \delta)$ as an infimum of the NPV and because the set $\Pi_\delta$ is monotonically increasing in $\delta$, $\tilde{v}(x; \delta)$ is monotonically decreasing.

In order to complete the proof of Theorem \ref{theorem_convergence}, it is left to show that $\tilde{v}(x; \delta) \xrightarrow{\delta \uparrow \infty} \tilde{v}(x; \infty)$ for any $x \geq b^*(\infty)$.

Fix any $\underline{x} < b^*(\infty)$.  We have $\tilde{v}(x; \delta) = \tilde{v}(\underline{x}; \delta) + \int_{\underline{x}}^x \tilde{v}'(z; \delta) \diff z$ for all $x \geq \underline{x}$.
% = \tilde{v}(b^*(\delta); \delta) + \int_{b^*(\infty)}^x \tilde{v}(z; \delta) \diff z + \int^{b^*(\delta)}_{b^*(\infty)} \tilde{v}(z; \delta) \diff z.
By Lemma \ref{lemma_convergence_below_b}, we have $\tilde{v}(\underline{x}; \delta) \xrightarrow{\delta \uparrow \infty} \tilde{v}(\underline{x}; \infty)$.  In addition, for all $\underline{x} \leq z \leq x$, by the convexity of $\tilde{v}$ as in Lemma  \ref{lemma_convexity_v}, $\tilde{v}'(z; \delta) \leq \tilde{v}'(x; \delta)$,
which is bounded in $\delta$ by Lemma \ref{lemma_convergence_derivative_v}. Hence, Fatou's lemma gives
$\limsup_{\delta \uparrow \infty}\int_{\underline{x}}^x \tilde{v}'(z; \delta) \diff z \leq \int_{\underline{x}}^x \tilde{v}'(z; \infty) \diff z = \tilde{v}(x; \infty) - \tilde{v}(\underline{x}; \infty)$.
This shows $\limsup_{\delta \uparrow \infty}\tilde{v}(x; \delta) \leq \tilde{v}(x; \infty)$.
Because $\tilde{v}(x; \delta) \geq \tilde{v}(x; \infty)$, we have $\tilde{v}(x; \delta) \xrightarrow{\delta \uparrow \infty} \tilde{v}(x; \infty)$ for any $x \geq b^*(\infty)$, as desired.

%  because $v(z; \delta)$ is monotonically decreasing in $\delta$  and is bounded from above by $v(z; \underline{\delta})$ for any $\delta \geq \underline{\delta}$, monotone convergence and Lemma \ref{lemma_convergence_derivative_v} give
%\begin{align*}
%\lim_{\delta \uparrow \infty}\int_{\underline{x}}^x \tilde{v}'(z; \delta) \diff z = \int_{\underline{x}}^x \tilde{v}'(z; \infty) \diff z = \tilde{v}(x; \infty) - \tilde{v}(\underline{x}; \infty),
%\end{align*}
%showing the convergence $\tilde{v}(x; \delta) \xrightarrow{\delta \uparrow \infty} \tilde{v}(x; \infty)$ for any $x \geq b^*(\infty)$, as desired.
\end{proof}

\section{Examples under phase-type \lev processes} \label{section_numerics}

In this section, we focus on the case the processes $X$ (and $Y$) have i.i.d.\ phase-type distributed jumps \cite{Asmussen_2004} of the form
$X_t  - X_0= \tilde{\gamma} t+\sigma B_t - \sum_{n=1}^{N_t} Z_n$, $0\le t <\infty$, %\label{X_phase_type}
%\end{equation}
for some $\tilde{\gamma} \in \R$ and $\sigma \geq 0$.  Here $B=\{B_t; t\ge 0\}$ is a standard Brownian motion, $N=\{N_t; t\ge 0\}$ is a Poisson process with arrival rate $\kappa$, and  $Z = \left\{ Z_n; n = 1,2,\ldots \right\}$ is an i.i.d.\ sequence of phase-type-distributed random variables with representation $(m,{\bm \alpha},{\bm T})$; see \cite{Asmussen_2004}. These processes are assumed mutually independent.  The class of processes of this type is important because it can approximate any spectrally negative \lev process in law (see \cite{Asmussen_2004} and \cite{Egami_Yamazaki_2010_2}). 

The Laplace exponent \eqref{laplace_spectrally_positive} of $X$ is then a rational function and hence the scale function (and hence their resolvent measures \eqref{resolvent_measure} as well) admit analytical expressions; straightforward computation yields the optimal threshold level $b^*$ as well as the value function $v_{b^*}$\footnote{Analytical expressions of these can be found in a supplementary note given in https://sites.google.com/site/kyamazak.}.  

As numerical illustrations, we focus on the quadratic case $h(z) = z^2$, $z \in \R$.
We numerically verify the optimality and then study the behavior with respect to $\delta$.  Throughout, let us fix $q = 0.05$. For the processes $X$ and $Y$, we let  $\sigma = 0.2$ and $\kappa = 1$ and, for the jump size distribution, we use the phase-type distribution given by $m=6$ 
that approximates the Weibull random variable with parameter  $(2,1)$ (see \cite{Egami_Yamazaki_2010_2} for the values of ${\bm T}$ and $\bm \alpha$).

%
% \begin{align*}
%&{\bm T} = \left[ \begin{array}{rrrrrr}   -5.6546  &  0.0000   &      0.0000  &  0.0000  &  0.0000  &  0.0000 \\
%    0.6066  & -5.6847 &   0.0000  &  0.0166  &  0.0089  &  5.0526 \\
%    0.2156  &  4.3616  & -5.6485  &  0.9162 &   0.1424 &   0.0126 \\
%    5.6247 &   0.0000   & 0.0000  & -5.6786  &  0.0000  &  0.0000 \\
%    0.0107  &  0.0000  &  0.0000 &   5.7247 &  -5.7420   & 0.0000 \\
%    0.0136  &  0.0000 &   0.0000  &  0.0024 &   5.7022 &  -5.7183 \end{array} \right],  \quad {\bm \alpha} =    \left[ \begin{array}{l}
%    0.0000 \\
%    0.0007 \\
%    0.9961 \\
%    0.0000 \\
%    0.0001 \\
%    0.0031  \end{array} \right];
%\end{align*}
%this gives an approximation to the Weibull random variable with parameter  $(2,1)$ (see \cite{Egami_Yamazaki_2010_2} for the accuracy of the approximation).

%We first illustrate the computation of the optimal threshold level $b^*$ and the associated value function.  Here 
We fix $\tilde{\gamma} = 5.5$, $\delta = 5$ (and hence $\tilde{\gamma}_Y = 0.5$).  As it has been shown in Section \ref{subsection_candidate}, the function $I$ is increasing. Hence, in general the bisection method can be applied to identify the value of $b^*$.  
%For our choice of $h$ as a quadratic function, as is clear from Lemma \ref{lemma_phase_type} (or more directly from Example \ref{example_quadratic}), $I$ is an affine function of $b$ and hence the value of $b^*$ can be computed directly.    
Using the obtained $b^*$, we compute the value function $v_{b^*}(x) = v_{b^*}^{(1)}(x) + v_{b^*}^{(2)}(x) 1_{\{x > b^*\}}$.

In order to confirm the optimality of the refraction strategy, we plot on the top of Figure \ref{figure_caption_optimality}, for the cases $\beta = -5$ and $\beta =5$, the value function  $v_{b^*}(x)$ along with suboptimal NPV's $v_{b}(x)$ for $b \in \{b^*-1, b^*-0.5, b^*+0.5, b^*+1\}$.  We observe that $v_{b^*}$ is indeed minimal (uniformly in $x$) and $b^*$ is the sole choice that makes the value function smooth and convex.  In addition, the slope at $b^*$ of $v_{b^*}(x)$ equals $\beta$. This is consistent with our observation that $v_{b^*}(x)$ satisfies \eqref{equiv_inequality2}.

%Figure \ref{figure_caption_optimality} shows
%
%
%In the left panel of Figure \ref{figure_caption_optimality}, we show the  function $I(b)$.  As it has been shown in Section \ref{subsection_candidate}, it is indeed strictly increasing.  Hence, the bisection method can be used to obtain $b^*$ (or using directly the expression in Example \ref{example_quadratic} for the quadratic case).

\begin{figure}[htbp]
\begin{center}
\begin{minipage}{1.0\textwidth}
\centering
\begin{tabular}{cc}
 \includegraphics[scale=0.37]{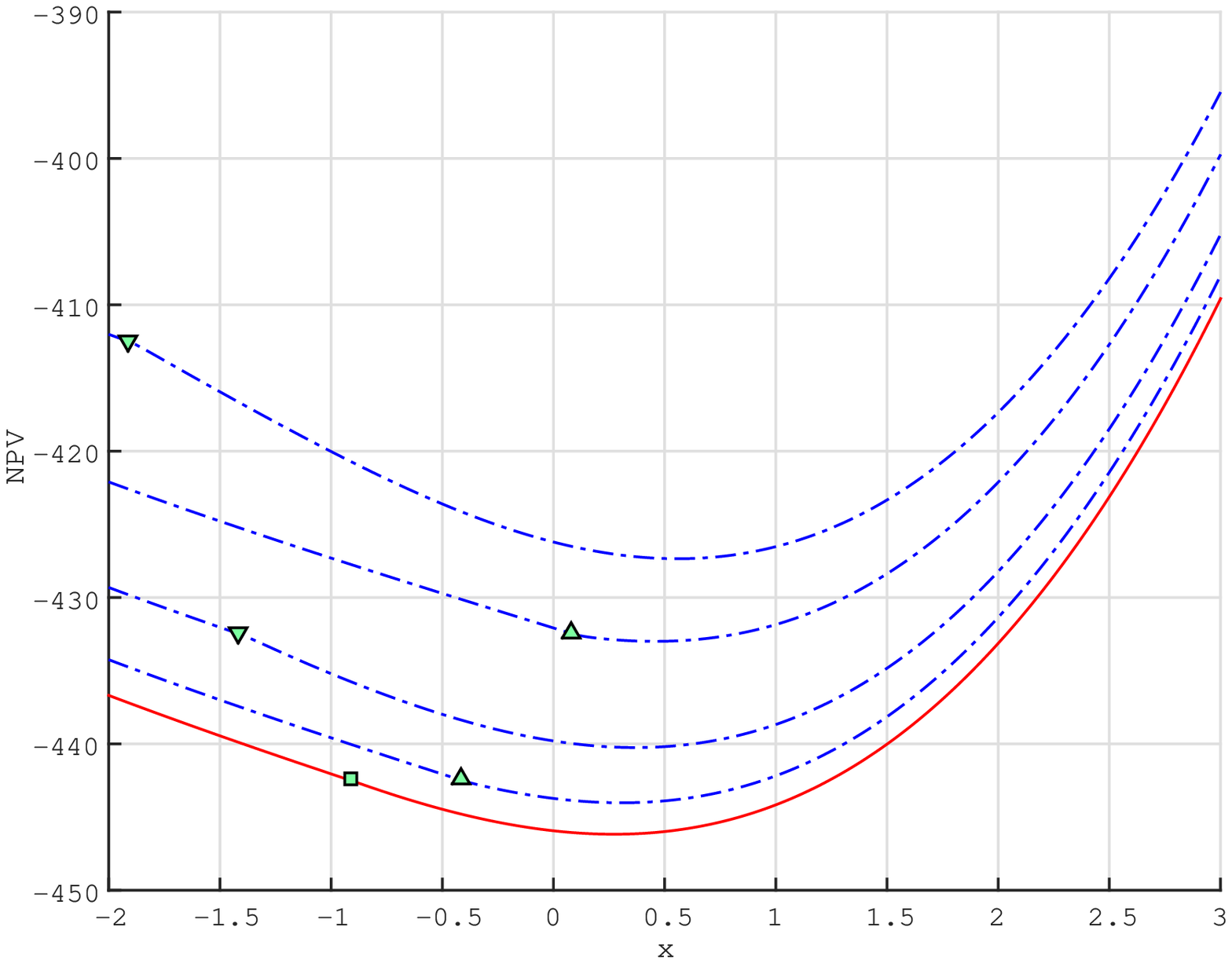} & \includegraphics[scale=0.37]{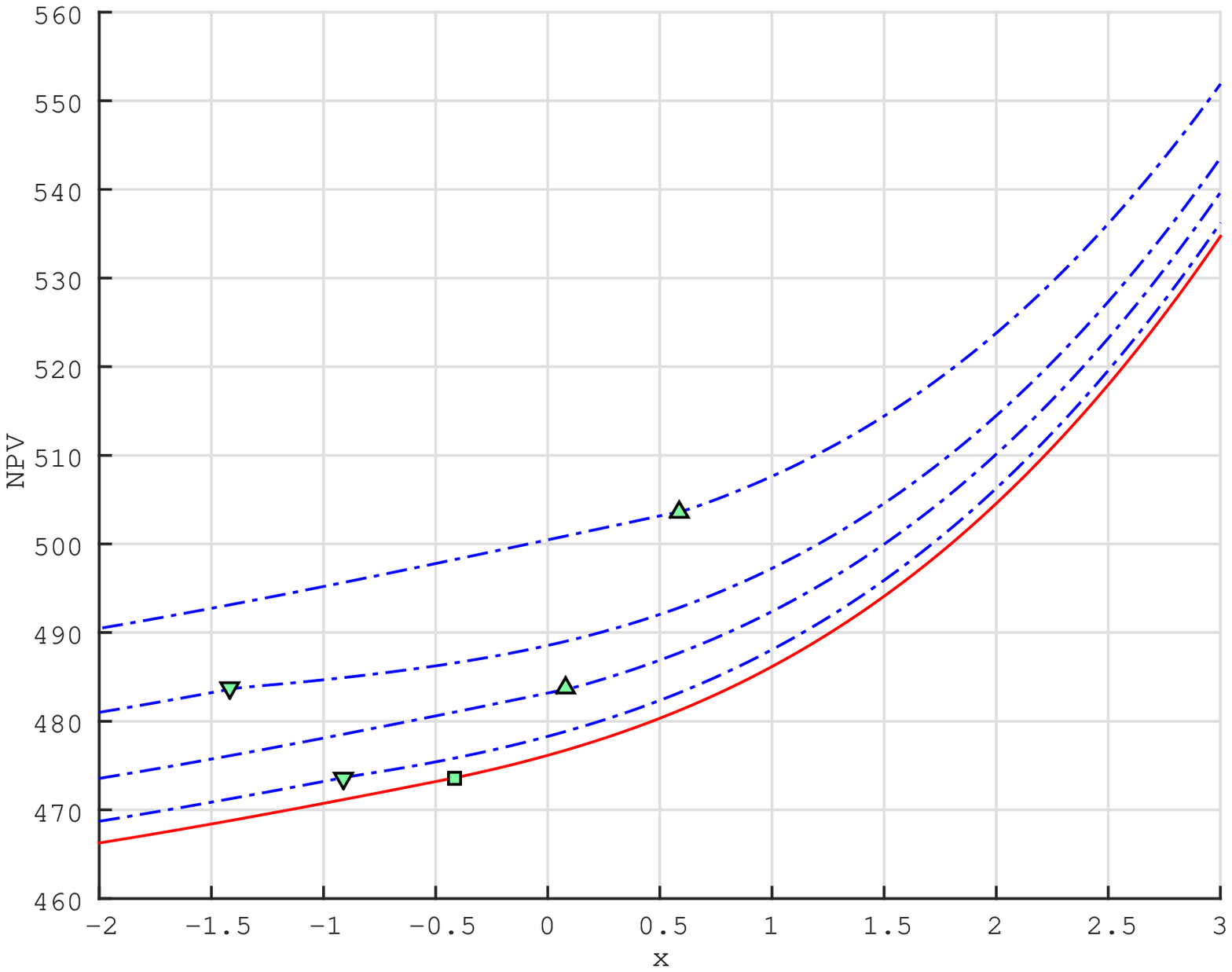}  \\
  \includegraphics[scale=0.37]{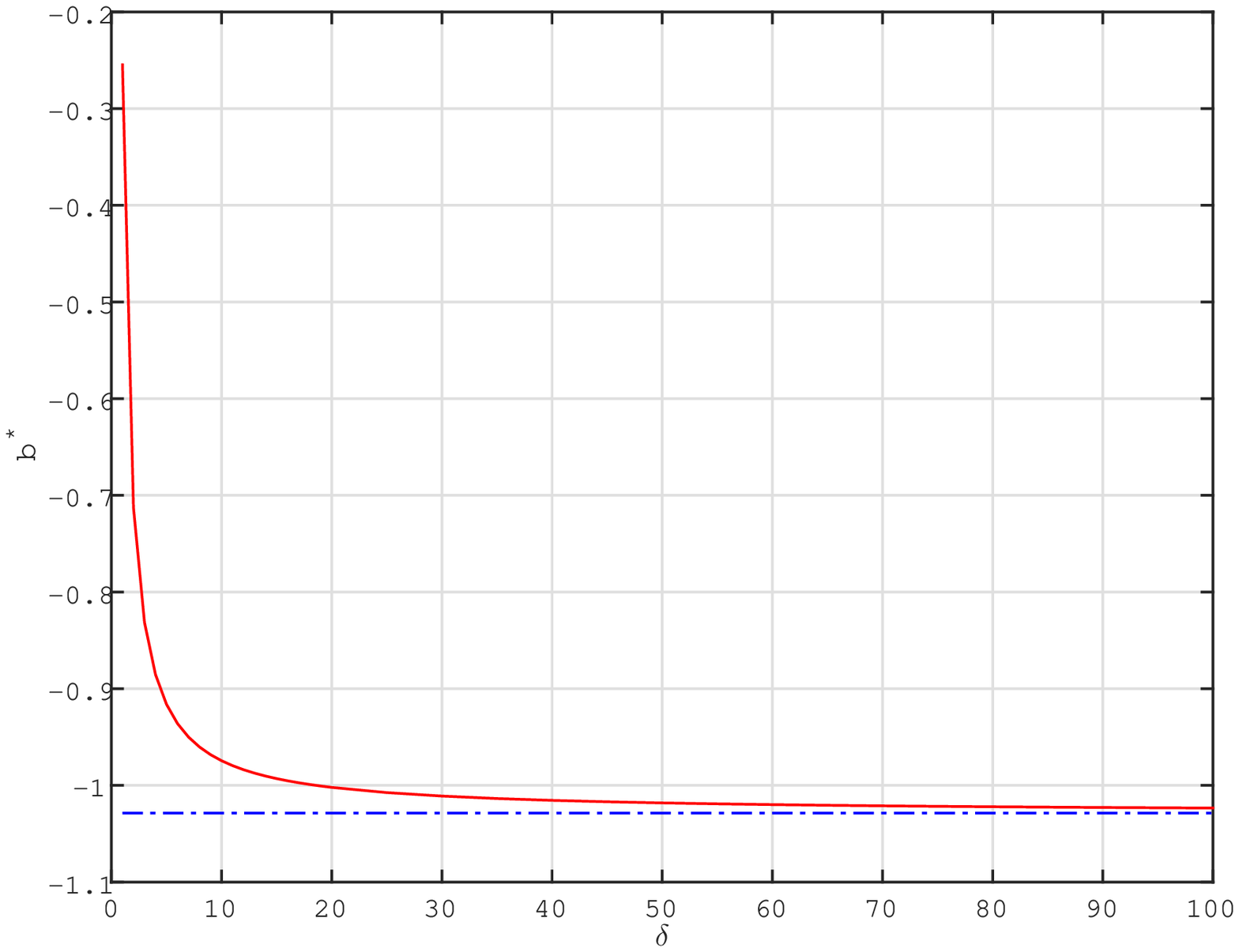} & \includegraphics[scale=0.37]{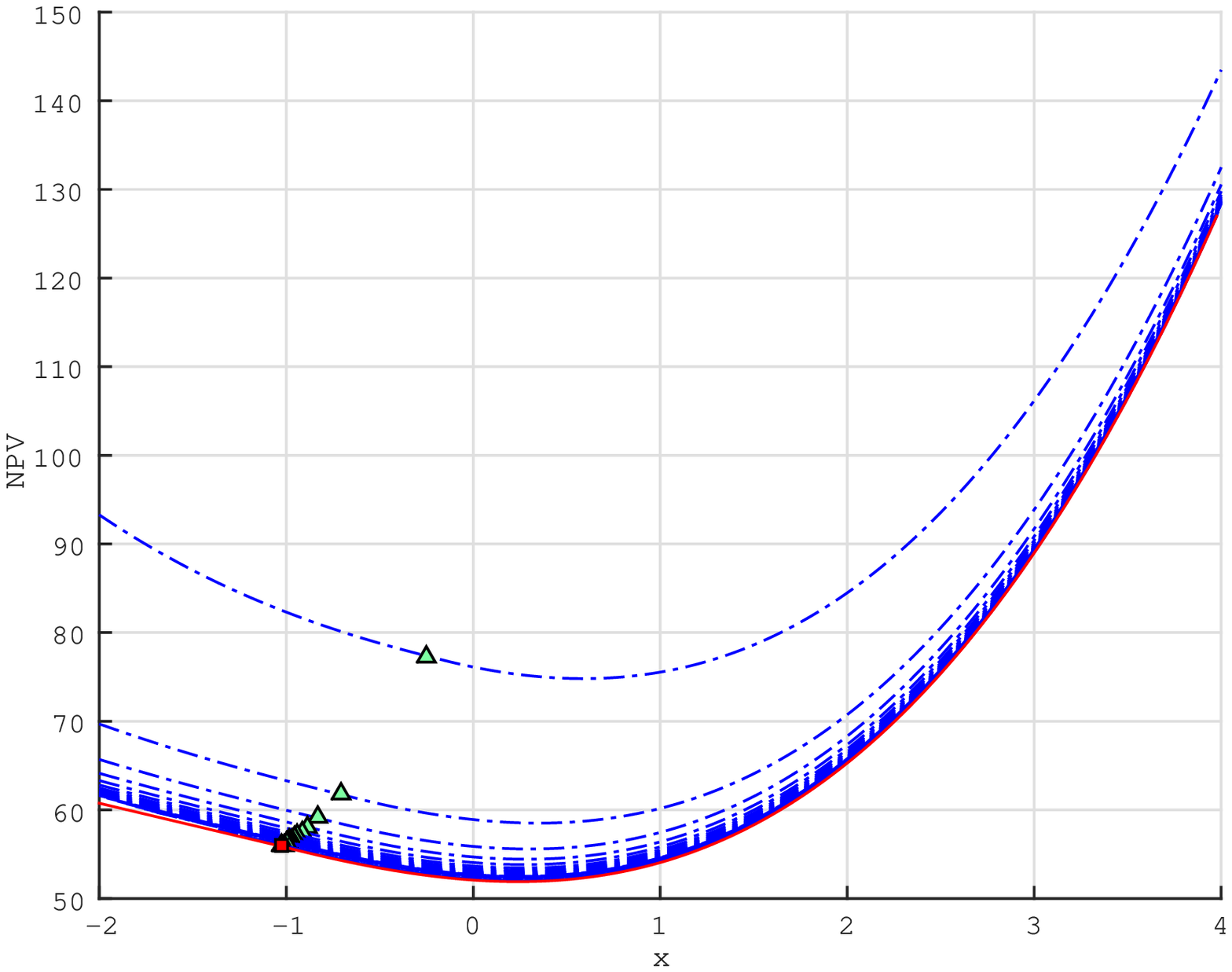} 
 \end{tabular}
\end{minipage}
\caption{\footnotesize (Top) Plots of $v_b(x)$ for the cases $\beta = -5$ (left) and $\beta = 5$ (right). Each panel shows $v_{b^*}(x)$ (solid) in comparison to $v_{b}(x)$ (dotted) for $b \in \{b^*-1, b^*-0.5,b^*+0.5, b^*+1\}$.  The point $(b^*, v_{b^*}(b^*))$ is indicated by the square mark on the solid line. Up-pointing and down-pointing triangles show the points $(b,v_{b}(b))$ for $b > b^*$ and $b < b^*$, respectively. 
(Bottom) Plots of convergence as $\delta \uparrow \infty$. The left panel shows $b^*(\delta)$ for $\delta$ running from $1$ to $100$.  The value of $b^*(\infty)$ is indicated by the dotted line. On the right panel,  the functions
$\tilde{v}(x; \delta)$ as in \eqref{v_tilde_delta} are shown as dotted lines for $\delta \in \{ 1,2,\ldots, 20, 40, 60,80,100\}$.  The solid line shows the function $\tilde{v}(x; \infty)$. The square mark indicates  the point $(b^*(\infty), \tilde{v}(b^*(\infty); \infty))$ while the up-pointing triangles show $(b^*(\delta),\tilde{v}(b^*(\delta); \delta))$.
} \label{figure_caption_optimality}
\end{center}
\end{figure}

%We now study the behavior with respect to the unit cost (or reward) $\beta$.  It is noted that with our choice of $h$, by Lemma \ref{lemma_classification} $b^*$ is always finite for any choice of $\beta \in \R$. Figure \ref{figure_caption_beta}
% shows the value functions for $\beta = -10, -9, \ldots, 9, 10$.  As it is clear analytically, it is monotonically increasing in the value $\beta$.  In addition, $b^*$ is increasing in $\beta$ (which can also be confirmed easily from \eqref{def_I_new}).
%
%\begin{figure}[htbp]
%\begin{center}
%\begin{minipage}{1.0\textwidth}
%\centering
%\begin{tabular}{c}
% \includegraphics[scale=0.5]{figure_beta} \end{tabular}
%\end{minipage}
%\caption{The sensitivity with respect to $\beta$.  The value function $v_{b^*}$ is plotted for $\beta = -10,-9, \ldots, 9, 10$. It is monotonically increasing in $\beta$.  The points $v_{b^*}(b^*)$ are indicated by the square marks.} \label{figure_caption_beta}
%\end{center}
%\end{figure}

We then analyze the convergence as $\delta \uparrow \infty$.  As we have studied in Section \ref{section_convergence}, we fix the process $Y$ with $\tilde{\gamma}_Y = 0.5$ and confirm the convergence $\tilde{v}(x; \delta)$ of \eqref{v_tilde_delta} to $\tilde{v}(x; \infty)$ of  \eqref{v_tilde_infty}. Throughout, we fix $\tilde{\beta} = 5$ ($\beta = -5$).
At the bottom of Figure \ref{figure_caption_optimality}, we first plot in the left panel  the values of $b^*(\delta)$ with respect to $\delta$ along with the value of $b^*(\infty)$.
We see that $b^*(\delta)$ is decreasing in $\delta$ and indeed converges to $b^*(\infty)$ as $\delta$ gets large.  This is consistent with Proposition \ref{convergence_b_star}. In the right panel, we show the value functions  $\tilde{v}(x; \delta)$ as in \eqref{v_tilde_delta} for $\delta \in \{ 1,2,\ldots, 20, 40, 60,80,100\}$ along with the function $\tilde{v}(x; \infty)$.  As has been shown in Theorem \ref{theorem_convergence},   $\tilde{v}(x; \delta)$  is decreasing in $\delta$ and converges to  $\tilde{v}(x; \infty)$.

\appendix

\section{Proof of Lemma \ref{lemma_u_b}} \label{proof_lemma_u_b}
\subsection{Proof of \eqref{v_b_x_derivative_b}}
(i)
%We shall first show here that we can choose sufficiently small $-M < 0 \wedge b$ and $c > 0$ such that
%\begin{align*}
%y \mapsto h(y)\frac {|r_{b}(x+\epsilon, y) - r_b(x, y)| } \epsilon
%\end{align*}
%is bounded in $\epsilon \in (0, c]$ by some  function that is integrable over $(-\infty, -M]$.
%
In order to show that the derivative can go into the integral, we shall show that for sufficiently small $-M$,
\begin{align}
\frac \partial {\partial x}\int_{-\infty}^{-M} h(y) r_{b}^{(1)} (x, y) \diff y = \int_{-\infty}^{-M} h(y) \frac \partial {\partial x} r_{b}^{(1)} (x, y) \diff y. \label{interchange_below_M}
\end{align}

We first see that, for $\epsilon > 0$,
\begin{align*}
r_b^{(1)} (x+ \epsilon, y) =  \mathrm{e}^{\Phi(q) \epsilon} r_b^{(1)} (x, y) - W^{(q)} (x+\epsilon - y) +  \mathrm{e}^{\Phi(q) \epsilon}W^{(q)}(x-y), \quad y < b,
\end{align*}
and hence
\begin{align*}
\frac {r_b^{(1)} (x+ \epsilon, y) - r_b^{(1)} (x, y)} \epsilon = \frac { \mathrm{e}^{\Phi(q) \epsilon} -1} \epsilon r_b^{(1)} (x, y) -  \mathrm{e}^{\Phi(q) (x+\epsilon-y)} \frac { W_{\Phi(q)} (x+\epsilon - y) - W_{\Phi(q)}(x-y)   } \epsilon.
\end{align*}

Here $( \mathrm{e}^{\Phi(q) \epsilon} - 1) /\epsilon$ is bounded in $\epsilon > 0$ on compacts and $v_b^{(1)}(x) = \int_{-\infty}^\infty h(y) r_b^{(1)} (x, y) \diff y$ is well-defined.  Hence, for the proof of \eqref{interchange_below_M}, it is sufficient to show that
\begin{align*}
y \mapsto H(y; x, \epsilon) := \left| h(y)  \mathrm{e}^{-\Phi(q)y}  \frac {W_{\Phi(q)} (x+\epsilon - y) - W_{\Phi(q)}(x-y)} \epsilon \right|
\end{align*}
 is bounded in $\epsilon > 0$ by a function that is integrable over $(-\infty, -M)$.

First we observe that, by Remark \ref{remark_smoothness_zero} (3), $(W_{\Phi(q)} (x+\epsilon - y) - W_{\Phi(q)}(x-y)) / \epsilon \leq  {\sup}_{0 \leq \epsilon' \leq \epsilon}W_{\Phi(q)}' ((x+\epsilon'-y)+)$.
Moreover, for any $y < 0$ and $0 < \epsilon' < \epsilon$,  Remark \ref{remark_smoothness_zero} (3) gives
\begin{align*}
\frac {W_{\Phi(q)}'((x+\epsilon' - y)+)} {W_{\Phi(q)}'((x - y)+)} \leq \frac {W_{\Phi(q)}(x+\epsilon' - y)} {W_{\Phi(q)}(x - y)} \leq \frac {W_{\Phi(q)}(x+\epsilon - y)} {W_{\Phi(q)}(x - y)}.
\end{align*}
Fix any $c,k > 0$. By \eqref{W_q_limit} (and hence $W_{\Phi(q)}(x+ k)/W_{\Phi(q)}(x) \xrightarrow{x \uparrow \infty}1$), we can choose sufficiently small $-M < 0$, such that for any $y \in (-\infty, -M]$, $ {W_{\Phi(q)}(x+k - y)} / {W_{\Phi(q)}(x - y)} \leq 1+ c$.
%Hence we have a bound: for $0 < \epsilon < k$, $W'_{\Phi(q)} (x+\epsilon' - y) \leq (1+c) W'_{\Phi(q)} (x- y)$.
%\begin{multline*}
%\epsilon^{-1} |r_{b+\epsilon} (x, y) - r_{b}(x,y) | \\
%\leq (1+c) e^{\Phi(q)x}\Big[ e^{-\Phi(q)y}  W_{\Phi(q)}' (b-y)  + e^{-\Phi(q)y}  \varphi(q)  \int_0^\infty e^{-(\varphi(q) -\Phi(q)) z} W_{\Phi(q)}' (z+b-y) \diff z \Big]. \end{multline*}
%Hence, we are left to show that $y \mapsto h(y) e^{-\Phi(q)y}  W_{\Phi(q)}' (b-y) $ and $y \mapsto e^{-\Phi(q)y}  h(y) \int_0^\infty e^{-(\varphi(q) -\Phi(q)) z} W_{\Phi(q)}' (z+b-y) \diff z$ are integrable over $(-\infty, -M)$.
Collecting these inequalities, we have a bound $H(y; x, \epsilon) \leq (1+c)  \mathrm{e}^{-\Phi(q) y}W'_{\Phi(q)} ((x- y)+) |h(y)|$ for  $0 < \epsilon < k$ and $y < -M$. Because \eqref{eq_integrability1} is finite, we have the claim. Hence \eqref{interchange_below_M} holds.

(ii) By (i), the derivative can be interchanged over the integral.  Hence,
%\begin{align*}
%v_b^{(1)}(x)
%&=\int_b^\infty (h(y)+ \beta \delta)  r_b(x, y) \diff y + \int_{-\infty}^b h(y) r_b(x,y) \diff y \\
%&:=   e^{\Phi(q) (x  - b)} \frac {\varphi(q)-\Phi(q)} {\delta \Phi(q)} \Big[  \int_0^\infty h(y+b)  e^{- \varphi(q) y} \diff y +   \frac {\beta \delta}  {\varphi(q)} \Big]  \\ &+ \int_{(x-b) \vee 0}^0 h(y+b) \Big[ e^{\Phi(q)(x-b)} \frac {\varphi(q)-\Phi(q)} {\Phi(q)} \int_0^\infty e^{-\varphi(q)z} W^{(q)\prime} (z-y) \diff z - W^{(q)}(x-b-y) \Big] \diff y \\
%&+ \int_{-\infty}^{(x-b) \vee 0} h(y+b) \Big[ e^{\Phi(q)(x-b)} \frac {\varphi(q)-\Phi(q)} {\Phi(q)} \int_0^\infty e^{-\varphi(q)z} W^{(q)\prime} (z-y) \diff z - W^{(q)}(x-b-y) \Big] \diff y
%\end{align*}
\begin{align} \label{v_b_1_after_derivative}
\begin{split}
v_b^{(1)\prime}&(x)
%&=\int_b^\infty (h(y)+ \beta \delta)  r_b(x, y) \diff y + \int_{-\infty}^b h(y) r_b(x,y) \diff y \\
=   -\mathbf{1}_{\{x < b\}} h(x) W^{(q)}(0) +  \mathrm{e}^{\Phi(q) (x  - b)} \frac {\varphi(q)-\Phi(q)} {\delta} \Big[  \int_0^\infty h(y+b)   \mathrm{e}^{- \varphi(q) y} \diff y +   \frac {\beta \delta}  {\varphi(q)} \Big]  \\ &+ \int_{-\infty}^0 h(y+b) \Big[  \mathrm{e}^{\Phi(q)(x-b)} (\varphi(q)-\Phi(q))  \int_0^\infty  \mathrm{e}^{-\varphi(q)z} W^{(q)\prime} (z-y) \diff z - W^{(q)\prime}(x-b-y) \Big] \diff y,
\end{split}
\end{align}
where the first term appears due to the discontinuity of the scale function at zero as in Remark \ref{remark_smoothness_zero} (2) for the case of bounded variation.

In order to apply the integration by parts, we first confirm that
\begin{align} \label{product_h_vanish_in_limits}
h(y+b) \Big[  \mathrm{e}^{\Phi(q)(x-b)} (\varphi(q)-\Phi(q)) \int_0^\infty  \mathrm{e}^{-\varphi(q)z} W^{(q)} (z-y) \diff z - W^{(q)}(x-b-y) \Big] \xrightarrow{y \downarrow -\infty} 0.
\end{align}
Indeed, %by \eqref{v_b_expression} and because $v_{b}^{(1)}(x)$ is finite 
by \eqref{h_r_abs_integrable1}, $\lim_{y \downarrow -\infty}h(y+b) r_b^{(1)}(x, y+b) = 0$.
In addition,
%\begin{align*}
%&r_b^{(1)}(x, y+b)- \Big[  \mathrm{e}^{\Phi(q)(x-b)} (\varphi(q)-\Phi(q)) \int_0^\infty  \mathrm{e}^{-\varphi(q)z} W^{(q)} (z-y) \diff z - W^{(q)}(x-b-y) \Big] \\
%&=  \mathrm{e}^{\Phi(q)(x-b)} \frac {\varphi(q)-\Phi(q)} {\Phi(q)} \int_0^\infty  \mathrm{e}^{-\varphi(q)z} \Theta^{(q)} (z-y) \diff z \\ &=   \mathrm{e}^{\Phi(q)(x-b)} \frac {\varphi(q)-\Phi(q)} {q} \E [ \mathrm{e}^{\varphi(q) \underline{X}_{\EQ}} \mathbf{1}_{\{ \underline{X}_{\EQ} < y \}}] \in [0,  \mathrm{e}^{\Phi(q)(x-b)} \frac {\varphi(q)-\Phi(q)} {q}  \mathrm{e}^{\varphi(q) y}].\end{align*}
\begin{align*}
&r_b^{(1)}(x, y+b)- \Big[  \mathrm{e}^{\Phi(q)(x-b)} (\varphi(q)-\Phi(q)) \int_0^\infty  \mathrm{e}^{-\varphi(q)z} W^{(q)} (z-y) \diff z - W^{(q)}(x-b-y) \Big] \\
&=  \mathrm{e}^{\Phi(q)(x-b)} \frac {\varphi(q)-\Phi(q)} {\Phi(q)} \int_0^\infty  \mathrm{e}^{-\varphi(q)z} \Theta^{(q)} (z-y) \diff z =   \mathrm{e}^{\Phi(q)(x-b)} \frac {\varphi(q)-\Phi(q)} {q} e^{-\varphi(q)y}\E [ \mathrm{e}^{\varphi(q) \underline{X}_{\EQ}} \mathbf{1}_{\{ \underline{X}_{\EQ} < y \}}]\\
&\in \Big[0,  \mathrm{e}^{\Phi(q)(x-b)} \frac {\varphi(q)-\Phi(q)} {q} \p \{\underline{X}_{\EQ}<y \}\Big].
\end{align*}
Using identity \eqref{finiteness_laplace_negagtive} %(the existence of an exponential moments for $\underline{X}_{\EQ}$)  
together with Assumption \ref{assump_h}, 
\begin{align}
\p \{\underline{X}_{\EQ}<y \}h(y+b) \xrightarrow{y \downarrow -\infty} 0. \label{prob_exp_inf_vanish}
\end{align} 
% Hence, 
%the product of this and $h(y+b)$ vanishes as $y \downarrow \infty$.
Hence, \eqref{product_h_vanish_in_limits} holds.

Using \eqref{scale_function_laplace}, \eqref{def_varphi} and \eqref{product_h_vanish_in_limits}, together with the following identity
\begin{align}
\int_0^\infty  \mathrm{e}^{-\varphi(q) z} W^{(q)} (z) \diff z  = (\psi(\varphi(q))-q)^{-1}= (\varphi (q) \delta)^{-1}, \label{laplace_special_case}
\end{align}
integration by parts yields
\begin{align*}
v_b^{(1) \prime}(x)
%&=\int_b^\infty (h(y)+ \beta \delta)  r_b(x, y) \diff y + \int_{-\infty}^b h(y) r_b(x,y) \diff y \\
&=    \mathrm{e}^{\Phi(q) (x  - b)} \frac {\varphi(q)-\Phi(q)} {\delta \varphi(q)} \Big[   \int_0^\infty  \mathrm{e}^{-\varphi(q) y} h'(y+b) \diff y +   {\beta \delta}   \Big]  + h(b)   W^{(q)}(x-b) \\ &+\int_{-\infty}^0 h'(y+b)\Big[  \mathrm{e}^{\Phi(q)(x-b)} (\varphi(q)-\Phi(q)) \int_0^\infty  \mathrm{e}^{-\varphi(q)z} W^{(q)} (z-y) \diff z - W^{(q)}(x-b-y) \Big] \diff y.
\end{align*}
Taking the difference between this and  $\int_{-\infty}^\infty h'(y) r_{b}^{(1)}(x,y) \diff y$, we have the result.

\subsection{Proof of \eqref{v_b_x_derivative_b_2}}
In view of \eqref{v_J_diff_b_star}, let us decompose $v_{b}^{(2)}(x) = v_{b}^{(2,1)}(x) + \delta v_{b}^{(2,2)}(x)$ where
\begin{align*}
 \begin{split}
v_{b}^{(2,1)}(x)  &:=\int_{0}^\infty (h(y+b) + \beta \delta) \Big\{   \mathrm{e}^{-\varphi(q)y}M(x; b) - \mathbb{W}^{(q)}(x-b-y) \Big\} \diff y, \\
v_{b}^{(2,2)}(x)  &:= \int_{-\infty}^{0} h(y+b) \Big\{ M(x; b) \int_{0}^\infty  \mathrm{e}^{-\varphi(q) z} W^{(q)\prime}(z-y) \diff z - \int_{b}^x \mathbb{W}^{(q)}(x-z) W^{(q)\prime} (z-b-y) \diff z \Big\} \diff y.
\end{split}
\end{align*}

%\begin{align*}
%&v_{b^*} (x) - J(x; b^*) \\ &=\int_{b^*}^\infty (h(y) + \red{\beta}\delta) \Big\{  e^{-\varphi(q)(y-b)}M(x; b^*) - \mathbb{W}^{(q)}(x-y) \Big\} \diff y \\
%&+ \int_{-\infty}^{b^*} h(y) \Big\{ \delta M(x; b^*) e^{\varphi(q) b^*} \int_{b^*}^\infty e^{-\varphi(q) z} W^{(q)\prime}(z-y) \diff z - \delta \int_{b^*}^x \mathbb{W}^{(q)}(x-z) W^{(q)\prime} (z-y) \diff z \Big\} \diff y.
%\end{align*}
For the function $v_{b}^{(2,1)}$, integration by parts gives
\begin{multline*}
%&\int_{0}^\infty (h(y+b^*) +  \red{\beta} \delta) \Big\{  e^{-\varphi(q)y}M(x; b^*) - \mathbb{W}^{(q)}(x-b^*-y) \Big\} \diff y \\
%&= \Big[ (h(y+b^*)+  \red{\beta} \delta) \Big\{- \frac 1 {\varphi(q)} e^{-\varphi(q)y}M(x; b^*) + \overline{\mathbb{W}}^{(q)}(x-b^*-y) \Big\}\Big]_{y=0}^{y = \infty} \\ &+ \int_0^\infty h'(y+b^*)  \Big\{ \frac 1 {\varphi(q)}  e^{-\varphi(q)y}M(x; b^*) - \overline{\mathbb{W}}^{(q)}(x-b^*-y) \Big\} \diff y \\
v_{b}^{(2,1)}(x) =  (h(b)+  \beta \delta) \Big\{ \frac {M(x; b)} {\varphi(q)} - \int_0^{x-b} \mathbb{W}^{(q)}(z) \diff z \Big\} + \int_{0}^\infty h'(y+b)  \Big\{ \frac {M(x; b)} {\varphi(q)}   \mathrm{e}^{-\varphi(q)y} - \int_0^{x-b-y}\mathbb{W}^{(q)}(z) \diff z \Big\} \diff y.
\end{multline*}
Differentiating this and noting that
\begin{align}
M'(x;b)  &= (\varphi(q)-\Phi(q)) \mathbb{W}^{(q)} (x-b)  + \Phi(q) M(x; b),  \label{M_prime_relation}
\end{align}
we have
\begin{multline} \label{v_prime_21}
v_{b}^{(2,1) \prime}(x) =  \frac {\Phi(q)} {\varphi(q)}(h(b)+  \beta \delta) \Big\{ {-\mathbb{W}^{(q)} (x-b)  +  M(x; b)}  \Big\} \\ + \int_0^\infty h'(y+b)  \Big\{ \frac { (\varphi(q)-\Phi(q)) \mathbb{W}^{(q)} (x-b)  + \Phi(q) M(x; b)} {\varphi(q)}   \mathrm{e}^{-\varphi(q)y}- \mathbb{W}^{(q)}(x-b-y) \Big\} \diff y.
%    &= - (h(b^*)+  \beta \delta)  \frac { \Phi(q) } {\varphi(q)}  \Big[\mathbb{W}^{(q)} (x-b^*)  -  {M(x; b^*)} \Big]   + \int_0^\infty h'(y+b^*)  \Big\{ \frac {M'(x; b^*) } {\varphi(q)}  e^{-\varphi(q)y}- \mathbb{W}^{(q)}(x-b^*-y) \Big\} \diff y.
%&= - (h(b^*)+  \red{\beta} \delta)  \frac { \Phi(q)} {\varphi(q)}  \Big[ \mathbb{W}^{(q)} (x-b^*)  -  {M(x; b^*)} \Big]      \\
%&+ \Big[ h'(y) \Big\{   - \frac {e^{-\varphi(q)(y-b)}} {\varphi(q)^2} \Big[ (\varphi(q)-\Phi(q) )\mathbb{W}^{(q)} (x-b^*)  +   \Phi(q)  M(x; b^*)  \Big] + \overline{\mathbb{W}}^{(q)}(x-y) \Big\} \Big]_{y=b^*}^\infty \\
%&- \int_{b^*}^\infty h''(y)   \Big\{   - \frac {e^{-\varphi(q)(y-b)}} {\varphi(q)^2} \Big[ (\varphi(q)-\Phi(q) )\mathbb{W}^{(q)} (x-b^*)  +   \Phi(q) M(x; b^*)  \Big] + \overline{\mathbb{W}}^{(q)}(x-y) \Big\} \diff y \\
%&=- (h(b^*)+  \red{\beta} \delta)  \frac { \Phi(q)} {\varphi(q)}  \Big[ \mathbb{W}^{(q)} (x-b^*)  -  {M(x; b^*)} \Big]   \\
%&- \Big[ h'(b^*) \Big\{   - \frac 1 {\varphi(q)^2} \Big[ (\varphi(q)-\Phi(q) )\mathbb{W}^{(q)} (x-b^*)  +   \Phi(q)  M(x; b^*)  \Big] + \overline{\mathbb{W}}^{(q)}(x-b^*) \Big\} \Big] \\
%&+ \int_{b^*}^\infty h''(y)   \Big\{  \frac {e^{-\varphi(q)(y-b)}} {\varphi(q)^2} \Big[ (\varphi(q)-\Phi(q) )\mathbb{W}^{(q)} (x-b^*)  +   \Phi(q) M(x; b^*)  \Big] - \overline{\mathbb{W}}^{(q)}(x-y) \Big\} \diff y.
\end{multline}
%Here
%\begin{align*}
%&\frac {e^{-\varphi(q)(y-b)}} {\varphi(q)^2} \Big[ (\varphi(q)-\Phi(q) )\mathbb{W}^{(q)} (x-b^*)  +  M(x; b^*)  \Big] - \overline{\mathbb{W}}^{(q)}(x-y) \\
%&=    \frac {e^{-\varphi(q)(y-b)}} {\varphi(q)^2} \Big[ (\varphi(q)-\Phi(q) )\mathbb{W}^{(q)} (x-b^*)  \\ &+  (\varphi(q)-\Phi(q)) \Big[ \overline{\mathbb{W}}^{(q)} (x-b^*)  +  e^{-\Phi(q) b^*} \int_{b^*}^x \Phi(q) e^{\Phi(q) z} \overline{\mathbb{W}}^{(q)}(x-z) \diff z \Big]\Big] - \overline{\mathbb{W}}^{(q)}(x-y)
% \end{align*}
% Its derivative with respect to $y$ equals
%\begin{align*}
% -\frac {e^{-\varphi(q)(y-b)}} {\varphi(q)} \Big[ (\varphi(q)-\Phi(q) )\mathbb{W}^{(q)} (x-b^*)  \\ +  (\varphi(q)-\Phi(q)) \Big[ \overline{\mathbb{W}}^{(q)} (x-b^*)  +  e^{-\Phi(q) b^*} \int_{b^*}^x \Phi(q) e^{\Phi(q) z} \overline{\mathbb{W}}^{(q)}(x-z) \diff z \Big]\Big] + \mathbb{W}^{(q)}(x-y)
% \end{align*}

In order to deal with the function $v_{b}^{(2,2)}$, we shall first show the limit
 \begin{align}
\lim_{y \downarrow -\infty} h(y+b) N(y; x, b) = 0,  \label{limit_N_y}\end{align}
where we define
\begin{align*}
N(y; x, b) :=  M(x; b)  \int_{0}^\infty  \mathrm{e}^{-\varphi(q) z} W^{(q)}(z-y) \diff z - \int_{b}^x \mathbb{W}^{(q)}(x-z) W^{(q)} (z-b-y) \diff z,
\end{align*}
so as to apply the integration by parts.

First, because $v^{(2)}_b$ is well defined and finite by \eqref{h_r_abs_integrable2}, $\lim_{y \downarrow -\infty} h(y+b) r_b^{(2)}(x, y+b)  = 0$.
In addition, \begin{multline*}
\delta^{-1} r_b^{(2)}(x, y+b) -  \Phi(q) N (y; x, b)
= M(x; b) \int_{0}^\infty  \mathrm{e}^{-\varphi(q) z} \Theta^{(q)}(z-y) \diff z -  \int_{b}^x \mathbb{W}^{(q)}(x-z) \Theta^{(q)} (z-b-y) \diff z,
%= \Big\{ M(x; b^*) \int_{0}^\infty  \mathrm{e}^{-\varphi(q) z} \Theta^{(q)}(z-y) \diff z -  \int_{b^*}^x \mathbb{W}^{(q)}(x-z) \Theta^{(q)} (z-b^*-y) \diff z \Big\}
\end{multline*}
whose absolute value is bounded by
%\begin{align*}
%&M(x; b) \int_{0}^\infty  \mathrm{e}^{-\varphi(q) z} \Theta^{(q)}(z-y) \diff z + \int_{b}^x \mathbb{W}^{(q)}(x-z) \Theta^{(q)} (z-b-y) \diff z \\
%&\leq M(x; b) \int_{0}^\infty  \mathrm{e}^{-\varphi(q) z} \Theta^{(q)}(z-y) \diff z+  \mathbb{W}_{\varphi(q)}(x-b) \int_{0}^\infty  \mathrm{e}^{\varphi(q) (x-z-b)} \Theta^{(q)} (z-y) \diff z  \\
%&= M(x; b) \E [ \mathrm{e}^{\varphi(q) \underline{X}_{\EQ}} \mathbf{1}_{\{ \underline{X}_{\EQ} < y \}}]+ \mathbb{W}_{\varphi(q)}(x-b)  \mathrm{e}^{-\varphi(q)(x-b)}\E [ \mathrm{e}^{\varphi(q) \underline{X}_{\EQ}} \mathbf{1}_{\{ \underline{X}_{\EQ} < y \}}] \\
%&\leq  \mathrm{e}^{\varphi(q) y}  [M(x; b) +  \mathbb{W}_{\varphi(q)}(x-b)  \mathrm{e}^{-\varphi(q)(x-b)} ], \end{align*}
\begin{align*}
&M(x; b) \int_{0}^\infty  \mathrm{e}^{-\varphi(q) z} \Theta^{(q)}(z-y) \diff z + \int_{b}^x \mathbb{W}^{(q)}(x-z) \Theta^{(q)} (z-b-y) \diff z \\
&\leq M(x; b) \int_{0}^\infty  \mathrm{e}^{-\varphi(q) z} \Theta^{(q)}(z-y) \diff z+  \mathbb{W}_{\varphi(q)}(x-b) \int_{0}^\infty  \mathrm{e}^{\varphi(q) (x-z-b)} \Theta^{(q)} (z-y) \diff z  \\
&= \Big(M(x; b) e^{-\varphi(q)y}\E [ \mathrm{e}^{\varphi(q) \underline{X}_{\EQ}} \mathbf{1}_{\{ \underline{X}_{\EQ} < y \}}]+ \mathbb{W}_{\varphi(q)}(x-b)  \mathrm{e}^{-\varphi(q)(x-b)}e^{-\varphi(q)y}\E [ \mathrm{e}^{\varphi(q) \underline{X}_{\EQ}} \mathbf{1}_{\{ \underline{X}_{\EQ} < y \}}] \Big) \frac {\Phi(q)} q \\
&\leq    [M(x; b) +  \mathbb{W}_{\varphi(q)}(x-b)  \mathrm{e}^{-\varphi(q)(x-b)} ]\p \{ \underline{X}_{\EQ}<y \} \frac {\Phi(q)} q, \end{align*}
where  $\mathbb{W}_{\varphi(q)}(z) :=  \mathrm{e}^{-\varphi(q) z} \mathbb{W}^{(q)}(z)$.
This together with \eqref{prob_exp_inf_vanish} shows the  limit \eqref{limit_N_y}.
%The same argument we did for \eqref{product_h_vanish_in_limits} implies $ \p \{\underline{X}_{\EQ}<y \}h(y+b) \xrightarrow{y \downarrow -\infty} 0$ and hence shows the limit \eqref{limit_N_y}.

Now, integration by parts followed by Fubini's theorem, together with the identities \eqref{laplace_special_case} and $\int_{b}^x\mathbb{W}^{(q)} (x-z) W^{(q)} (z-b)\diff z
% &= \int_0^{x-b}\mathbb{W}^{(q)} (x-b-z) W^{(q)} (z)\diff z \\ 
= \delta^{-1} (\int_0^{x-b}\mathbb{W}^{(q)} (z) \diff z -  \int_0^{x-b}{W}^{(q)} (z) \diff z )$,
gives \begin{align*}
v_{b}^{(2,2)}(x)
%&= \Big[ h(y) \Big\{ - M(x; b^*) e^{\varphi(q) b^*} \int_{b^*}^\infty e^{-\varphi(q) z} W^{(q)}(z-y) \diff z + \int_{b^*}^x \mathbb{W}^{(q)}(x-z) W^{(q)} (z-y) \diff z \Big\} \Big]_{y = -\infty}^{y = b^*} \\
%&+ \int_{-\infty}^{b^*} h'(y) \Big\{ M(x; b^*) e^{\varphi(q) b^*} \int_{b^*}^\infty e^{-\varphi(q) z} W^{(q)}(z-y) \diff z -  \int_{b^*}^x \mathbb{W}^{(q)}(x-z) W^{(q)} (z-y) \diff z \Big\} \diff y  \\
&=  \frac {h(b)} \delta \Big\{ - \frac {M(x; b)} {\varphi (q)} +  \int_0^{x-b}\mathbb{W}^{(q)} (z) \diff z -  \int_0^{x-b}{W}^{(q)} (z) \diff z\Big\}  + \int_{b}^x  \mathbb{W}^{(q)}(x-z)  A(z; b) \diff z.
%&+ \int_{-\infty}^{0} h'(y+b) \Big\{  M(x; b)  \int_{0}^\infty e^{-\varphi(q) z} W^{(q)}(z-y) \diff z -  \int_{b}^x \mathbb{W}^{(q)}(x-z) W^{(q)} (z-b-y) \diff z \Big\} \diff y.
\end{align*}
Here we define
\begin{align*}
A(z; b) := \int_{-\infty}^0 h'(y+b) \Big\{ (\varphi(q)-\Phi(q))   \mathrm{e}^{-\Phi(q) (b-z)} \int_{0}^\infty  \mathrm{e}^{-\varphi(q) w} W^{(q)}(w-y) \diff w - W^{(q)} (z-y-b)  \Big\} \diff y,
\end{align*}
which is finite for all $z \in [b, x]$. Indeed, we have $A(z; b)= e^{-\Phi(q) (z-x)} A(x;b) + \int_{-\infty}^0 h'(y+b) (e^{- \Phi(q) (z-x)}W^{(q)} (x-y-b)- W^{(q)} (z-y-b)  ) \diff y$ where the integral part is finite by similar arguments as in Remark \ref{remark_finiteness_v_1_2}; if it is infinity, it contradicts with that $v_b^{(2,2)}$ is finite.

Hence, taking a derivative and by \eqref{M_prime_relation},
\begin{align*}
v_{b}^{(2,2) \prime}(x) &=   \frac {h(b)} \delta \Big\{ - \frac {M'(x; b)} {\varphi (q)} +   \mathbb{W}^{(q)} (x-b) -  W^{(q)} (x-b) \Big\}  + \int_{-\infty}^0 h'(y+b) B(y, x,b) \diff y \\
&=    \frac {h(b)} \delta \Big\{  \frac {\Phi(q)} {\varphi(q)} \Big( \mathbb{W}^{(q)} (x-b) -  M(x; b) \Big) -  W^{(q)} (x-b) \Big\}  + \int_{-\infty}^0 h'(y+b) B(y, x,b) \diff y
\end{align*}
with
\begin{align*}
B(y, x,b)& := \mathbb{W}^{(q)}(0) \Big\{ (\varphi(q)-\Phi(q))   \mathrm{e}^{-\Phi(q) (b-x)} \int_{0}^\infty  \mathrm{e}^{-\varphi(q) w} W^{(q)}(w-y) \diff w - W^{(q)} (x-y-b)  \Big\}  \\
&+ \int_{b}^x  \mathbb{W}^{(q)\prime}(x-z)   \Big\{ (\varphi(q)-\Phi(q))   \mathrm{e}^{-\Phi(q) (b-z)} \int_{0}^\infty  \mathrm{e}^{-\varphi(q) w} W^{(q)}(w-y) \diff w - W^{(q)} (z-y-b)  \Big\} \diff z.
\end{align*}
By \eqref{some_result_int_by_parts} and  integration by parts, we can write
%we obtain
%$$\int_0^\infty  \mathrm{e}^{-\varphi(q)z} W^{(q)\prime} (z-y) \diff z + W^{(q)}(-y) )/ \varphi(q)= \int_0^\infty  \mathrm{e}^{-\varphi(q) z} W^{(q)} (z-y) \diff z,$$
%and therefore
\begin{align*}
B(y, x,b)
%&= \mathbb{W}^{(q)}(x-b) \Big\{ (\varphi(q)-\Phi(q))   \int_{0}^\infty e^{-\varphi(q) w} W^{(q)}(w-y) \diff w - W^{(q)} (-y)  \Big\}  \\
%&+ \int_{b}^x  \mathbb{W}^{(q)}(x-z)   \Big\{ \Phi(q) (\varphi(q)-\Phi(q))  e^{-\Phi(q) (b-z)} \int_{0}^\infty e^{-\varphi(q) w} W^{(q)}(w-y) \diff w - W^{(q)\prime} (z-y-b)  \Big\} \diff z \\
%&= \mathbb{W}^{(q)}(x-b) \Big\{ \frac {-\Phi(q)} {\varphi(q)}   W^{(q)} (-y)  + \frac {\varphi(q)-\Phi(q)} {\varphi(q)} \int_0^\infty e^{-\varphi(q) z} W^{(q)\prime}(z-y) \diff z \Big\}  \\
%&- M(x; b) \frac {\Phi(q)} {\varphi(q)}  [ W^{(q)} (-y)  + \int_0^\infty e^{-\varphi(q) z} W^{(q)\prime}(z-y) \diff z ]) + \int_{b}^x  \mathbb{W}^{(q)}(x-z)W^{(q)\prime} (z-y-b)  \diff z
&= \frac {(\varphi(q)-\Phi(q)) \mathbb{W}^{(q)} (x-b)  + \Phi(q) M(x; b)} {\varphi(q)}\Big( \int_0^\infty  \mathrm{e}^{-\varphi(q)z} W^{(q)\prime} (z-y) \diff z + W^{(q)}(-y)  \Big) \\ &- \mathbb{W}^{(q)} (x-b) W^{(q)}(-y) - \int_{b}^x \mathbb{W}^{(q)}(x-z) W^{(q)\prime} (z-y-b) \diff z.
\end{align*}

From this and \eqref{v_prime_21}, the expression
$\int_{-\infty}^\infty h'(y) r_{b}^{(2)}(x,y) \diff y - (v_{b}^{(2,1)\prime}(x)+\delta v_{b}^{(2,2) \prime}(x))
$
gives the result after some calculations.

\section{Proof of Lemma \ref{lemma_sufficiently_smooth}} \label{proof_lemma_sufficiently_smooth} In view of the expression \eqref{v_prime_equals_expectation}, $v_{b^*}'$ is continuous for all $x \in \R$. Hence, it remains to show that $v_{b^*}'$ is continuously differentiable for the case of unbounded variation.

By \eqref{resolvent_measure} and \eqref{v_prime_equals_expectation}, we can write $v_{b^*}'(x) = w^{(1)}(x) + w^{(2)}(x) \mathbf{1}_{\{x > b^*\}}$ where
\begin{align*}
w^{(1)}(x)
&:=  \frac {\varphi(q)-\Phi(q)} {\delta \Phi(q)}\int_{0}^\infty h'(y+b^*)   \mathrm{e}^{\Phi(q) (x-b^*) - \varphi(q) y } \diff y \\ &+ \int_{-\infty}^{0} h'(y+b^*) \Big[  \mathrm{e}^{\Phi(q)(x-b^*)} \frac {\varphi(q)-\Phi(q)} {\Phi(q)} \int_0^\infty  \mathrm{e}^{-\varphi(q)z} W^{(q)\prime} (z-y) \diff z - W^{(q)}(x-b^*-y) \Big] \diff y, \\
w^{(2)}(x) &:=\int_0^\infty h'(y+b^*)  \Big\{   \mathrm{e}^{-\varphi(q)y}M(x; b^*) - \mathbb{W}^{(q)}(x-b^*-y) \Big\} \diff y \\
&+ \delta \int_{-\infty}^{0} h'(y+b^*) \Big\{  M(x; b^*) \int_{0}^\infty  \mathrm{e}^{-\varphi(q) z} W^{(q)\prime}(z-y) \diff z -  \int_{b^*}^x \mathbb{W}^{(q)}(x-z) W^{(q)\prime} (z-b^*-y) \diff z \Big\} \diff y.
\end{align*}

Notice that \eqref{interchange_below_M} holds when $h$ is replaced with $h'$ (by Assumption  \ref{assump_h}). Hence, similarly to \eqref{v_b_1_after_derivative} (noting that $W^{(q)}(0) = 0$ for the case of unbounded variation),
\begin{align*} w^{(1)\prime}(x)
&=  \frac {\varphi(q)-\Phi(q)} {\delta}\int_{0}^\infty h'(y+b^*)   \mathrm{e}^{\Phi(q) (x-b^*) - \varphi(q) y} \diff y \\ &+ \int_{-\infty}^{0} h'(y+b^*) \Big[  \mathrm{e}^{\Phi(q)(x-b^*)} (\varphi(q)-\Phi(q))  \int_0^\infty  \mathrm{e}^{-\varphi(q)z} W^{(q)\prime} (z-y) \diff z - W^{(q)\prime}(x-b^*-y) \Big] \diff y.
\end{align*}
In addition, with the help of Fubini's theorem and noticing that $W^{(q)}(0) = 0$,\begin{align*}
w^{(2)\prime}(x) &=\int_{b^*}^\infty h'(y) \Big\{   \mathrm{e}^{-\varphi(q)(y-b^*)}M'(x; b^*) - \mathbb{W}^{(q)\prime}(x-y) \Big\} \diff y \\
&+ \delta \int_{b^*}^x \mathbb{W}^{(q)\prime}(x-z) \int_{-\infty}^{0} h'(y+b^*) \Big\{(\varphi(q)-\Phi(q))   \mathrm{e}^{-\Phi(q) (b^*-z)}  \int_{0}^\infty  \mathrm{e}^{-\varphi(q) w} W^{(q)\prime}(w-y) \diff w \\ &\qquad -  W^{(q)\prime} (z-b^*-y)  \Big\} \diff y \diff z.
\end{align*}
The continuity of $w^{(1)\prime}$ and $w^{(2)\prime}$ are clear for $x \neq b^*$. We have
\begin{align*}
\int_{-\infty}^{0} h'(y+b^*) \Big\{(\varphi(q)-\Phi(q))   \mathrm{e}^{-\Phi(q) (b^*-z)}  \int_{0}^\infty  \mathrm{e}^{-\varphi(q) w} W^{(q)\prime}(w-y) \diff w -  W^{(q)\prime} (z-b^*-y)  \Big\} \diff y \\
= \int_{-\infty}^{0} h'(y+b^*) \Big\{\Phi(q) r_b^{(1)}(z-b^*, y) -  \Theta^{(q)} (z-b^*-y)  \Big\} \diff y.
\end{align*}
By \eqref{def_Theta} and Remark \ref{remark_finiteness_v_1_2}, this is finite for $z \in [b^*, x]$.  Hence, $\lim_{x \downarrow b^*}w^{(2)\prime}(x) =0$.
\bibliographystyle{abbrv}
\bibliographystyle{apalike}

\bibliographystyle{agsm}

\bibliography{optimal_refracted}

	\end{document}